\documentclass[oneside]{article}
\usepackage{amsfonts}
\usepackage[a4paper, total={6in, 8in}]{geometry}
\usepackage{bm}
\usepackage{amsmath}
\usepackage{hyperref}
\usepackage{multicol}
\usepackage{amsfonts}
\usepackage{graphicx}
\usepackage{amsthm}
\usepackage{cleveref}
\usepackage{subcaption}
\usepackage{authblk}
\usepackage{xcolor}
\usepackage{soul}

\newcommand{\e}{\varepsilon}
\newcommand{\rto}{\overset{R}{\to}}
\newcommand{\mf}{\mathcal{F}}

\theoremstyle{definition}
\newtheorem{theorem}{Theorem}[section]
\newtheorem{remark}{Remark}[section]
\newtheorem{proposition}{Proposition}[section]
\newtheorem{definition}{Definition}[section]
\newtheorem{assumption}{Assumption}
\newtheorem{lemma}{Lemma}[section]
\newtheorem{corollary}{Corollary}[section]
\newtheorem{example}{Example}[section]

\begin{document}\sloppy
\title{Effective surface energies in nematic liquid crystals\\ as homogenised rugosity effects}
\author[1,4,5]{Razvan-Dumitru Ceuca}
\author[1]{Jamie M. Taylor}
\author[1,2,3]{Arghir Zarnescu}
\affil[1]{Basque Center for Applied Mathematics (BCAM), Bilbao, Bizkaia, Spain}
\affil[2]{
IKERBASQUE, Basque Foundation for Science, Maria Diaz de Haro 3, 48013, Bilbao, Bizkaia, Spain}
\affil[3]{“Simion Stoilow” Institute of the Romanian Academy, 21 Calea Grivit¸ei, 010702 Bucharest, Romania.}
\affil[4]{UPV/EHU, Universidad del Pa\'is Vasco/Euskal Herriko Unibertsitatea, Barrio Sarriena s/n, 48940 Leioa, Bizkaia, Spain}
\affil[5]{UAIC, Alexandru Ioan Cuza University of Iasi, Bulevardul Carol I, Nr. 11, 700506 Iasi, Romania
1}
\date{}
\maketitle
\abstract{We study  the effect of boundary rugosity in nematic liquid crystalline systems. We consider a highly general formulation of the problem, able to simultaneously deal with 
several liquid crystal theories. We use techniques of Gamma convergence and demonstrate that the effect of fine-scale surface oscillations may be replaced by an effective homogenised surface energy on a simpler domain. The  homogenisation limit is then quantitatively studied in a simplified setting, obtaining
convergence rates.
}

\section{Introduction}
Nematic liquid crystals are anisotropic  states of matter formed from elongated molecules that share attributes with both crystals and isotropic liquids. Whilst they appear as ``classical" fluids to the naked eye, the constituent rod-like molecules of a nematic admit long-range orientational ordering which leads to a macroscopic anisotropy of the system  and optical birefringence, a property more typical in solid crystals. More so, as these anisotropic structures often admit characteristic length scales comparable to the wavelength of visible light, this allows them to play key roles in a variety of optics applications, the most prominent of which is the liquid crystal display (LCD) \cite{kawamoto2002history,yeh2009optics}. 

While anisotropy is found in many systems and may be exploited in technological applications, the prevalence of liquid crystals arises not only from the interactions it is capable of having with electromagnetic fields and light, but from the fact that they posses a broken continuous symmetry. The consequence of this is that it is possible to reorient or reorganise the structure of a nematic via the imposition of modest external fields. That is to say, the material is {\it soft}. As well as their long history of technological use, in more recent history there has been a sustained interest in liquid crystalline systems in biological applications, where many naturally occuring biological structures may be identified with liquid crystalline phases (such as the structure of a bi-lipid membrane, or the collective dynamics of active systems from the scale of swimming bacteria to flocks of birds \cite{dell2018growing,heller1993molecular,toner1998flocks}), and complex materials with a liquid crystalline nature have presented themselves as candidate for soft prosthesis \cite{jeong2013advancements,thomsen2001liquid}.

 In both biological and technological contexts, the behaviour of liquid crystalline materials is highly dependent on the environment in which they are found. Owing to the dielectric and diamagnetic anisotropy of mesogenic (rod-like) molecules, the imposition of electromagnetic fields has become the standard method to control nematics in technological applications. Another important component of the system which determines the ground states, however, is the interactions that the material has with its bounding container, that is, surface effects. It should be noted that liquid crystals may have interfaces with air, other liquids, solids, or even have situations such as wetting where nematic-air and nematic-solid interfaces each play a role, and each can provide a variety of behaviours depending on the particular materials involved \cite{bates1997molecular, meyer1972point, pan1987laser,pizzirusso2012predicting, slavinec1997determination}. In this work we will consider the case which is most apparent in technological applications, that of the nematic-solid interface. 

The classical approach to modelling nematic-solid interfaces in a nematic system is to presume that the surface is somewhat smooth, and molecules generally have a preference to sit  at a given angle to the surface, with the most common cases being that they prefer to be parallel (planar degenerate anchoring) or perpendicular (homeotropic anchoring)  to the surface. A surface anchoring energy represents this preference  in mathematical modelling, although at physically meaningful length scales it is often reasonable to replace the surface energy with a Dirichletboundary condition, which may be viewed as a case in which the effective surface energy becomes infinitely large away from its minimiser. Given that the geometry of the surface must therefore cause fluctuations in the nematic, at least near the surface itself, it is natural to wish to quantify these effects, and more so ask ourselves if we can exploit them. Additive manufacturing processes and photolithography open the possibility of producing well-designed surfaces which may improve control of nematic systems \cite{cheng1979liquid,niitsuma2008contact,yi2009topographic}. More so, the ZBD (Zenithal Bistable Device) is a successful commercial bistable display that exploits surface-induced distortions of the nematic, showing a glimpse of the future possibilities \cite{bryan1997grating,wood200011}.

As often happens in soft-matter physics, the presence of multiple length scales in nematics leads to a variety of models that one may use to describe them, depending on the types of features and mechanisms that one wishes to capture and describe. Previous research into the behaviour of nematic-solid interfaces is no different, with molecular dynamics simulations, including simulations up to atom-level resolution \cite{roscioni2013predicting,roscioni2017predicting,zhang1996substrate}, mean-field approaches \cite{osipov1993density,teixeira2016nematic,yokoyama1997density}, and the continuum mechanics language of PDEs and the calculus of variations \cite{romero2010scaling,sen1987landau}. In particular there has been a wealth of research into understanding the effects of one-dimensional periodic surfaces, with a particular focus on understanding how these provoke the presence of defects and create ``effective" surface energies that depend on the structure itself \cite{harnau2007effective,kondrat2005nematic}. 

Within this work we will consider the continuum mechanics approach, and aim to understand via rigorous asymptotic analysis the way in which undulated surfaces, in which the wavelength is of comparable size to the amplitude, can lead to effective surface energies in the limit as the amplitude converges to zero. Problems of this flavour have been considered in the language of homogenisation of PDEs in a domain with an oscillating boundary, where certain scalar, linear, {\it rugose} systems may be rigorously proven to have certain effective behaviours in the limit. These have been considered, for example, in the context of \cite{allaire1999,amirat2011,avellaneda1987,belyaev1988,chechkin1999} or \cite{friedman1997}, but the list is not by any means exhaustive. A contemporary overview of the literature from this direction can be found,  for example, in the introduction of \cite{amirat2011}. The nature of physically meaningful surface energies in the context nematic liquid crystals however provides models that have yet to be considered in the literature within this homogenisation framework. 

Our approach will be two-pronged. Firstly, we will consider a large, general class of models, which include the ``classical" non-linear models for nematic liquid crystals in arbitrary dimension, as well as the mathematically similar Ginzburg-Landau model of superconductivity \cite{bethuel1994ginzburg}, and demonstrate through techniques of $\Gamma$-convergence that, under appropriate conditions and within a large class of energy functionals, in the limit as the rugosity converges to zero, we can obtain an effective surface energy, whilst the interior energy of the nematic is unchanged. 

Our second direction considers a simplified setting of a two-dimensional slab with periodic rugosity and a quadratic free energy, which provides a toy model of a {\it paranematic}. That is, a high-temperature system of mesogenic molecules which has melted into an isotropic state, but still admits some local nematic ordering induced by the surface. In this case, the simplicity of the system compared to the general case allows us to make much stronger conclusions, and we are able to provide quantitative estimates on how ground states behave in the homogenised limit. 

\subsection{Models of liquid crystals}
Owing to the presence of multiple length scales in a liquid crystal, a large number of models exist, coming in many distinct flavours, which aim to describe particular mechanisms and predict or explain particular behaviours. Within this work we will only consider the perspective of continuum mechanics, in which the language used is that of partial differential equations and the calculus of variations. However, even within this class of models there are still various descriptions that capture different aspects of the material. 

\subsubsection{Landau-de Gennes}
The Landau-de Gennes Q-tensor theory is a popular model for describing nematic liquid crystals, owing to its ability to describe biaxiality and defects \cite{de1993physics,mottram2014introduction}. If we have a domain $\Omega\subset\mathbb{R}^3$ which contains nematic liquid crystal, we presume that the (local) distribution of orientations near a point $x\in\Omega$ can be described by a probability distribution $f_x:\mathbb{S}^2\to[0,\infty)$, where we view $p\in\mathbb{S}^2$ to represent the orientation of the long axis of a molecule. Molecules are presumed to be (statistically) head-to-tail symmetric, which gives that $f_x(p)=f_x(-p)$, and consequently the first order moment of $f_x$ always vanishes. To this end, in the process of coarse-graining, we consider the second normalised moment $Q$, which is a traceless symmetric $3\times 3$ matrix known as the {\it Q-tensor}, defined by 
\begin{equation}\label{eqQtensor}
Q(x)=\int_{\mathbb{S}^2}f_x(p)\left(p\otimes p-\frac{1}{3}I\right)\,dp\in \text{Sym}_0(3)=\left\{A\in \mathbb{R}^{3\times 3}:\, \text{Tr}(A)=0,\, A=A^T\right\}.
\end{equation}
We highlight several characteristics of the Q-tensor:
\begin{itemize}
\item If $f_x(p)=\frac{1}{4\pi}$, that is, the system is isotropic (disordered), then $Q(x)=0$. 
\item If $f_x$ is axisymmetric, that is, $f_x(p)=g_x(p\cdot e)$ for a distinguished direction $e\in\mathbb{S}^2$ and some $g_x:[-1,1]\to\mathbb{R}$, then $Q(x)=s\left(e\otimes e-\frac{1}{3}I\right)$ for some scalar $s$. We say that such a Q-tensor is {\it uniaxial}, with {\it director} $e$ and {\it scalar order parameter} $s$.
\item Any $Q(x)$ which can be defined via \eqref{eqQtensor} for $f_x\in L^1(\mathbb{S}^2)$ must satisfy the eigenvalue constraint $\lambda_{\min}(Q)>-\frac{1}{3}$, often referred to as {\it physicality} \cite{ball2010nematic}.
\end{itemize}

For simplicity, we neglect the ``microscopic" definition of $Q$ as in \eqref{eqQtensor} and consider ${Q:\Omega\to\text{Sym}_0(3)}$ to be the order parameter of interest, and interpret stable equilibria of the nematic as minimisers of the free energy 
\begin{equation}
\int_\Omega F_{el}(\nabla Q(x),Q(x))+F_b(Q(x))\,dx+\int_{\partial\Omega}w(\nu(x),Q(x))\,dS(x).
\end{equation}
Here $F_{el}: \left(\text{Sym}_0(3)\times\mathbb{R}^3\right)\times \text{Sym}_0(3)\to [0,\infty)$ is a frame indifferent elastic energy which penalises spatial distortions. The elastic energy is typically quadratic in $\nabla Q$, which gives $W^{1,2}$ as the natural function space for the problem. A commonly considered form for $\mathcal{F}_{el}$ is 
\begin{equation}
F_{el}(D,Q)=\frac{L_1}{2}D_{ijk}D_{ijk}+\frac{L_2}{2}D_{ijk}D_{ikj}+\frac{L_3}{2}D_{ijj}D_{ikk},
\end{equation}
where Einstein notation is used,  and $L_1,L_2,L_3$ are constants. More general forms are possible, however, such as in \cite{longa1987extension}.

$F_b$ is a frame indifferent bulk energy, which is minimised either at isotropic Q-tensors ($Q=0$), or uniaxial Q-tensors with a certain fixed, positive scalar order parameter ($Q=s_0\left(n\otimes n-\frac{1}{3}I\right)$, for $s_0>0$ fixed and any $n\in\mathbb{S}^2$). The typical expressions for $Q$ are either quartic or sixth-order polynomials in $Q$, the most used being the quartic one
\begin{equation}
F_b(Q)=\frac{a}{2}\text{Tr}(Q^2)-\frac{b}{3}\text{Tr}(Q^3)+\frac{c}{4}\text{Tr}(Q^4).
\end{equation} 
for material constants $b>0,c>0$, while $a$ depends on both the material and temperature, and may either be positive or negative. Furthermore Katriel {\it et. al.} \cite{katriel1986free} and Ball and Majumdar \cite{ball2010nematic} have provided examples of a bulk energy which is singular as $\lambda_{\min}(Q)\to-\frac{1}{3}$, which is employed to ensure the physicality of solutions. In the case of {\it paranematics}, where the  the temperature is above the nematic-isoptric transition temperature and $Q=0$ is the ground state, it is possible to obtain some degree of nematic order,  that is, non-zero $Q$, by the imposition of external fields or interactions with surfaces or colloids. In this case, one may replace the bulk energy with a quadratic, penalising the deviation from this ground state, $F_b(Q)=\frac{c}{2}|Q|^2$ for $c>0$, as has been done in \cite{galatola2001nematic,galatola2003interaction}.

The surface energy $w$ depends not only on $Q$, but the outer unit normal $\nu:\partial\Omega\to\mathbb{S}^2$, which is necessary to capture the behaviour that molecules tend to prefer  a specific angle with respect to the surface, and will depend on the particular materials considered. There are a wide variety of surface energies considered in the literature, but via frame indifference, it is known that $w$ must be of the form 
\begin{equation}
w(\nu,Q)=\tilde{w}(\text{Tr}(Q^2),\text{Tr}(Q^3),Q\nu\cdot\nu,Q^2\nu\cdot\nu)
\end{equation}
for some function $\tilde{w}$, which depends only on the scalar invariants of $Q,\nu$ \cite{canevari2020design,smith1971isotropic}. The simplest of all such models comes from the presumption of homeotropic alignment, and impose the surface energy 
\begin{equation}
w(\nu,Q)=\frac{w_0}{2}\left|Q-s_0\left(\nu\otimes\nu-\frac{1}{3}I\right)\right|.
\end{equation}

It should be noted that, for mathematical simplicity, analogues of the Q-tensor model can be defined for a two-dimensional universe, in which case we define $Q\in \text{Sym}_0(2)$ to be a $2\times 2$ symmetric traceless matrix, defined as the average of $\left(p\otimes p-\frac{1}{2}I\right)$ for $p\in\mathbb{S}^1$. 

\subsubsection{Oseen-Frank}
We recall that in a nematic described by the Landau-de Gennes model, the ``bulk" energy $F_b$ is minimised at $Q$ of the form $Q=s_0\left(n\otimes n-\frac{1}{3}I\right)$, where $s_0>0$ is fixed and $n$ can be any unit vector. It is known that in domains much larger than the characteristic correlation length of the molecules the penalisation of the bulk energy dominates, and it is appropriate to impose the constraint that $Q$ minimises the bulk energy pointwise \cite{bauman2012analysis,canevari2021,gartland2015scalings,liu2018oseen, majumdar2010landau,taylor2018oseen}. In certain cases, at least, this becomes equivalent to the Oseen-Frank energy, in which we use only the director $n:\Omega\to\mathbb{S}^2$ to describe the nematic \cite{ball2011orientability}. The typical free energy we consider is then 
\begin{equation}\label{eqOF}
\int_\Omega F_{el}(n(x),\nabla n(x))\,dx+\int_{\partial\Omega}w(\nu(x),n(x))\,dS(x).
\end{equation}

The elastic energy is typically given of the form 
\begin{equation}
F_{el}(n,\nabla n)=\frac{K_1}{2}(\nabla \cdot n)^2+\frac{K_2}{2}(n\cdot\nabla\times n-\tau)^2+\frac{K_3}{2}|n\times\nabla\times n|^2,
\end{equation}
where $K_1,K_2,K_3,\tau$ are material constants. Note that if $\tau=0$, the energy is minimised at the undistorted state, $\nabla n=0$. However if $\tau\neq 0$, then the energy is minimised at a {\it cholesteric} state, where the ground state admits a helical structure \cite{bedford2014global}.

For the surface energy, frame invariance implies that $w$ may only depend on $(\nu\cdot n)^2$. The simplest possible such surface energy is the Rapini-Papoular surface energy \cite{rapini1969distorsion}
\begin{equation}
w(\nu,n)=\frac{w_0}{2}|\nu\cdot n|^2,
\end{equation}
where $w_0>0$ implies that the director wishes to be in the tangent plane to the surface (planar degenerate anchoring), and $w_0<0$ implies that director prefers to be orthogonal to the surface (homeotropic anchoring).
\begin{remark}\label{remarkOFVector}
An equivalent formulation of the Oseen-Frank model \eqref{eqOF} is to consider director fields $n\in W^{1,2}(\Omega,\mathbb{R}^3)$ with the free energy 
\begin{equation}
\int_\Omega F_{el}(n(x),\nabla n(x))+F_b(n(x))\,dx+\int_{\partial\Omega}w(\nu(x),n(x))\,dS(x),
\end{equation}
where $F_{el}$,$w$ are as in \eqref{eqOF} and the bulk energy is defined by $F_b(n)=0$ if $|n|=1$, and $F_b(n)=+\infty$ otherwise. This particular formulation allows us to consider $n$ as living in a vector space, rather than a Sobolev space of manifold-valued maps, albeit with a highly singular energy functional. 
\end{remark}

\subsection{Outline of the paper and main results}

\label{theoremGammaConvergence}

The paper is broadly organised into three sections. In \Cref{sect:Gamma_conv_general}, we consider a general problem that encompasses the main theories of nematic liquid crystals, and we consider the asymptotic limit as the surface oscillations converge to zero with fine structure. In \Cref{subsecModels}, we interpret these general results in the context of the Landau-de Gennes and Oseen Frank models of nematic liquid crystals. Finally, in \Cref{sect:Error}, we consider a simplified model in which we are able to proof more precise estimates on the asymptotic behaviour of the model, including error estimates in $L^2$. We now outline the key aspects of our analysis. 

\subsubsection{Summary of \Cref{sect:Gamma_conv_general}}

Within this section we consider integral functionals of the form 

\begin{equation}
\mathcal{F}_\e(u)=\int_{\Omega_\e}F(\nabla u(x),u(x),x)\,dx + \int_{\Gamma_\e}w(P_\Gamma x,\nu(x),n(x))\,dS(x). 
\end{equation}
for $u\in W^{1,p}(\Omega_\e,V)$ for a finite dimensional vector space $V$, and $P_\Gamma$ the projection onto a particular manifold $\Gamma$. The precise assumptions necessary for our analysis are presented in  \Cref{subsecTechnicalAssumptions}. One of the key assumptions is that $\Omega_\e$ is a graph of a function defined on a limiting domain $\Omega$, encoded by a function $\varphi_\e:\Gamma\to\mathbb{R}$, satisfying $||\varphi_\e||_\infty =O(\e)$, $||D\varphi_\e||_\infty=O(1)$, which generates a Young measure in a particular sense, as outlined in \Cref{assumpOmegaEps}. The assumptions on $F,w$ account for many typical models encountered in the modelling of liquid crystals. Furthermore, as the domains $\Omega_\e$ are $\e$ dependent, we need to define a notion of convergence of functions $u_\e\in W^{1,p}(\Omega_\e,V)$ to functions in $u\in W^{1,p}(\Omega,V)$, which we call {\it rugose convergence}, denoted $u_\e\rto u$, and is defined in \Cref{defRugoseConv}. Once we have established our technical assumptions, we proceed with identifying the $\Gamma$-limit of $\mf_\e$ with respect to rugose convergence.

First we prove our equicoercivity result, showing that if $\mathcal{F}_\e(u_\e)$ is uniformly bounded, then we may extract a subsequence $u_{\e_j}$ and $u\in W^{1,p}(\Omega,V)$ so that $u_{\e_j}\rto u$ (\Cref{propCompactness}). The argument is standard, and reflects that given in, for instance, \cite{belyaev1988},\cite{chechkin1999} and \cite{friedman1997}, whereby the only technicality is to ensure that the rugose limit is weakly differentiable with the correct derivative. Generally, we see that under rugose convergence, the interior term, that is, 
$$\int_{ \Omega_\e} F(\nabla u(x),u(x),x)\,dx,$$
is ``well-behaved", and does not require much technical consideration. For the surface term however, first we note that, up to an small error, the surface term may be replaced by a simpler function that, rather than depending on the full geometry, only sees the rugosity via $\hat{D}\varphi_\e$, according to \Cref{corollarySimplerSurface}. In particular, the argument shows that curvature  of $\Omega$ has a vanishing contribution in the limit. This reduction in complexity of the model also allows us to relate the surface energy in the limit with the Young measures described in \Cref{assumpOmegaEps}. Once this has been done, we show that rugose convergence $u_\e\rto u$ also implies the weak-$W^{1,p}$ convergence of $u_\e$ composed with a change of variables $\Phi_\e:\Omega\to \Omega_\e$\Cref{lemmaChangeOfVars}, which gives us a well-behaved notion of a type of ``trace" onto $\Gamma$ for functions in $W^{1,p}(\Omega_\e,V)$, and is used to show that for $u_\e\rto u$, 
\begin{equation}
\int_{\Gamma_\e} w(x, \nu_\e(x),u_\e(x))\,dS(x) \to \int_\Gamma w_h(x,u(x))\,dS(x)
\end{equation}
in \Cref{propSurfaceContinuous}, where $w_h$ admits an explicit representation in \eqref{eqHomoSurface}. Combining this ``continuity" of the surface term with classical lower semicontinuity results on the interior gives the liminf inequality necessary for the $\Gamma$-convergence (\Cref{propLiminf}). To show the limsup inequality, we construct recovery sequences by extending our candidate $u$ via reflections to a larger domain, and considering its restriction to $\Omega_\e$ (\Cref{propLimsup}).

{ Finally, we note that $\Gamma$ convergence, strictly speaking, requires that the domain of definition of $\mathcal{F}_\e$ is independent of $\e>0$, which is not the case in the problem we consider. However, this will not be problematic, as our conclusion is to employ the fundamental theorem of $\Gamma$-convergence, which holds nonetheless via the exact same argument as the classical case (\Cref{propGammaConvVarying}).

With these ingredients, we are able to prove the main theorem of this section.

\begin{theorem}
Let $\Omega_\e$, for $\e_0>\e>0$, $\Omega$ denote sets satisfying \Cref{assumpOmega} and \Cref{assumpOmegaEps}, the interior energy $F$ satisfy \Cref{assumpInterior}, and the surface energy $w$ satisfy \Cref{assumpSurface1}. Then there exists a local energy functional  $w_{h}:\Gamma\times V\to\mathbb{R}$ so that for every $u\in L^q$, $\lim\limits_{\e\to 0}w(\cdot,\tilde{\nu}_\e,u)\gamma_\e\overset{L^1(\Gamma)}{\rightharpoonup} w_{h}(\cdot,u)$, and $\mf_\e\overset{\Gamma}{\to} \mf$ with respect to $W^{1,p}$-rugose convergence, where $\mf:W^{1,p}(\Omega,V)\to[0,\infty)$ is given by 
\begin{equation}
\mf(u)=\int_\Omega F(\nabla u(x),u(x),x)\,dx +\int_{\Gamma}w_{h}(x,u(x))\,dS(x).
\end{equation} 
The homogenised surface energy $w_h:\Gamma\times V\to\mathbb{R}$ is given by 
\begin{equation}
w_h(x,u)=\int_{\nu+T_x\Gamma}w\left(x,\frac{1}{|v|}v,u\right)|v|\,d\mu_x(v),
\end{equation}
where $(\mu_x)_{x\in\Gamma}$ is the Young measure generated by $\hat{D}\varphi_\e:\Gamma\to \mathbb{R}^N$ as given in \Cref{assumpSurface2}.
Furthermore the energies $\mf_\e$ are equicoercive with respect to $W^{1,p}$-rugose convergence, and if $u_\e$ are minimisers of $\mf_\e$ for every $\e_0>\e>0$, there exists a subsequence $\e_j\to 0$ and minimiser $u$ of $\mf$ such that $u_{\e_j}\rto u$.  
\end{theorem}

\subsubsection{Summary of \Cref{subsecModels}}

We consider the Oseen-Frank and the Landau-de Gennes models for nematic liquid crystals and we illustrate in both cases an example of homogenised surface energy for the limit problem, depending on the choice of surface energy for the rugose problem.

For the Oseen-Frank model, we consider the Rapini-Papoular surface energy, which is given by
\begin{align*}
w(\nu,n)=\dfrac{w_0}{2}\big(n\cdot\nu\big)^2,
\end{align*}
where $n\in W^{1,2}(\Omega,\mathbb{R}^3)$, and we prove that the homogenised surface energy in this case is of the form
\begin{align*}
w_h(x,n)=\dfrac{1}{2}A_{ef}\;n\cdot n,
\end{align*}
where $A_{ef}:\Gamma\to\mathbb{R}^{3\times 3}$ is a $3\times 3$ symmetric tensor given by
\begin{align*}
A_{ef}=w_0\int_{\nu(x)+T_x\Gamma}\dfrac{1}{|v|}v\otimes v\;\text{d}\mu_x(v).
\end{align*}

For the Landau-de Gennes model, we consider the following form for the surface energy:
\begin{align*}
w(x,\nu,Q)=\dfrac{w_0}{2}\bigg|Q-s_0\bigg(\nu\otimes\nu-\dfrac{1}{3}I\bigg)\bigg|^2
\end{align*}
and we prove that, in this case, that the homogenised surface energy is of the following form:
\begin{align*}
w_h(x,Q)=\dfrac{w_{ef}}{2}\big|Q-Q_{ef}\big|^2+R,
\end{align*}
where $R$ is a remainder that depends only on $x$ and not on $Q$, hence it does not affect any minimisers, and the terms $w_{ef}$ and $Q_{ef}$ are given by:
\begin{align*}
& w_{ef}=w_0\int_{\nu(x)+T_x\Gamma}|v|\;\text{d}\mu_x,\\
& Q_{ef}=\dfrac{s_0}{w_{ef}}\int_{\nu(x)+T_x\Gamma}\bigg(\dfrac{1}{|v|^2}v\otimes v-\dfrac{1}{3}I\bigg)|v|\;\text{d}\mu_x(v).
\end{align*}

\subsubsection{Summary of \Cref{sect:Error}}

In physical situations one has, after suitable non-dimensionalisations, a small but finite $\varepsilon$ parameter. Thus in order to understand the relevance of the theories previously developed, that study the limit problem, one needs to obtain error estimates. Moreover one aims to obtain error estimates with a convergence as good as possible and it will turn out that our main result shows that by using weaker norms one can improve the convergence rate, without needing to resort to correctors. In order to understand these issues we analyze the most simplified setting possible that still has a certain physical relevance.  

\bigskip

We consider the situation of a two-dimensional slab with periodic rugosity and the Landau-de Gennes model for the description of the nematic liquid crystal used. More specifically, the limiting domain is of the form $\Omega_0=\{(x,y)\;|\;x\in[0,2\pi),\;y\in(0,R)\}$, where $R>0$ is a constant, and the rugose domain is of the form $\Omega_{\e}=\{(x,y)\;|\;x\in[0,2\pi),\;y\in(\varphi_{\e}(x),R)\}$, where $\varphi_{\e}(x)=\e\varphi(x/\e)$ and $\varphi:\mathbb{R}\to\mathbb{R}$ is a $C^2$ $2\pi$-periodic function with $\varphi\geq 0$. We denote with $\Gamma_{\e}=\{(x,\e\varphi(x/\e))\;|\;x\in[0,2\pi)\}$ the rugose boundary and with $\Gamma_R=\{(x,R)\;|\;x\in[0,2\pi)\}$ the fixed upper boundary of the domains. We consider a quadratic free energy of the following form:
\begin{align*}
\mathcal{F}_{\e}[Q]=\int_{\Omega_{\e}}\big|\nabla Q\big|^2+c|Q|^2\;\text{d}(x,y)+\int_{\Gamma_{\e}}\dfrac{w_0}{2}\big|Q-Q_{\e}^0\big|^2\;\text{d}\sigma_{\e}+\int_{\Gamma_R}\dfrac{w_0}{2}\big|Q-Q_{R}\big|^2\;\text{d}\sigma_R,
\end{align*}
where $c>0$ is constant, $w_0>0$ is the anchoring strength, $Q_{\e}^0=\nu_{\e}\otimes\nu_{\e}-\frac{1}{2}I$ and $Q_R=\nu_R\otimes\nu_R-\frac{1}{2}I$  $\big(\nu_{\e}$ and $\nu_R$ are the outward normals to $\Gamma_{\e}$ and $\Gamma_R\big)$.

In this simplified model, using \Cref{prop:Q_eff}, we are able to identify the homogenised surface energy as $\dfrac{w_{ef}}{2}\big|Q-Q_{ef}\big|^2$, with $w_{ef}={w_0}\gamma$ and $Q_{ef}=\dfrac{1}{\gamma}\begin{pmatrix}
G_1 & G_2\\
G_2 & -G_1
\end{pmatrix}$, where $\gamma$, $G_1$ and $G_2$ are defined in \Cref{defn:g_functions}. The homogenised free energy functional is then of the form
\begin{align*}
\mathcal{F}_0[Q]=\int_{\Omega_0}\big|\nabla Q\big|^2+c|Q|^2\;\text{d}(x,y)+\int_{\Gamma_0}\dfrac{w_{ef}}{2}\big|Q-Q_{ef}\big|^2\;\text{d}\sigma_0+\int_{\Gamma_R}\dfrac{w_0}{2}\big|Q-Q_{R}\big|^2\;\text{d}\sigma_R,
\end{align*}
where $\nu_0$ is the outward normal to $\Gamma_0=\{(x,0)\;|\;x\in[0,2\pi)\}$.

\bigskip
Let $Q_{\e}$ be the minimiser of $\mathcal{F}_{\e}$ and $Q_0$ the minimiser of $\mathcal{F}_0$. In \cite{belyaev1988} and \cite{friedman1997}, the authors are able to prove that $\|Q_{\e}-Q_0\|_{H^1(\Omega_{\e})}\leq C\sqrt{\e}$. According to \cite{chechkin1999}, our simplified model is under the case $0=\beta=\alpha-1$, in which they prove that $\|Q_{\e}-Q_0\|_{H^1(\Omega_{\e})}\leq K_2(\sqrt{\e}+1)$. Both in \cite{allaire1993} and \cite{amirat2011}, it is proved that $\big(Q_{\e}\big)_{\e>0}$ converges strongly in $L^2(\Omega_{\e})$ to $Q_0$, under various assumptions for the domains. Using boundary layers, in \cite{allaire1999} the authors are able to prove that $\|Q_{\e}-Q_0-\e Q_1\|_{H^1(\Omega_{\e})}\leq C\sqrt{\e}$, where $Q_1$ is a first-order boundary term. In this work, we are able to prove the following $L^2$ error estimate:

\begin{theorem}
For any $p\in(2,+\infty)$, there exists an $\e$-independent constant $C$ such that:
\begin{align*}
\|Q_0-Q_{\e}\|_{L^2(\Omega_{\e})}\leq C\cdot \e^{\frac{p-1}{p}},
\end{align*}
where the constant $C$ depends on $c$, $w_0$, $p$, $\|\varphi\|_{L^{\infty}([0,2\pi))}$, $\|\varphi'\|_{L^{\infty}([0,2\pi))}$, $\Omega_0$ and $\|Q_0\|_{W^{1,\infty}(\Omega_0)}$.
\end{theorem}

It is easy to observe that the fraction $\frac{p-1}{p}$, with $p\in(2,+\infty)$, allows us to obtain any desired exponent from the interval $(1/2,1)$. In order to prove this theorem, we first show in \Cref{subsect:reg} that $Q_{\e}$ and $Q_0$ exist and admit $W^{2,p}$ regularity, for any $p\in(2,+\infty)$. Then, we adapt in \Cref{subsect:integral_inequalities} the proofs from \cite{belyaev1988} and \cite{friedman1997} to the case of $W^{1,p}$ functions in order to obtain \Cref{prop:ineq_Omega_eps_rhs}. The result that dictates the exponent of $\e$ from our error estimate is \Cref{lemma:C_3}. A similar estimate to this lemma represents \cite[Lemma 5.1, (15)]{amirat2011}, where the exponent obtained is $\frac{d+2}{2d}$ for $L^{\frac{2d}{d-2}}$ estimates, for any $d>2$. The proof of our error estimate is also based on the construction of an extension operator, from $W^{1,p}(\Omega_{\e})$ to $W^{1,p}(\Omega_0)$, which is defined in \Cref{defn:uniform_ext} and has $\e$-independent bounds. With all of these ingredients, we are able then to prove the main result of this part, in \Cref{subsect:proof_err_est}.

\section{$\Gamma$-convergence in the general case}\label{sect:Gamma_conv_general}

Within this section we aim to describe the asymptotic behaviour of minimisers of a general class of nonlinear energy functionals. It is well known that pointwise convergence of a sequence of energy functionals is not sufficient to guarantee convergence of their minimisers to a minimiser of the limiting energy. In response to this challenge, many techniques have been proposed, and within this section we will use the techniques of $\Gamma$-convergence, of which we recall the definition below, following \cite[Theorem 2.1]{braides2006handbook}

\begin{definition}
Let $X$ be a topological space and $\mathcal{F}_\e:X\to\mathbb{R}\cup\{+\infty\}$ be energy functionals for every $\e>0$. Then we say that $\mf_\e$ $\Gamma$-converges to $\mf:X\to\mathbb{R}\cup\{+\infty\}$ if the following hold:
\begin{itemize}
\item (Liminf inequality) For every sequence $x_\e\to x$ in $X$, we have that $\liminf\limits_{\e\to 0}\mf_\e(x_\e)\geq \mf(x)$. 
\item (Limsup inequality) For every $x\in X$, there exists a sequence $x_\e\to x$ such that $\limsup\limits_{\e\to 0}\mf_\e(x_\e)=\mf(x)$. 
\end{itemize}
Furthermore, we say that the functionals are {\it equicoercive} if for any sequence $(x_\e)_{\e>0}$ such that $\mf_\e(x_\e)$ is uniformly bounded, there exists a subsequence $\e_j\to 0$ such that $x_{\e_j}\to x$ for some $x\in X$. 
\end{definition}

The relevance of $\Gamma$-convergence in minimisation problems is made apparent by the fundamental theorem of $\Gamma$-convergence, which states that if $\mf_\e$ are a series of equicoercive functionals and $\mf_\e$ $\Gamma$-converges to $\mf$, then for any sequence $x_\e$ such that $\lim\limits_{\e\to 0}\mf_\e(x_\e)-\inf\limits_{x'} \mf_\e(x')=0$, there exists a subsequence $\e_j\to 0$ and $x\in X$,  such that $x_{\e_j}\to x$ and $\mf(x)=\min\limits_{x'}\mf(x')$ (see, e.g., \cite[Theorem 2.10]{braides2006handbook}). In words, a sequence of almost minimisers of $\mf_\e$, up to a subsequence, converges to a minimiser of the $\Gamma$-limit. Furthermore, $\min\limits_{x'}\mf(x')=\lim\limits_{\e\to 0}\inf\limits_{x'\in X}\mf_\e(x')$.

\subsection{Technical assumptions and statement of the main results}\label{subsecTechnicalAssumptions}

\subsubsection{Geometric assumptions}
\begin{assumption}[Properties of $\Omega$]\label{assumpOmega}
$\Omega\subset\mathbb{R}^N$ is a bounded domain with boundary $\Gamma$. $\Gamma$ is a $C^2$ compact manifold without boundary of dimension $N-1$. Throughout, the function $\nu:\Gamma\to\mathbb{S}^{N-1}$ will denote the outward unit normal vector to $\Gamma$.

\end{assumption}

\begin{assumption}\label{assumpOmegaEps}
$\Omega_\e$ will denote uniformly bounded open subsets of $\mathbb{R}^N$ such that:
\begin{enumerate}
\item For every $U\subset\subset \Omega$, there exists $\e_0>0$ so that if $\e<\e_0$,  $U\subset\subset \Omega_\e$. 
\item For every $U\subset\subset \mathbb{R}^N\setminus \Omega$, there exists $\e_0>0$ such that for every $\e<\e_0$, $U\cap \Omega_\e=\emptyset $.
\end{enumerate} 
We denote $\Gamma_\e=\partial\Omega_\e$. We require that there exists $\varphi_\e\in  W^{1,\infty}(\Gamma,\mathbb{R})$  with the following properties:
\begin{enumerate}
\item $\Gamma\ni x\mapsto x+\varphi_\e(x)\nu(x)\in\Gamma_\e$ defines a bijection.
\item $||\varphi_\e||_{L^\infty(\Gamma)} = O(\e)$.
\item $||D\varphi_\e||_{L^\infty(\Gamma)} =O(1)$.
\end{enumerate}
Furthermore, we denote $\nu_\e:\Gamma_\e\to \mathbb{S}^{N-1}$ the outer unit normal to $\Gamma_\e$, and the surface element $\gamma_\e:\Gamma\to (0,\infty)$ such that 
\begin{equation}
\int_{\Gamma_\e}g(x)\,dS(x)=\int_{\Gamma}g\Big(x+\varphi_\e(x)\nu(x)\Big)\gamma_\e(x)\,dS(x)
\end{equation}
for all integrable $g:\Gamma_\e\to\mathbb{R}$. Finally, we define $\tilde{\nu}_\e:\Gamma\to\mathbb{S}^{N-1}$ to be 
\begin{equation}
\tilde{\nu}_\e(x)=\nu_\e(x+\varphi_\e(x)\nu(x)).
\end{equation}
\end{assumption}

\begin{remark}
We note that $D$ refers to the derivative on $\Gamma$, $D\varphi(x)\in (T_{x}\Gamma)^*$. Given a $C^1$ curve $\xi:[-1,1]\to \Gamma$ with $\xi(0)=x$ and $\xi'(0)=\tau\in T_{x}\Gamma$, 
\begin{equation}
\langle D\varphi_\e(x),\tau\rangle= \frac{\partial\varphi_\e\circ\xi}{\partial t}(0).
\end{equation}
Furthermore, for an orthonormal basis $(\tau_i(x))_{i=1}^{N-1}$ of $T_x\Gamma$, we have that 
\begin{equation}
|D\varphi_\e(x)|^2_\Gamma=\sum\limits_{i=1}^{N-1}|\langle D\varphi_\e(x),\tau_i(x)\rangle|^2,
\end{equation}
where $|\cdot|_\Gamma$ is the norm on $T_x\Gamma$ induced by the metric on $\Gamma$.
\end{remark}

\begin{definition}\label{defProjection}
We denote $P_{\Gamma}:\mathbb{R}^N\to\Gamma$ the nearest point projection to $\Gamma$. We presume that $\delta$ is sufficiently small so that 
\begin{equation}
\{x\in\Omega : d(x,\partial\Omega)=\delta\}
\end{equation} 
is a $C^1$ domain, parametrised by 
\begin{equation}
\{x\in\Omega : d(x,\Gamma)=\delta\}=\{x-\delta\nu(x):x\in\partial\Omega \}. 
\end{equation}
Furthermore, we presume $\delta$ is sufficiently small so that $P_{\Gamma}|_{\Gamma+B_\delta}$ is a $C^1$ function. By \cite[Lemma 4]{lewis2008alternating}, as $\Gamma$ is compact, such a $\delta$ exists.
\end{definition}

\begin{remark}
 Let $X,V$ be metric spaces where $X$ is a positive, complete measure space. We say that a sequence of functions $f_i:X\to V$ generates a Young measure $(\mu_x)_{x\in X}$, where for a.e. $x\in X$, $\mu_x$ defines a positive measure on $V$, if for every continuous $F:V\to\mathbb{R}$ such that $F(f_i)$ is uniformly integrable,
\begin{equation}
\int_{X} F(f_i)\,dx\to \int_{X\times V} F(v)\,d\mu_x(v)\,dx.
\end{equation}
In fact, if this holds and $X$ is an open domain in $\mathbb{R}^M$ and $V=\mathbb{R}^N$, then following \cite{ball1989version}, for any function $F:X\times V\to\mathbb{R}$ which is $X$-measurable in its first variable and continuous in its second,  if the functions $x\mapsto F(x,f_i(x))$ are uniformly integrable, then 
\begin{equation}
\int_{X} F(x,f_i)\,dx\to \int_{X\times V} F(x,v)\,d\mu_x(v)\,dx.
\end{equation}
It is straightforward to extend this result to the case when $X=\Gamma$, which satisfies \Cref{assumpOmega}, by considering a finite open covering and local coordinate systems. 
\end{remark}

 To illustrate Young measures, we consider a simple example
\begin{example}
Let $\varphi:\mathbb{S}^1\times \mathbb{S}^1\to\mathbb{R}$ be a smooth function. We will freely identify $\mathbb{S}^1$ with $[0,2\pi]$, and $\varphi$ with a $2\pi$-doubly periodic function on $\mathbb{R}^2$, throughout this example. That is, $\varphi(x,y)=\varphi(x+2k_1\pi,y+2k_2\pi)$ for all $k_1,k_2\in\mathbb{Z}$. Let $f_i(x)=\varphi(x,ix)$ for integer $i$. Then, given any continuous function $F:\mathbb{R}\to \mathbb{R}$, we have that 
\begin{equation}\label{eqTwoScaleStuff}
\begin{split}
\lim\limits_{i\to \infty}\int_{\mathbb{S}^1}F(f_i(x))\,dx = &\lim\limits_{i\to \infty}\int_0^{2\pi}F(\varphi(x,ix))\,dx\\
= & \frac{1}{2\pi}\int_0^{2\pi}\int_0^{2\pi}F(\varphi(x,x'))\,dx'\,dx
\end{split}
\end{equation}
by, for example, \cite[Lemma 9.1]{cioranescu1999}. Now if, for every $x$, we define the measure 
\begin{equation}\label{eqYoungMeasure}
\mu_x(A)=\frac{1}{2\pi}|\{z\in [0,2\pi]:\varphi(x,z)\in A\}|,
\end{equation}
 we may re-write the integral as 
\begin{equation}
\frac{1}{2\pi}\int_0^{2\pi}F(\varphi(x,x'))\,dx'=\int_{\mathbb{R}}F(v)\,d\mu_x(v)\,dx 
\end{equation}
for every $x\in [0,2\pi]$. Thus for every continuous function $F:\mathbb{R}\to\mathbb{R}$, we have
\begin{equation}
\int_{\mathbb{S}^1}F(f_i(x))\,dx\to \int_{\mathbb{S}^1}\int_{\mathbb{R}}F(v)\,d\mu_x(v)\,dx 
\end{equation}
by Fubini's theorem. Thus $(\mu_x)_{x\in\mathbb{S}^1}$ is the family of Young measures generated by $f_i$.

Furthermore, it is immediate from \eqref{eqYoungMeasure} that $\mu_x(\mathbb{R})=1$. This property does not hold generally for Young measures, rather that $\mu_x(\mathbb{R})\leq 1$. 

We further note that $f_i$  in our example has weak-$L^2$ limit given by 
\begin{equation}
\int_0^{2\pi}\varphi(x,x')\,dx'=\int_{\mathbb{R}}v\,d\mu_x(v),
\end{equation}
which can be seen by taking $F(v)=v$ in \eqref{eqTwoScaleStuff}. The property that the weak limit of $f_i$ is, pointwise, given by the average of $\mu_x$ is general. This highlights the fact that Young measures contain information about the weak limit via its average, but also ``retain" information about oscillations that the weak limit that would otherwise lose. It is thus a natural language for the problem we consider, where we wish to keep track of how fine-scale oscillations contribute to our surface energy, and will later be quantified in \Cref{propWeakSurfaceConv}. 

\end{example}

\begin{assumption}\label{assumpSurface2}
 Given $\varphi_\e$ satisfying \Cref{assumpOmegaEps}, we define the map $\hat{D}\varphi_\e:\Gamma\to\mathbb{R}^N$ as follows. The derivative operator $D\varphi_\e:\Gamma\to (T_x\Gamma)^*$ acts on tangent vectors $\tau\in T_x\Gamma$ via the duality pairing $\langle D\varphi_\e(x),\tau\rangle$. At each point $x\in\Gamma$, we take an orthonormal basis $(\tau_i(x))_{i=1}^{N-1}$ of $T_x\Gamma$, and define $\hat{D}\varphi_\e(x)$ via
\begin{equation}\label{eqPseudoDeriv}
\hat{D}\varphi_\e(x)=\nu(x)-\sum\limits_{i=1}^{N-1}\langle D\varphi_\e(x),\tau_i(x)\rangle\tau_i(x),
\end{equation}
Furthermore, we presume that $\hat{D}\varphi_\e$ generates a Young measure $(\mu_x)_{x\in\Gamma}$.
\end{assumption}

\begin{remark}
 As the gradient $D\varphi_\e(x)$ is defined as a linear operator on the tangent space $T_x\Gamma$, we will employ $\hat{D}\varphi_\e$ as an analogue of the gradient of $\varphi_\e$ which is $\mathbb{R}^N$ valued, and will be necessary to simplify computations involving the surface normal of $\Gamma_\e$ in the sequel (see \Cref{corollarySimplerSurface}). We note that $\hat{D}\varphi_\e$ does not depend on the (pointwise) orthonormal bases taken to define $\tau_i$.

 We will provide some preliminary results necessary for dealing with Young measures in \Cref{propYoungMeasure}
\end{remark}

\subsubsection{Functional assumptions}

\begin{assumption}[Properties of $F$]\label{assumpInterior}
Let $V$ be a finite dimensional vector space.  We take $D$ to be a bounded set such that $\Omega \cup\bigcup\limits_{\e>0}\Omega_\e\subset D$. Then ${F:(V\times\mathbb{R}^N)\times V\times D \to [0,+\infty]}$ denotes a Carath\'eodory function, convex in its first variable, satisfying the following $p$-growth lower bound, that there exists some $C>0$ so that
\begin{equation}
F(A,u,x)\geq C(|A|^p+|u|^p-1)
\end{equation}
 for all $A\in V\times\mathbb{R}^N$, $u\in V$, $x\in D$. Furthermore we presume there exists a homogeneous functional $F_h:(V\times\mathbb{R}^N)\times V\to[0,\infty]$ and $C_1>1$ such that 
\begin{equation}
F_h(A,u)\leq F(A,u,x)\leq C_1 F_h(A,u)
\end{equation}
 for all $A\in V\times\mathbb{R}^N,u\in V,x\in D$. Furthermore, we presume that for every $L_0>1$, there exists some $C>0$ so that for every $L\in\mathbb{R}^{N\times N}$ with $ L_0^{-1}I<L^TL<L_0I$ (where the inequality corresponds to that as bilinear forms),

\begin{equation}\label{eqWeirdBound}
F_h(AL,u)\leq C\left(F_h(A,u)+1\right). 
\end{equation}
for all $A\in V\times\mathbb{R}^N,u\in V$.
Finally, we presume that there exists some $u\in W^{1,p}(\Omega,V)$, where $\Omega$ is as in \Cref{assumpOmega}, such that 
\begin{equation}\label{ass:uF}
\int_\Omega F(\nabla u(x),u(x),x)\,dx<+\infty. 
\end{equation}
\end{assumption}

\begin{remark}
The assumptions on $F$ are somewhat non-standard, and are designed to account  simultaneously for various energies relevant in the modelling of liquid crystals, where each has its own non-standard aspects. Generally, such energies find themselves of the form
\begin{equation}\label{eqSplitEnergy}
F(A,u,x)= F_{el}(A,u)+F_{b}(u),
\end{equation}
where the elastic energy $F_{el}$ is $p$-positively homogeneous in its first argument and $F_b$ may have  different growth to $F_{el}$. In particular, $F_b$ may only be finite on a compact or bounded set. These are typical in models of nematics, such  as Landau-de Gennes, where  $F_{el}$ is quadratic, and $F_b$ is either a polynomial or singular bulk potential  which lacks an upper bound in terms of $|Q|^2$. A more extreme case is Oseen-Frank, in which one may interpret the unit vector constraint $|u|=1$ almost everywhere as one of the form \eqref{eqSplitEnergy} with $F_b(u)=0$ if $|u|= 1$ and $F_b(u)=\infty$ otherwise, and thus no power of $u$ can bound $F$ from above, as discussed in \Cref{remarkOFVector}.

Furthermore, the presence of a magnetic field would lead to $x$-dependence in $F$, via the addition of a term $\chi (H\cdot n)^2$ for Oseen-Frank or $\chi QH\cdot H$, where $H:\mathbb{R}^3\to\mathbb{R}^3$ is the externally imposed magnetic field. 

It is immediate that these models all comply with \Cref{assumpInterior}, with the restriction that $H$, if present, be bounded.

The reason for \eqref{eqWeirdBound} is such that if $u\in W^{1,p}(U)$ and $g:U\to g(U)$ is a bi-lipschitz map, then for $\tilde{u}=u\circ g^{-1}$
\begin{equation}
\int_{g(U)} F_h(\nabla \tilde{u}(x),\tilde{u}(x))\,dx\leq CL_0\int_U F_h(\nabla u(x),u(x))+1\,dx,
\end{equation}
where $L_0$ depends on the bi-Lipschitz constants of $g$. This inequality will allow us to define finite-energy extensions of  $W^{1,p}(\Omega,V)$ functions via reflections in the sequel. 

Furthermore, we note that whilst we are considering $F$ to be convex in the gradient of $u$ this particular assumption may be relaxed to $F$ which defines a weakly lower semicontinuous energy functional in $W^{1,p}$. For example quasiconvex, non-convex energies, such as forms considered in a elasticity like $F(A,u,x)=|A|^p +\det(A)^{-1}$ with $\det(A)>0$ and $+\infty$ otherwise, with $p>N$, may be considered, provided the growth bounds we impose are satisfied. 
Finally, we note that the existence of $u\in W^{1,p}(\Omega,V)$ such that \eqref{ass:uF} holds suffices in order to ensure that the $\Gamma$-limit functional that will be obtained is non infinite everywhere.
\end{remark}

\begin{assumption}[Assumptions on $w$]\label{assumpSurface1}
Let $1<p<\infty$ be as in \Cref{assumpInterior} and $1\leq q< p^*$, where $p^*=\frac{(N-1)p}{N-p}$ if $p<N$ or $p^*=\infty$ if $p\geq N$. We presume that $w:\Gamma\times\mathbb{S}^{N-1}\times V\to [0,\infty)$ satisfies the following. 
\begin{itemize}
\item There exists $C>0$ such that for a.e. $x\in\Gamma$, every $\nu \in \mathbb{S}^{N-1}$, $u_1,u_2\in V$, ${|w(x,\nu,u_1)-w(x,\nu,u_2)|\leq C(|u_1|^{q-1}+|u_2|^{q-1}+1)|u_1-u_2|}$. 
\item There exists $C>0$ such that for a.e. $x\in\Gamma$, every $\nu\in\mathbb{S}^{N-1}$, $u\in V$, $|w(x,\nu,u)|<C(|u|^q+1)$
\item There exists $C>0$ such that for a.e. $x\in\Gamma$, every $\nu_1,\nu_2\in\mathbb{S}^{N-1}$, $|w(x,\nu_1,u)-w(x,\nu_2,u)|<C(|u|^q+1)|\nu_1-\nu_2|$. 
\end{itemize}
\end{assumption}

\begin{remark}\label{remarkMultiLinearSurface}
We have in mind the case where 
\begin{equation}
w(x,\nu,u)=\sum\limits_{i=0}^m a_i(\nu)[u,...,u],
\end{equation}
where for every $\nu\in \mathbb{S}^{N-1}$, $a_i(\nu)$ is an $i$-linear map on $V$, and $m<q$. That is, $w$ is of polynomial type in $u$ with degree $m$, less than $q$, with coefficients that depend on $\nu$. We will take $a_i$ to be continuous in $\nu$ in such cases, and as we will see in \Cref{propPolynomialSurface}, functions of this form admit a greatly simplified homogenised surface energy. In particular, the homogenised surface energy will also be of polynomial type with degree $m$.

Furthermore, the choice of $q<p^*$ in the growth bounds for $w$ is chosen as a technical assumption, as this implies that the trace operator from $W^{1,p}(\Omega_\e)$ is a compact operator into $L^q(\Gamma_\e)$. 

\end{remark}

\begin{definition}\label{defFunctional}
We define the functionals $\mathcal{F}_\e:W^{1,p}(\Omega_\e,V)\to\mathbb{R}\cup\{+\infty\}$ by 
\begin{equation}
\mathcal{F}_\e(u)=\int_{\Omega_\e}F(\nabla u(x),u(x),x)\,dx +\int_{\Gamma_\e}w(P_\Gamma x,\nu_\e(x),u(x))\,dS(x). 
\end{equation}
\end{definition}

\subsubsection{Function space assumptions}

\begin{definition}
Let $U\subset \mathbb{R}^N$ be a measurable set. Let $V$ be a finite dimensional vector space, and $1<p<\infty$. We define the extension by zero operator $E_U:\bigcup\limits_{U'\subset U} L^p(U',V)\to L^p(U,V)$, where for $u\in L^p(U',W)$, $E_Uu=u$ on $U'$, $E_Uu=0$ otherwise. Similarly, we define the restriction operator $R_U:\bigcup\limits_{U\subset U'} L^p(U',V)\to L^p(U,V)$ by $R_Uu=u|_U$.
\end{definition}

\begin{definition}[Rugose convergence of functions]\label{defRugoseConv}
 Let $\Omega,\Omega_\e$ satisfy Assumptions \ref{assumpOmega} and \ref{assumpOmegaEps}, and $V$ be a finite dimensional vector space. Let $u_\e \in W^{1,p}(\Omega_\e,V)$ for every $\e>0$. Let $D$ be a bounded set, such that $\Omega_\e\subset\subset D$ for every $\e>0$. Then we say that $u_\e$ converges to $u\in W^{1,p}(\Omega,V)$ in the rugose sense, denoted ${u_\e\rto u}$, if the following hold:
\begin{itemize}
\item For every $U\subset\subset \Omega$, $R_Uu_\e\overset{W^{1,p}(U,V)}{\rightharpoonup} R_Uu$.
\item $E_{D}u_\e\overset{L^p(D,V)}{\rightharpoonup} E_{D}u$. 
\item $E_{D}\nabla u_\e\overset{L^p(D,V)}{\rightharpoonup} E_{D}\nabla u$.
\end{itemize}

\end{definition}
\begin{remark}
We note that \Cref{defRugoseConv} is independent of the choice of $D\subset\mathbb{R}^N$, provided that $\Omega_\e\subset\subset D$ for all $\e>0$.
\end{remark}

\begin{theorem}\label{theoremGammaConvergence}
Let $\Omega_\e$, for $\e_0>\e>0$, $\Omega$ denote sets satisfying \Cref{assumpOmega} and \Cref{assumpOmegaEps}, the interior energy $F$ satisfy \Cref{assumpInterior}, and the surface energy $w$ satisfy \Cref{assumpSurface1}. Then there exists an effective surface energy density function $w_{h}:\Gamma\times V\to\mathbb{R}$ so that for every $u\in L^q(\Gamma)$, $\lim\limits_{\e\to 0}w(\cdot,\tilde{\nu}_\e,u)\gamma_\e\overset{L^1(\Gamma)}{\rightharpoonup} w_{h}(\cdot,u)$, and $\mf_\e\overset{\Gamma}{\to} \mf$ with respect to $W^{1,p}$-rugose convergence, where $\mf:W^{1,p}(\Omega,V)\to[0,\infty)$ is given by 
\begin{equation}
\mf(u)=\int_\Omega F(\nabla u(x),u(x),x)\,dx +\int_{\Gamma}w_{h}(x,u(x))\,dS(x).
\end{equation} 
The homogenised surface energy $w_h:\Gamma\times V\to\mathbb{R}$ is given by 
\begin{equation}
w_h(x,u)=\int_{\nu+T_x\Gamma}w\left(x,\frac{1}{|v|}v,u\right)|v|\,d\mu_x(v),
\end{equation}
where $(\mu_x)_{x\in\Gamma}$ is the Young measure generated by $\hat{D}\varphi_\e:\Gamma\to \mathbb{R}^N$ as given in \Cref{assumpSurface2}.
Furthermore the energies $\mf_\e$ are equicoercive with respect to $W^{1,p}$-rugose convergence, and if $u_\e$ are minimisers of $\mf_\e$ for every $\e_0>\e>0$, there exists a subsequence $\e_j\to 0$ and minimiser $u$ of $\mf$ such that $u_{\e_j}\rto u$.  
\end{theorem}

\subsection{Proof of $\Gamma$-convergence}\label{subsecGammaConv}

\subsubsection{Compactness}

\begin{proposition}\label{propCompactness}
 Let $\Omega,\Omega_\e$ satisfy Assumptions \ref{assumpOmega} and \ref{assumpOmegaEps}, $w$ satisfy Assumptions \ref{assumpSurface1} and \ref{assumpSurface2}, and $F$ satisfy \Cref{assumpInterior}. Let ${u_\e\in W^{1,p}(\Omega_\e,V)}$ for every $\e>0$ be such that 
\begin{equation}
\mf_\e(u_\e)<M. 
\end{equation}
Then there exists a subsequence $\epsilon_j\to 0$ and $u\in W^{1,p}(\Omega,V)$ such that $u_j:=u_{\e_j}\rto u$.
\end{proposition}
\begin{proof}
Take $D$ to be a bounded, open set such that $\Omega_\e\subset\subset D$ for all $\e>0$. Using the non-negativity of the surface term, we immediately obtain the estimate 
\begin{equation}\begin{split}
M\geq \mf_\e(u_\e)\geq &\int_{\Omega_\e}F(\nabla u_\e(x),u_\e(x),x)\,dx\\
\geq &\int_{\Omega_\e}C\left(|\nabla u_\e(x)|^p+|u_\e(x)|^p-1\right)\,dx \\
=&-C|\Omega_\e|+C\int_{D}|E_D\nabla u_\e(x)|^p+|E_Du_\e(x)|^p\,dx .
\end{split}
\end{equation}
As $|\Omega_\e|$ is uniformly bounded, and thus both $E_Du_\e$, $E_D\nabla u_\e$ are uniformly bounded in $L^p(D)$, we may thus extract a subsequence such that $E_Du_j\rightharpoonup u_0$ and $E_D\nabla u_j\overset{L^p}{\rightharpoonup} A$ for some $u_0\in L^p(D,V)$ and $A\in L^p(D,V\times\mathbb{R}^N)$. Furthermore, for any $U\subset D$ with $U\subset\subset\mathbb{R}^N\setminus \Omega$, $U\cap \Omega_\e=\emptyset$ for small enough $\e$, therefore $E_Du_j,E_D\nabla u_j$ will both be zero on $U$ for sufficiently small epsilon. This implies $\text{supp}(u_0)\cap U=\text{supp}(A)\cap U = \emptyset$ for all such $U$. Taking the union of all such $U$ implies that  $\text{supp}(u_0),\text{supp}(A)\subset \bar{\Omega}$. 

We now define $u=R_\Omega u_0$. We aim to show that $R_\Omega A = \nabla u$. To do so we take any $\eta \in \mathcal{D}(\Omega,V)$. In particular, $\text{supp}(\eta )\subset\subset\Omega$, which in turn implies that for all sufficiently small $\e>0$, $\text{supp}(\eta )\subset \Omega_\e$. Therefore we know that 
\begin{equation}
\begin{split}
0=\int_{\Omega_\e} u_{\e_j}\cdot\nabla \eta +\nabla u_{\e_j}\cdot\eta  \,dx=&\int_{\text{supp}(\eta )} u_{\e_j}\cdot\nabla \eta +\nabla u_{\e_j}\cdot\eta \,dx\\
\to & \int_{\text{supp}(\eta )} u_0\cdot\nabla \eta +A\cdot\eta \,dx\\
=& \int_{\Omega} u\cdot\nabla \eta +R_\Omega A\cdot\eta \,dx\\
\end{split}
\end{equation}
Therefore $R_\Omega A = \nabla u$ in the sense of distributions, however as $A$ is in $L^p(\Omega,V)$, this implies that $u\in W^{1,p}(\Omega,V)$. 

It only remains to show that $R_U u_j\overset{W^{1,p}}{\rightharpoonup}R_U u$ on $U\subset\subset \Omega$. However, as $||R_Uu_j||_{W{1,p}}$ is uniformly bounded and $R_Uu_j\overset{L^p}{\rightharpoonup}R_Uu,\, R_U\nabla u_j\overset{L^p}{\rightharpoonup}R_U\nabla u$, as all sub-sequences must have a $W^{1,p}$-weakly converging sub-sub-sequence, whose limit may only be $R_Uu$. 
\end{proof}

\subsubsection{Surface terms}

We first provide some preliminary results on the family of Young measures $(\mu_x)_{x\in\Gamma}$ generated by $\hat{D}\varphi_\e$. In particular, we highlight that the assumption that $\hat{D}\varphi_\e$ generates a family of Young measures may always be taken up to a subsequence, provided $\varphi_\e$ satisfies the appropriate Lipschitz-type boudns. 

\begin{proposition}\label{propYoungMeasure}
Let $\Omega,\Omega_\e$ satisfy Assumptions \ref{assumpOmega} and \ref{assumpOmegaEps}. Then $\mu_x(V)=1$, and $\text{supp}(\mu_x)\subset \nu(x)+T_x\Gamma$. 

Furthermore, we have that for any sequence $\varphi_\e:\Gamma\to\mathbb{R}$ satisfying the bounds $||\varphi_\e||_\infty<C\e$, $||D \varphi_\e||_\infty<C$ , we may always extract a subsequence $\e_j\to 0$ such that $\hat{D}\varphi_{\e_j}$ generates a family of Young measures. 
\end{proposition}
\begin{proof}

First we show the compactness statement. If $\Gamma$ were a measurable, positive measure subset of $\mathbb{R}^{N-1}$, then as $|\hat{D}\varphi_\e(x)|$ is uniformly bounded, the result would be a direct consequence of \cite[Theorem 1]{ball1989version}. It is however straightforward to generalise this to the case where $\Gamma$ is a $C^2$ manifold, by taking a partition of unity $(\xi_i)_{i=1}^M$, so that $\xi_i\geq 0$ for all $i$, $\sum\limits_{i=1}^M\xi_i=1$, and there exists domains $(U_i)_{i=1}^M$ with $\bigcup\limits_{i=1}^MU_i=\Gamma$, $\text{supp}(\xi_i)\subset U_i$. Take $U_i$ such that there exists a $C^1$ map $\eta_i:\tilde{U}_i\to U_i$, where $\tilde{U}_i\subset\mathbb{R}^{N-1}$. Then, given any continuous function $F:\mathbb{R}^N\to\mathbb{R}$, we have that 
\begin{equation}\label{eqDecompYoung}
\int_\Gamma F(x,\hat{D}\varphi_\e(x))\,dx = \sum\limits_{i=1}^M\int_{\tilde{U}_i} F(\eta_i(x'),\hat{D}\varphi_\e(\eta_i(x')))\xi_i(\eta_i(x'))J_i(x')\,dx',
\end{equation}
where $J_i$ are the corresponding volume elements. From this, it is immediate that the results for domains in $\mathbb{R}^{N-1}$ may be lifted from the result of Ball applied to the functions $f^i_\e=\hat{D}\varphi_\e\circ\eta_i$, and by inverting the integral decomposition \eqref{eqDecompYoung}, the compactness result holds. Explicitly, as a subsequence of $f^i_\e$, $f^i_{\e_j}$, generates Young measures $\mu^i_{\eta^i(x)}=\mu^i_{x'}$ for $x\in U^i$, we have that 
\begin{equation}\begin{split}
&\lim\limits_{j\to\infty}\sum\limits_{i=1}^M\int_{\tilde{U}_i} F(\eta_i(x'),\hat{D}\varphi_\e(\eta_i(x')))\xi_i(\eta_i(x'))J_i(x')\,dx'\\
&=\sum\limits_{i=1}^M\int_{\tilde{U}_i} \int_{\mathbb{R}^N} F(\eta_i(x'),v)\xi_i(x')J_i(x')\,d\mu^i_{x'}\,dx'\\
=&\sum\limits_{i=1}^M\int_{\tilde{U}_i} \int_{\mathbb{R}^N} F(\eta_i(x'),v)\,d\mu^i_{x'}\xi_i(\eta_i(x'))J_i(x')\,dx'\\
=&\sum\limits_{i=1}^M\int_{U_i} \int_{\mathbb{R}^N} F(x,v)\,d\mu^i_x\xi_i(x)\,dx\\
=& \int_{\Gamma}\int_{\mathbb{R}^N} F(x,v)\sum\limits_{i=1}^M\xi_i(x)\,d\mu^i_x(v)\,dx.
\end{split}
\end{equation}
Thus we see that the Young measure $\mu_x=\sum\limits_{i=1}^M\xi_i(x)\,d\mu^i_x$ is the family of Young measures corresponding to $\hat{D}\varphi_\e$. 

To show that $\mu_x({\mathbb{R}^N})=1$, via the decomposition \eqref{eqDecompYoung}, again we apply the result of Ball on each domain, which implies that $\mu^i_x({\mathbb{R}^N})=1$ for almost every $x$ as the integrand is uniformly bounded from above, and since $\sum\limits_{i=1}^M\xi_i=1$, we have that $\mu_x({\mathbb{R}^N})=1$ almost everywhere. 

Finally, to show that $\text{supp}(\mu_x)\subset\nu(x)+T_x\Gamma$ almost everywhere, we consider the function 
\begin{equation}
F(x,v)=\left|(I-\nu(x)\otimes \nu(x))(v-\nu(x))\right|^2.
\end{equation}
We note that this is a $C^1$, non-negative function, which is zero if and only if $v\in \nu(x)+T_x\Gamma$. As $\hat{D}\varphi_\e$ is uniformly bounded, it is thus immediate that $F(x,\hat{D}\varphi_\e(x))$ is a uniformly integrable sequence of functions, and therefore as $\hat{D}\varphi_\e(x)\in \nu(x)+T_x\Gamma$,
\begin{equation}
0=\int_\Gamma F(x,\hat{D}\varphi_\e(x))\,dx\to \int_\Gamma\int_V F(x,v)\,d\mu_x(v)\,dx.
\end{equation}
Therefore, we must have that $\int_V F(x,v)\,d\mu_x(v)=0$ for almost every $x$, which thus implies that $\text{supp}(\mu_x)\subset \nu(x)+T_x\Gamma$. 

\end{proof}

We now provide a lemma that allows us to greatly simplify the surface terms

\begin{lemma}\label{corollarySimplerSurface}
Let the surface energy $w$ and $\varphi_\e$ satisfy \Cref{assumpSurface1}, \Cref{assumpSurface2}, and $1\leq q\leq p^*$ be as in \Cref{assumpSurface1}. Then for every $u\in L^q(\Gamma,V)$, there exists $c>0$ such that for $\e$ sufficiently small,
\begin{equation}
\int_\Gamma\left|w(x,\nu_\e(x),u(x))\gamma_\e(x)-w\left(x,\frac{1}{|\hat{D}\varphi_\e(x)|}{\hat{D}\varphi_\e(x)},u(x)\right)|\hat{D}\varphi_\e(x)|\right|\,dS(x)<c\e.
\end{equation}
\end{lemma}

\begin{proof}
We defer the proof to \Cref{appSurface}
\end{proof}

\begin{proposition}\label{propWeakSurfaceConv}
 Let $\Omega,\Omega_\e$, $\varphi_\e$ satisfy Assumptions \ref{assumpOmega} and \ref{assumpOmegaEps}, and $w_h$ satisfy Assumptions \ref{assumpSurface2} and \ref{assumpSurface1}, with $p^*>q>1$ as in \Cref{assumpSurface1}.
There exists a homogenised surface energy $w_h:\Gamma\times V\to\mathbb{R}$ given by 
\begin{equation}\label{eqHomoSurface}
w_h(x,u)=\int_{\nu(x)+T_x\Gamma}w\left(x,\frac{1}{|v|}v,u\right)|v|\,d\mu_x(v),
\end{equation}
where $(\mu_x)_{x\in\Gamma}$ is the Young measure generated by $\hat{D}\varphi_\e$, as in \Cref{assumpSurface2}, such that for every $u\in L^q(\Gamma,V)$, 
\begin{equation}
\int_\Gamma w(x,\tilde{\nu}_\e(x),u(x))\gamma_\e(x)\,dx\to \int_\Gamma w_h(x,u(x))\,dx,
\end{equation}
 with $\tilde{\nu}_\e,\gamma_\e$ as defined in \Cref{assumpOmegaEps}.
\end{proposition}
\begin{proof}
By \Cref{corollarySimplerSurface}, we see that it suffices to prove that 
\begin{equation}
\begin{split}
&\int_\Gamma w\left(x,\frac{1}{|\hat{D}\varphi_\e(x)|}\hat{D}\varphi_\e(x),u(x)\right)|\hat{D}\varphi_\e(x)|\,dS(x)\\
\to & \int_\Gamma\int_{\nu(x)+T_x\Gamma}w\left(x,\frac{1}{|v|}v,u(x)\right)|v|\,d\mu_x(v)\,dS(x).
\end{split} 
\end{equation}
For fixed $u\in L^q$, we define the auxiliary function $\tilde{w}:\Gamma\times\mathbb{R}^N\to\mathbb{R}$ by
\begin{equation}
\tilde{w}(x,v)=w\left(x,\frac{1}{|v|}v,u(x)\right)|v|\eta(v),
\end{equation}
where $\eta$ is a smooth cutoff function, zero in a neighbourhood of the origin, and $\eta(v)=1$ for $|v|>\frac{1}{2}$. This defines a Carath\'eodory function for $u\in L^q$, and as we only consider $v=\hat{D}\varphi_\e$, which satisfies $|v|\geq 1$, we see that 
\begin{equation}
\tilde{w}(x,\hat{D}\varphi_\e(x))=w\left(x,\frac{1}{|\hat{D}\varphi_\e(x)|}\hat{D}\varphi_\e(x),u(x)\right)|\hat{D}\varphi_\e(x)|
\end{equation}
almost everywhere. We also see that for every $u$, $\tilde{w}(\cdot,\hat{D}\varphi_\e)$ is uniformly integrable via the definition of $\tilde{w}$ and \Cref{assumpSurface1}, as 
\begin{equation}
|\tilde{w}(x,\hat{D}\varphi_\e)|\leq C|\hat{D}\varphi_\e|(1+|u|^q),
\end{equation}
using that $|u|^q$ is integrable and $\hat{D}\varphi_\e$ admits uniform $L^\infty$ bounds. In particular, this implies that 
\begin{equation}
\int_\Gamma\tilde{w}(x,\hat{D}\varphi_\e(x))\,dS(x)\to\int_\Gamma \int_{\nu(x)+T_x\Gamma}\tilde{w}(x,v)\,d\mu_x(v)\,dS(x),
\end{equation}
by \cite{ball1989version}, at which point re-writing $\tilde{w}$ in terms of $w$ and $u$ gives the necessary result. 
\end{proof}

\begin{corollary}\label{corollaryHomoBounds}
 The function $w_h$ as given by \eqref{eqHomoSurface} satisfies the same growth conditions in $u$ as $w$ in \Cref{assumpSurface1}. That is, 
\begin{itemize}
\item There exists $C>0$ such that for a.e. $x\in\Gamma$ and $u\in V$, $|w_h(x,u)|\leq C(|u|^q+1)$. 
\item $|w_h(x,u_1)-w_h(x,u_2)|\leq C |u_1-u_2|\left(|u_1|^{q-1}+|u_2|^{q-1}+1\right)$ for a.e. $x\in\Gamma$ and all $u_1,u_2\in V$. 
\end{itemize}
\end{corollary}
\begin{proof}
The result is a straightforward consequence of the integral representation of $w_h$ in \eqref{eqHomoSurface}, using that $\text{supp}(\mu_x)\subset\nu(x)+T_x\Gamma$ and $\mu_x(\mathbb{R}^N)=1$, as proven in \Cref{propYoungMeasure}

\end{proof}

In the case of simpler surface energies, we are able to provide $w_h$ in a more explicit sense. 

\begin{proposition}\label{propPolynomialSurface}
 Let $p^*>q>1$, as in \Cref{assumpSurface1}. Let $w$ be of the form 
\begin{equation}
w(x,\nu,u)=\sum\limits_{i=0}^m a_i(\nu)[u,u,...,u],
\end{equation}
where $a_i(\nu)$ is an $i$-linear map on a  finite dimensional vector space $V$, identified with a member of $(V^*)^i$, $m$ is an integer satisfying $m<q$, and $a_i$ is continuous in $\nu$. Then 
\begin{equation}
a_i(\nu_\e)\gamma_\e\overset{*}{\rightharpoonup}A_i=\int_{\nu+T_x\Gamma}a_i(v)|v|\,d\mu_x(v),
\end{equation}
where the convergence, explicitly, is weak-* $L^\infty$ convergence, and
\begin{equation}
w_h(x,u)=\sum\limits_{i=0}^m A_i(x)[u,u,...,u].
\end{equation}
\end{proposition}
\begin{proof}
First we note that as $a_i$ is continuous and $\nu_\e$ is bounded, $a_i(\nu_\e)\gamma_\e$ is certainly bounded in $L^\infty$. Furthermore, using \Cref{corollarySimplerSurface}, we may estimate uniformly in $\e$ that $a_i(\nu_\e)\gamma_\e=a_i\left(\frac{1}{|\hat{D}\varphi_\e|}\hat{D}\varphi_\e\right)|\hat{D}\varphi_\e|+O(\e)$. Therefore, by \cite{ball1989version}, we must have that 
\begin{equation}
a_i(\nu_\e)\gamma_\e\overset{L^1}{\rightharpoonup}\int_{\nu(x)+T_x\Gamma}a_i(v)|v|\,d\mu_x(v)=A_i.
\end{equation}  We note that $a_i(\nu_\e)$ also admits uniform $L^\infty$ bounds, and thus is pre-compact with weak-* $L^\infty$ convergence. As weak-* $L^\infty$ limits and weak-$L^1$ limits must be equal if both exist, the pre-compactness with respect to weak-* $L^\infty$ convergence and weak $L^1$ convergence of $a_i(\nu_\e)$ imply that $a_i(\nu_\e)\overset{*L^\infty}{\rightharpoonup}A_i$ also.
\end{proof}

In order to understand how the surface energy behaves in the limit, our aim will be to define a ``trace" of $u_\e\in W^{1,p}(\Omega_\e,V)$ onto $\Gamma$, which is continuous under rugose convergence. That is, we want a way of understanding how $u_\e|_{\Gamma_\e}$ ``approaches" $u|_\Gamma$. To do so, we will define a bi-Lipschitz change of variables $\Phi_\e:\Omega\to\Omega_\e$, which is such that if $u_\e\rto u$, then $u_\e\circ \Phi_\e\overset{W^{1,p}}{\rightharpoonup} u$. We may then understand the trace of $u_\e\circ \Phi_\e$ onto $\Gamma$.

\begin{definition}\label{defProjection}
 We define the transformation $\Phi_\e:\Omega\to \Omega_\e$ as follows. Let $\delta>0$ be fixed, and sufficiently small so that the projection $P_\Gamma$, as defined in \Cref{defProjection}, is well-defined and $C^1$. We then define $\Omega^\delta = \{x\in\Omega:d(x,\partial\Omega)>\delta\}$, and consider only $\e$ sufficiently small so that $||\varphi_\e||_\infty<\frac{1}{2}\delta$, so that the following are well defined. If $x\in\Omega^\delta$, $\Phi_\e x:=x$. If  $x\in\Omega\setminus\Omega^\delta$, then we may write that $x=P_{\Gamma}x+\delta(t_x-1)\nu(P_{\Gamma}x)$ for some $t_x\in [0,1]$, which is readily evaluated as 
\begin{equation}
t_x = 1-\frac{|x-P_\Gamma x|}{\delta}.
\end{equation} Then we define 
\begin{equation}\label{eqDefBilip}
\Phi_\e x =P_{\Gamma}x+\left(\left(\varphi_\e(P_\Gamma x)+\delta\right)t_x-\delta\right)\nu(P_{\Gamma}x)
\end{equation}
\end{definition}
We note that the transformation restricted to lines of the form $s\mapsto x+s\nu(x)$ for $x\in\Gamma$ and $s$ sufficiently small is merely a linear rescaling, so that $\Phi_\e x = x$ for $x\in \partial\Omega^\delta$, and $\Phi_\e x = x+\varphi_\e(x)\nu(x)$ for $x\in \Gamma$.

\begin{figure}[h]\begin{center}
\begin{subfigure}[t]{0.4\textwidth}
\begin{center}
\includegraphics[width=\textwidth]{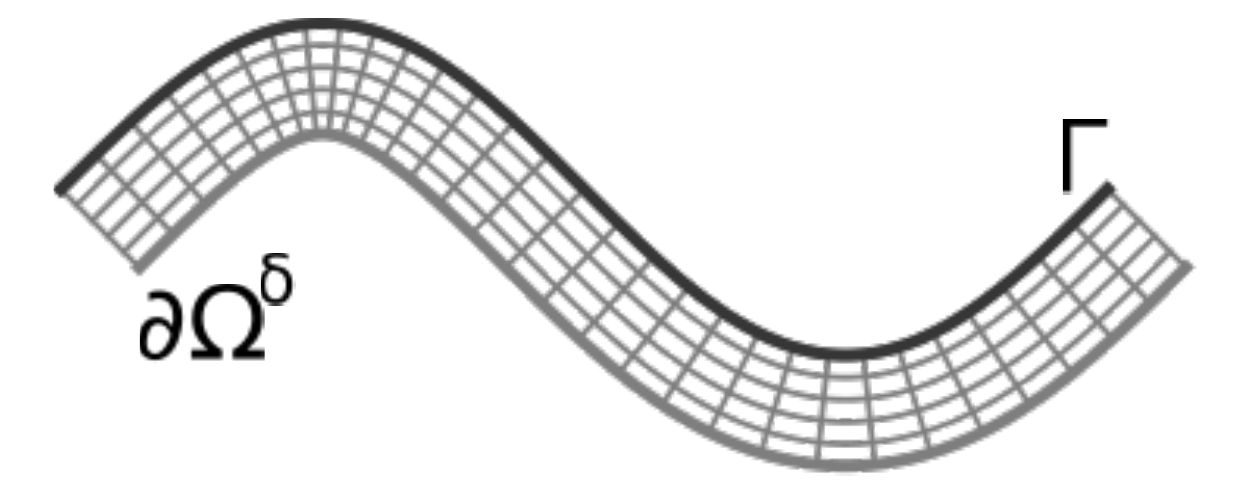}
\caption{
$\Omega\setminus\Omega^\delta$ with corresponding $(x,t)$ coordinate lines. $\Gamma$ is the  upper thick line, and $\partial\Omega^\delta$ the lower thick line.
}
\end{center}
\end{subfigure}
\hspace{0.025\textwidth}
\begin{subfigure}[t]{0.4\textwidth}
\begin{center}
\includegraphics[width=\textwidth]{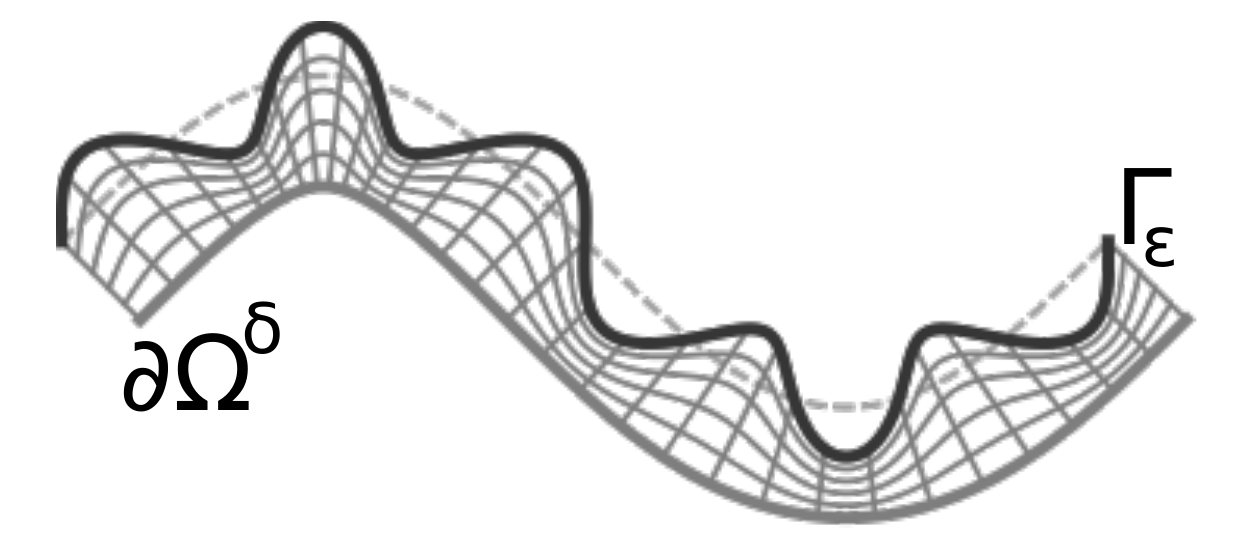}
\caption{$\Phi_\e(\Omega\setminus\Omega^\delta)$ with corresponding $(x,t)$ coordinate lines post-deformation. $\Gamma$ is visible as a dashed line, $\Gamma_\e$ is the upper thick line, and $\partial\Omega^\delta$ the lower thick line.
}
\end{center}
\end{subfigure}\captionsetup{width=0.85\linewidth}
\caption
 {The $(x,t)$ coordinate system before and after application of $\Phi_\e$ on $\Omega\setminus\Omega^\delta$. }
\end{center}
\end{figure}

\begin{proposition}\label{propBilipInverse}
 For $\e>0$ sufficiently small and $\Phi_\e$ defined as in \eqref{eqDefBilip}, $\Phi_\e$ is invertible, and $\Phi_\e^{-1}:\Omega_\e\to\Omega$ is given by $\Phi_\e^{-1}x=x$ on $\Omega^\delta$, and 
\begin{equation}\label{eqInverseBilip}
\Phi_\e^{-1}x =  P_{\Gamma}x+\delta\frac{s_x-\varphi_\e(P_\Gamma x)-\delta}{\delta+\varphi_\e(P_\Gamma x)}\nu(P_{\Gamma}x),
\end{equation}
where 
\begin{equation}\label{eqS}
s_x= |x-P_\Gamma x-\delta\nu(P_\Gamma x)|.
\end{equation}
In particular, $\Phi_\e$ defines a family of uniformly bi-Lipschitz maps between $\Omega$ and $\Omega_\e$. That is, there exists a constant $L>0$ such that for all $\e$, $x,y\in\Omega$
\begin{equation}
L^{-1}|x-y|<|\Phi_\e(x)-\Phi_\e(y)|<L|x-y|.
\end{equation}
\end{proposition}
\begin{proof}
To prove that the expression is the inverse  on $\Omega_\e\setminus\Omega^\delta$, first we notice that for such $x$, $P_{\Gamma}\Phi_\e(x)=P_{\Gamma}x$, so that $x=P_\Gamma(x)+\delta(s_x-1)\nu(P_\Gamma(x))$, where it is straightforward to verify that $s_x$ is as in \eqref{eqS}. Thus it remains to find $y\in\Omega\setminus\Omega^\delta$ so that

\begin{equation}
P_\Gamma x+\delta(s_x-1)\nu(P_\Gamma x)=\Phi_\e y = P_\Gamma y+ \left(\left(\varphi_\e(P_\Gamma y)+\delta\right)t_y-\delta\right)\nu(P_{\Gamma}y),
\end{equation}
 where $P_\Gamma y = P_\Gamma x$. By subtracting $P_\Gamma x$ from both sides and taking norms, this reduces to the simple scalar equation
\begin{equation}
\delta(s_x-1)=(\varphi_\e(P_\Gamma x)+\delta) t_y-\delta,
\end{equation}
which is easily inverted to give 
\begin{equation}
t_y= \frac{\delta s_x}{\varphi_\e(P_\Gamma x)+\delta},
\end{equation}

which is readily substituted into $x=P_\Gamma x + \delta (t_x-1)\nu(P_\Gamma x)
$ to give the form in \eqref{eqInverseBilip}.

 To observe that these maps are bi-Lipschitz, it suffices to notice that both $\Phi_\e,\Phi_\e^{-1}$ are continuous functions, and their derivatives can be uniformly bounded from above via their algebraic expressions, using that $P_\Gamma$, $\varphi_\e$, $t_x,s_x$ are all Lipschitz, with Lipschitz bounds that are independent of $\e$, and the algebraic expressions for $\Phi_\e,\Phi_\e^{-1}$ are non-singular on their domains of definition.

\end{proof}

\begin{lemma}\label{lemmaPhiConverge}
$\Phi_\e$, as defined in \eqref{eqDefBilip}, converges uniformly to the identity. Furthermore, for every $U\subset\subset\Omega$, there exists $\e_0>0$ so that for all $\e_0>\e>0$, $\Phi_\e(U)\subset\subset\Omega$.
\end{lemma}
\begin{proof}
Clearly on $\Omega^\delta$ there is nothing to prove as $\Phi_\e|_{\Omega^\delta}$ is the identity, so we consider the case with $x\in \Omega\setminus \Omega^\delta$. In this case we have that for $x=P_\Gamma x + \delta(t_x-1)\nu(P_\Gamma x)$,
\begin{equation}
\begin{split}
|x-\Phi_\e(x)|=&\left|P_{\Gamma}x+\delta(t_x-1)\nu(P_{\Gamma}x)-P_{\Gamma}x-\left(\left(\varphi_\e(P_\Gamma x)+\delta\right)t_x-\delta\right)\nu(P_{\Gamma}x)\right|\\
=&\left|\delta(t_x-1)-\left(\left(\varphi_\e(P_\Gamma x)+\delta\right)t_x-\delta\right)\right|\\
=&\left|\varphi_\e(P_\Gamma x)t_x\right|\\
\end{split}
\end{equation}
which, as $t_x\in[0,1]$ and $||\varphi_\e||_\infty\to 0$, converges to zero uniformly in $x$ as $\e\to 0$.

It is then immediate that for $U\subset\subset\Omega$ we must have that $\Phi_\e\subset\Omega$ for sufficiently small $\e$, since for any $\eta$ sufficiently small so that $U+B_\eta\subset\subset \Omega$, we can take $\e$ sufficiently small so that $|\Phi_\e(x)-x|<\eta$, and thus $\Phi_\e(U)\subset U+B_\eta$. 
\end{proof}

\begin{lemma}\label{lemmaChangeOfVars}
 Let $\Omega,\Omega_\e$ satisfy Assumptions \ref{assumpOmega},  \ref{assumpOmegaEps}. Let $1<p<\infty$, $u_\e\in W^{1,p}(\Omega_\e,V)$ be such that $u_\e\rto u_0\in W^{1,p}(\Omega,V)$. Then $u_\e\circ \Phi_\e\in W^{1,p}(\Omega,V)$ satisfies $u_\e\circ \Phi_\e\overset{W^{1,p}(\Omega)}{\rightharpoonup} u_0$ also. 
\end{lemma}
\begin{proof}

Our first observation is that it suffices to show that for any subsequence $\e_j\to 0$, there exists a further subsequence $\e_{j_k}\to 0$ such that $u_{\e_{j_k}}\circ\Phi_{\e_{j_k}}\overset{W^{1,p}}{\rightharpoonup}u_0$.  If this were true, then as all subsequences of $u_\e\circ\Phi_\e$ have a $W^{1,p}$-weakly convergening subsequence, all admitting the same limit, we must have that $u_\e\circ\Phi_\e\overset{W^{1,p}}{\rightharpoonup}u_0$.Then we observe that as $\Phi_\e$ is bi-Lipschitz with bi-Lipschitz constants bounded away from $0$ and $\infty$ independently of $\e$, and $||u_\e||_{W^{1,p}(\Omega_\e,V)}$ is bounded, we therefore have that $u_\e\circ\Phi_\e$ is bounded in $W^{1,p}(\Omega,V)$ \cite[Theorem 2.2.2.]{ziemer1989sobolev}. Therefore, given any subsequence $\e_j\to 0$, we may take a subsequence $\e_{j_k}\to 0$ such that $u_{\e_{j_k}}\circ\Phi_{\e_{j_k}}$ converges weakly in $W^{1,p}(\Omega,V)$, and strongly in $L^p(\Omega,V)$ to some $\tilde{u}_0$. As such, it remains only to show that $\tilde{u}_0=u_0$. 

 For brevity of notation, we denote $u_k=u_{\e_{j_k}},\Phi_k=\Phi_{\e_{j_k}}$. To show that $\tilde{u}_0=u_0$, it suffices to show that for any $U\subset\subset\Omega$, 
\begin{equation}
\lim\limits_{j\to\infty } \int_U |u_0(x)-u_k\circ\Phi_k(x)|^p\,dx = 0.
\end{equation}
First, we let $\hat{u}_0$ denote a Lipschitz approximant of $u_0$ on an open set $U'$ which satisfies 
\begin{equation}
\bigcap\limits_{\e<\e_0}\Phi_\e(U)\subset\subset U'\subset\subset \Omega.
\end{equation}
 An $\e_0>0$ that permits this chain of inclusions is guaranteed by \Cref{lemmaPhiConverge}. Then we estimate 
\begin{equation}
||u_0-u_k\circ\Phi_k||_{L^p(U)}\leq  ||u_0-\hat{u}_0||_{L^p(U)}+||\hat{u}_0-\hat{u}_0\circ\Phi_k||_{L^p(U)}+||\hat{u}_0\circ\Phi_k-u_k\circ\Phi_k||_{L^p(U)}.
\end{equation}
Using \Cref{lemmaPhiConverge}, as $\hat{u}_0$ is Lipschitz on $U'$, it is clear that $\lim\limits_{j\to 0}||\hat{u}_0-\hat{u}_0\circ\Phi_k||_{L^p(U,V)}$. Then, as $\Phi_k$ is uniformly bi-Lipschitz, we must have that $||\hat{u}_0\circ\Phi_k-u_k\circ\Phi_k||_{L^p(U,V)}\leq C ||\hat{u}_0-u_k||_{L^p(U',V)}$ for some $C>0$. Thus we may estimate 
\begin{equation}
\begin{split}
\limsup\limits_{\j\to \infty}||u_0-u_k\circ\Phi_k||_{L^p(U,V)}\leq &||u_0-\hat{u}_0||_{L^p(U,V)}+C\limsup\limits_{\j\to \infty}||\hat{u}_0-u_k||_{L^p(U',V)}\\
= & ||u_0-\hat{u}_0||_{L^p(U,V)}+C\limsup\limits_{\j\to \infty}||\hat{u}_0-u_0||_{L^p(U',V)},
\end{split}
\end{equation}
as, by assumption, $u_k\rto u$ and thus $u_k\overset{L^p(U',V)}{\to}u_0$. Finally, we note that as $\hat{u}_0$ was an arbitrary approximation, by taking $\hat{u}_0$ to approximate $u_0$ in $L^p(U,V)$, which is permitted by the density of Lipschitz functions in $L^p$, we obtain that $\limsup\limits_{\j\to \infty}||u_0-u_k\circ\Phi_k||_{L^p(U,V)}=0$. As $U$ was arbitrary, $u_0=\tilde{u}_0$.

\end{proof}

\begin{proposition}\label{propSurfaceContinuous}
Let $u_\e\in W^{1,p}(\Omega_\e,V)$ such that $u_\e\rto u$. Then 
\begin{equation}
\lim\limits_{\e\to 0}\int_{\Gamma_\e} w(P_\Gamma x,\nu_\e(x),u_\e(x))\,dS(x)  =  \int_\Gamma w_h(x,u(x))\,dS(x)
\end{equation}
\end{proposition}
\begin{proof}
From the \Cref{lemmaChangeOfVars}, we must have that $u_\e\circ \Phi_{\e}\overset{W^{1,p}(\Omega)}{\rightharpoonup}u$. Then,  for $q$ as in \Cref{assumpSurface1}, $(u\circ \Phi_{\e})|_\Gamma\overset{L^q(\Gamma)}{\to}u|_\Gamma$ as the embedding $W^{1,p}(\Omega,V)\hookrightarrow L^q(\Gamma,V)$ is compact.

We may then perform a change of variables on the integral to give 
\begin{equation}
\int_{\Gamma_\e} w(P_\Gamma x,\nu_\e(x),u_\e(x))\,dS(x)=\int_{\Gamma} w(x,\tilde{\nu}_\e(x),u_\e\circ \Phi_\e(x))\gamma_\e(x)\,dS(x),
\end{equation}
which combined with \Cref{propWeakSurfaceConv} and the strong $L^q(\Gamma,V)$ convergence of $u_{\e}\circ \Phi_\e|_\Gamma$ implies that 
\begin{equation}
\lim\limits_{\e\to 0}\int_{\Gamma_{\e}} w(x,\nu_{\e}(x),u_{\e}(x))\,dS(x)=\int_\Gamma w_h(x,u(x))\,dS(x).
\end{equation}


\end{proof}

\subsubsection{Proof of \Cref{theoremGammaConvergence}}

\begin{definition}
Define $\mf:W^{1,p}(\Omega,V)\to \mathbb{R}\cup\{+\infty\}$ by 
\begin{equation}
\mf(u)=\int_\Omega F(\nabla u(x),u(x),x)\,dx + \int_{\Gamma}w_h(x,u(x))\,dS(x),
\end{equation}
where $F$ is as in \Cref{assumpInterior}, and $w_h$ as in \eqref{eqHomoSurface}.
\end{definition}

\begin{proposition}[Liminf]\label{propLiminf}
Let $u_\e\in W^{1,p}(\Omega_\e,V)$ for $\e>0$, such that $u_\e\rto u\in W^{1,p}(\Omega,V)$. Then 
\begin{equation}
\liminf\limits_{\e\to 0}\mf_\e(u_\e)\geq \mf(u).
\end{equation}
\end{proposition}
\begin{proof}
Take a subsequence $\e_j\to 0$, $u_j=u_{\e_j}$, such that $\liminf\limits_{\e\to 0}\mf_\e(u_\e)=\lim\limits_{j\to\infty}\mf_{\e_j}(u_j)$. If the limit is $+\infty$ the result is trivial, so assume otherwise. In this case, by \Cref{propCompactness}, we can take a further subsequence (without relabelling) such that $u_j\rto u$ for some $u\in W^{1,p}(\Omega,V)$. In particular the surface term, following \Cref{propSurfaceContinuous}, satisfies 
\begin{equation}
\lim\limits_{j\to\infty}\int_{\Gamma_{\e_j}}w(P_\Gamma x,\nu_{\e_j}(x),u_j(x))\,dS(x)=\int_{\Gamma}w_h(x,u(x))\,dS(x). 
\end{equation}
We thus consider the interior term. Take $U\subset\subset \Omega$. For $j$ sufficiently large, $ U \subset \Omega_{\e_j}$, therefore, by the non-negativity of $F$, we have that 
\begin{equation}
\liminf\limits_{j\to\infty}\int_{\Omega_{\e_j}}F(\nabla u_j(x),u_j(x),x)\,dx\geq \liminf\limits_{j\to\infty}\int_{U}F(\nabla u_j(x),u_j(x),x)\,dx\geq \int_{U}F(\nabla u(x),u(x),x)\,dx
\end{equation}
by standard semicontinuity results for integral functionals \cite{dacorogna2007dm}. As $U$ is arbitrary, and the right-hand side is increasing in $U$ by the non-negativity of $F$, we therefore have that 
\begin{equation}
\liminf\limits_{j\to\infty}\int_{\Omega_{\e_j}}F(\nabla u_j(x),u_j(x),x)\,dx\geq \int_{\Omega}F(\nabla u(x),u(x),x)\,dx,
\end{equation} 
completing the proof. 
\end{proof}

\begin{proposition}\label{propExtension}
Let $\Omega$ satisfy \Cref{assumpOmega}, $F$ satisfy \ref{assumpInterior}, $w$ satisfy \ref{assumpSurface2}. Then for $\delta>0$ as in \ref{defProjection}, and $\Omega'=\Omega+B_\delta =\{x+y:x\in\Omega,|y|<\delta\}$. Then, for every $u\in W^{1,p}(\Omega,V)$ such that $\mathcal{F}(u)<+\infty$, there exists an extension $Tu\in W^{1,p}(\Omega',V)$ such that $u'|_\Omega=u$, and 
\begin{equation}\label{eqFiniteEnergyExt}
\int_{\Omega'} F(\nabla (Tu)(x),Tu(x),x)\,dx<+\infty.
\end{equation}

\end{proposition}

\begin{proof}
The proof strategy is to show first that there is a bounded linear extension operator, defined by reflections, $T:W^{1,p}(\Omega,V)\to W^{1,p}(\Omega',V)$. Then, we will show that this reflection operator is such that if $\mathcal{F}(u)$ has finite energy, then \eqref{eqFiniteEnergyExt} holds.  

First, we consider the simpler case when $u\in C^1(\bar{\Omega},V)$, which we will later extend via density. We will exploit a reflection argument, of a similar flavour to that of \cite[Theorem 5.22]{adams2003sobolev}, albeit simplified for the case at hand. First we define the reflection, in analogy to the argument presented in \Cref{propBilipInverse}. For $x\in\Omega'\setminus\Omega$, we may write 
\begin{equation}
x=P_\Gamma x + t_x\delta\nu(P_\Gamma x),
\end{equation}
where $t_x\in[0,1]$ is given explicitly via $t_x=\frac{1}{\delta}|x-P_\Gamma x|$. We thus define 
\begin{equation}\label{eqDefReflec}
Tu(x)=u(x-t_x\delta\nu(P_\Gamma x))
\end{equation}
for $x\in \Omega'\setminus \Omega$. As $P_\Gamma x$ and $t_x$ are Lipschitz functions, and $\nu$ is $C^1$, we thus have that the map $x\mapsto x-t_x \delta \nu(P_\Gamma x)$ is a Lipschitz map, and in fact bi-Lipschitz, by noting that its inverse is given by the simple formula $x'\mapsto x'+|x'-P_\Gamma x'|\nu(P_\Gamma x')$. Thus we see that $Tu$ defined by $Tu(x)=u(x)$ on $\bar{\Omega}$, and by \eqref{eqDefReflec} on $\Omega'\setminus\Omega$ defines a Lipschitz function, as $u\in C^1(\bar{\Omega},V)$. Furthermore, as the reflection mapping $x\mapsto x-t_x\delta\nu(P_\Gamma x)$ is bi-Lipschitz as a map between $\Omega\setminus\Omega^\delta$ (as defined in \Cref{defProjection}) and $\Omega'\setminus\Omega$, we have that 
\begin{equation}
\begin{split}
\int_{\Omega'}|\nabla (Tu)(x)|p\,dx =& \int_{\Omega}|\nabla (Tu)(x)|^p\,dx +\int_{\Omega'\setminus\Omega}|\nabla (Tu)(x)|^p\,dx \\
\leq & \int_{\Omega}|\nabla (Tu)(x)|^p\,dx+C\int_{\Omega\setminus\Omega^\delta}|\nabla u(x)|^p\,dx\\
\leq & (C+1)\int_\Omega |\nabla u(x)|^p,dx,
\end{split}
\end{equation}
where $C$ is independent of $u$, and only depends on the bi-Lipschitz constants of the reflection map. Similarly, and more trivially, we have that 
\begin{equation}
\int_{\Omega'}|Tu(x)|^p\,dx \leq C\int_\Omega |u(x)|^p\,dx,
\end{equation}
so that $Tu\in W^{1,p}(\Omega',V)$, and $||Tu||_{W^{1,p}(\Omega',V)}\leq C ||u||_{W^{1,p}(\Omega,V)}$. Finally, we note that $T$ is a linear map, and by the previous argument it is continuous, therefore by the density of $C^1(\bar{\Omega},V)$ in $W^{1,p}(\Omega,V)$, we may extend $T$ continuously to a bounded linear extension map $T:W^{1,p}(\Omega,V)\to W^{1,p}(\Omega',V)$. 

It now remains to show that $Tu$ satisfies \eqref{eqFiniteEnergyExt} if $\mathcal{F}(u)<+\infty$. Again, we exploit that the mapping $x\mapsto x-t_x\delta\nu(P_\Gamma x)$ is bi-Lipschitz, in tandem with \Cref{assumpInterior}, to give that 

\begin{equation}
\begin{split}
&\int_{\Omega'} F(T\nabla u(x), Tu(x),x)\,dx \\
=&\int_{\Omega} F(\nabla u(x), u(x),x)\,dx+\int_{\Omega'\setminus \Omega} F(\nabla Tu(x), Tu(x),x)\,dx\\
\leq & \int_{\Omega} F(\nabla u(x), u(x),x)\,dx+C\int_{\Omega\setminus\Omega^\delta} F(\nabla u(x), u(x),x)+1\,dx\\
\leq &(C+1) \left(\int_{\Omega} F(\nabla u(x), u(x),x)\,dx+ |\Omega|\right)
\end{split}
\end{equation}
so that
\begin{equation}
\int_{\Omega'}F(\nabla (Tu)(x),Tu(x),x)\,dx<\infty.
\end{equation}
\end{proof}

\begin{proposition}[Limsup]\label{propLimsup}
Let $\Omega,\Omega_\e$ satisfy Assumptions \ref{assumpOmega}, \ref{assumpOmegaEps}, $F,w$ satisfy Assumptions \ref{assumpInterior}, \ref{assumpSurface2}, and \ref{assumpSurface1}. Let $u\in W^{1,p}(\Omega,V)$. Then there exists a sequence $u_\e\in W^{1,p}(\Omega_\e,V)$ such that $u_\e\rto u$ and
\begin{equation}
\lim\limits_{\e\to 0}\mf_\e(u_\e)= \mf(u).
\end{equation}
\end{proposition}
\begin{proof}
Let $T:W^{1,p}(\Omega,V)\to W^{1,p}(\Omega',V)$, as in \Cref{propExtension}. We define $u_\e = R_{\Omega_\e}T u$, and note that this implies that $u_\e\in W^{1,p}(\Omega_\e,V)$. First we must show that $u_\e\rto u$. It is immediate that $u_\e|_U= u|_U$ for all $U\subset\Omega\cap\Omega_\e$, which consequently implies that $u_\e|_U,\nabla u_\e|_U$ converge strongly in $L^p(U)$ to $u,\nabla u$ respectively. As the extensions by zero of $u_\e,\nabla u_\e$ may only disagree with $u,\nabla u$ on $\Omega_\e\setminus \Omega$, we see that 

 \begin{equation}
 ||E_Du_\e- E_Du||_{L^p(D,V)}=\left(\int_{(\Omega_\e\setminus\Omega)\cup(\Omega\setminus\Omega_\e)}|Tu(x)|^p\,dx\right)^\frac{1}{p},
 \end{equation}

which as $|\Omega_\e\setminus \Omega|+|\Omega\setminus \Omega_\e|\to 0$ and $||Tu||_{ L^p(\Omega',V)}<\infty$, 
\begin{equation}
\lim\limits_{\e\to 0}\int_{(\Omega_\e\setminus\Omega)\cup(\Omega\setminus\Omega_\e)}|Tu(x)|^p\,dx=0.
\end{equation}
The same argument applies to show that $E_D\nabla u_\e\to E_D\nabla u$. Therefore $u_\e\rto u$

Moreso, $x\mapsto F(\nabla T u(x), T u(x),x)$ is an integrable function on $\Omega'$. As $\mathcal{L}^N\big((\Omega_\e\setminus\Omega )\cup (\Omega\setminus \Omega_\e)\big)\to 0$, this therefore implies that 
\begin{equation}\begin{split}
\lim\limits_{\e\to 0}\int_{\Omega_\e} F(\nabla u_\e(x),u_\e(x),x)\,dx=& \lim\limits_{\e\to 0}\int_{\Omega_\e} F(\nabla Tu(x),Tu(x),x)\,dx\\
=& \int_{\Omega} F(\nabla Tu(x),Tu(x),x)\,dx\\
  =&\int_{\Omega} F(\nabla u(x),u(x),x)\,dx .
\end{split}
\end{equation}
Finally, as $u_\e\rto u$, by \Cref{propSurfaceContinuous} we have that 
\begin{equation}
\lim\limits_{\e\to 0}\int_{\Gamma_{\e}}w(P_\Gamma x,\nu_{\e}(x),u_\e(x))\,dS(x) =\int_{\Gamma}w_h(x,u(x))\,dS(x). 
\end{equation}
\end{proof}

\begin{proposition}[Fundamental of $\Gamma$-convergence on varying domains]\label{propGammaConvVarying}
Let $\Omega,\Omega_\e$ satisfy Assumptions \ref{assumpOmega}, \ref{assumpOmegaEps}, $F,w$ satisfy Assumptions \ref{assumpInterior}, \ref{assumpSurface2}, and \ref{assumpSurface1}. Then if $u_\e\in W^{1,p}(\Omega_\e,V)$ are minimisers $\mf_\e$ for every $\e>0$, there exists a subsequence $\e_j\to 0$ and minimiser $u\in W^{1,p}(\Omega,V)$ of $\mf$ such that $u_{\e_j}\rto u$.

\end{proposition}

\begin{proof}
The proof follows exactly according to that of \cite{braides2006handbook}.

We begin by noting that there exists $u\in W^{1,p}(\Omega,V)$ such that $\mathcal{F}(u)<+\infty$. This holds as by  \Cref{assumpInterior}, there exists $u$ such that 
\begin{equation}
\int_\Omega F(\nabla u(x),u(x),x)\,dx<+\infty, 
\end{equation}
and we note that by the estimates in \Cref{corollaryHomoBounds} and the compact embedding $W^{1,p}(\Omega,V)\hookrightarrow L^q(\Gamma,V)$, we must also have that 
\begin{equation}
\int_\Gamma w_h(x,u(x))\,dx<+\infty,
\end{equation}
so that $\mathcal{F}(u)<+\infty$. By taking a recovery sequence $u_\e\in W^{1,p}(\Omega_\e,V)$ so that $u_\e\rto u$, which exists via \Cref{propLimsup}, we have that $\lim\limits_{\e\to 0}\mf_\e(u_\e)\to \mf(u)<+\infty$, and consequently it must hold that $\limsup\limits_{\e\to 0}\inf\limits_{W^{1,p}(\Omega_\e,V)}\mf_\e<+\infty$ also. Therefore minimisers $u_\e^*$ of $\mf_\e$ must satisfy the bound that $\mf_\e(u_\e^*)\leq M$ for all $\e_0>\e>0$ and some $M>0$. Thus by \ref{propCompactness}, there exists a subsequence $\e_j\to 0$ and $u^*\in W^{1,p}(\Omega,V)$ so that $u_{\e_j}^*\rto u^*$. By \ref{propLiminf}, thus 
\begin{equation}
\liminf\limits_{j\to\infty}\mf_{\e_j}(u_{\e_j}^*)\geq \mf(u^*). 
\end{equation}
Now, take any $u\in W^{1,p}(\Omega,V)$. By \ref{propLimsup}, there exists a sequence $(u_\e)_{\e>0}$ so that $u_\e\in W^{1,p}(\Omega_\e,V)$, $u_\e\rto u$, and $\lim\limits_{\e\to 0}\mf_\e(u_\e)=\mf(u)$. Then, as $u^*_\e$ are minimisers of $\mf_\e$ by assumption, we have that 
\begin{equation}
\mf(u)=\limsup\limits_{\e\to 0}\mf_\e(u_\e)\geq \limsup_{j\to\infty}\mf_{\e_j}(u_{\e_j})\geq  \liminf\limits_{j\to\infty}\mf_\e(u_{\e_j}^*)\geq \mf(u^*). 
\end{equation}
Therefore $u^*$ is a minimiser of $\mf$.
\end{proof}

\section{Applications to liquid crystals models}\label{subsecModels}

\subsection{Oseen-Frank}
We now turn to the cases of surface energies in models for nematic liquid crystals. First we consider the Oseen-Frank model, where  $V=\mathbb{R}^3$, although finite-energy states $n\in W^{1,2}(\Omega,V)$ satisfy $|n(x)|=1$ almost everywhere, as described in \Cref{remarkOFVector}. The Rapini-Papoular surface energy is given by 
\begin{equation}\begin{split}
w(\nu,n)= & \frac{w_0}{2}(n\cdot\nu)^2.\\
\end{split}
\end{equation}
The sign of $w_0$ produces different behaviour of the surface energy, although it does not change the existence of a lower bound, as $|n(x)|=1$ almost everywhere for finite energy $n$. For $w_0<0$, the energy corresponds to {\it homeotropic anchoring}, and is minimised when $n\cdot\nu=\pm 1$, whereby molecules prefer to be aligned perpendicularly to the surface. On the other hand, $w_0>0$ corresponds to {\it planar degenerate anchoring}, in which $n$ prefers to lie in the plane such that $\nu\cdot n =0$. In either case, we may write the energy, ignoring additive constants, as 
\begin{equation}
w(\nu,n) = \frac{w_0}{2}(n\otimes n )\cdot (\nu\otimes \nu). 
\end{equation}
In particular, following \Cref{propWeakSurfaceConv} we see that the homogenised surface energy is described entirely by a single $3\times 3$ symmetric tensor $A_{ef}:\Gamma\to \mathbb{R}^{3\times 3}$ given by, 
\begin{equation}
A_{ef}(x)=w_0\int_{\nu(x)+T_x\Gamma} \frac{1}{|v|}v\otimes v\,d\mu_x(v) ,
\end{equation}
from which the homogenised surface energy can be written as 
\begin{equation}
w_h(x,n)=\frac{1}{2}A_{ef}n\cdot n. 
\end{equation}
In particular, from its definition we see that $A_{ef}$ is a symmetric $3\times 3$ matrix and thus admits a spectral decomposition $A_{ef} = \sum\limits_{i=1}^3 \lambda_ie_i\otimes e_i$ at each point $x\in\Gamma$. As the trace is a linear function, we may take the trace before or after the integral to obtain
\begin{equation}
\frac{1}{w_0}\text{Tr}(A_{ef})=\int_{\nu(x)+T_x\Gamma} |v|\,d\mu_x(v)\geq 1,
\end{equation} 
and similarly for any $e\in \mathbb{S}^{2}$,
\begin{equation}
\frac{1}{w_{0}}A_{ef}e\cdot e=\int_{\nu(x)+T_x\Gamma} (e\cdot v)^2|v|^{-1}\,d\mu_x(v)\in [0,\frac{1}{w_0}\text{Tr}(A_{ef})].
\end{equation}

We see that there are four possibilities for $A_{ef}$, qualitatively speaking. We order the eigenvalues such that $\lambda_1\leq \lambda_2\leq \lambda_3$. 
\begin{enumerate}
\item $\lambda_1 = \lambda_2 < \lambda_3$. In this case, the energy behaves as a degenerate anchoring in the $e_1,e_2$ plane with weight $\lambda_3-\lambda_2$. We may write the limiting surface energy as $A_{ef} n\cdot n = (\lambda_3-\lambda_2) (e_3\cdot n)^2+\lambda_2$. We note that the latter term is $n$ independent, and $(\lambda_3-\lambda_2)$ is positive. 
\item $\lambda_1 < \lambda_2 = \lambda_3$. In this case, similarly to the previous case, we have a simple preferred axis $e_1$ giving weak homeotropic anchoring with weight $\lambda_1-\lambda_3$, and the energy may be written as $A_{ef} n\cdot n =  (\lambda_1-\lambda_3) (e_1\cdot n)^2+\lambda_3$, where we note that the coefficient $\lambda_1-\lambda_3$ is negative. 
\item $\lambda_1 = \lambda_2 = \lambda_3$. This reduces to no anchoring condition on $n$, as $A_{ef} n\cdot n= \lambda_1$ for all $n$. 
\item $\lambda_1 < \lambda_2 < \lambda_3$. Here we have a preferred axis $e_1$, however the energy is not rotationally symmetric around $e_1$, with deviations in the $e_2$ being preferable to those in the $e_3$ direction. 

Finally, we note that up to a constant $w_h$ is uniquely defined only by to the traceless part of $A_{ef}$, by writing 
\begin{equation}
A_{ef}=\left(A-\frac{\text{Tr}(A_{ef})}{3}I \right)+\frac{\text{Tr}(A_{ef})}{3}I=A^\circ_{ef}+\frac{\text{Tr}(A_{ef})}{3}I,
\end{equation} we see that 
\begin{equation}
w_h(x,n)=A^\circ_{ef}n\cdot n + \frac{1}{3}\text{Tr}(A_{ef}),
\end{equation}
where the latter term is independent of $n$ and thus does not affect minimisers. 

\end{enumerate}

\subsection{Landau-de Gennes}\label{subsubsect:LdG}

The simplest surface energy for Landau-de Gennes, in which $V=\text{Sym}_0(3)$, the traceless symmetric matrices, is given as 
\begin{equation}
w(x,\nu,Q)=\frac{w_0}{2}\left|Q-s_0\left(\nu\otimes \nu-\frac{1}{3}I\right)\right|^2.
\end{equation}
The constant $w_0>0$ is the anchoring strength, clearly a larger value of $w_0$ corresponds to a higher penalisation for deviation from the preferred state. This energy is minimised at $Q=s_0\left(\nu\otimes \nu-\frac{1}{3}\right)$, meaning that molecules generally want to lie parallel to the surface if $s_0>0$ or perpendicularly if $s_0<0$. Similarly to the Oseen-Frank model, we can investigate the behaviour of the model in the limit using \Cref{propWeakSurfaceConv}, first by expanding the energy as 
\begin{equation}
w(x,\nu,Q)=\frac{w_0}{2}|Q|^2-w_0s_0Q\cdot\left(\nu\otimes\nu-\frac{1}{3}I\right)+\frac{w_0 s_0^2}{3}.
\end{equation}
Then we define the effective anchoring strength $w_{ef}$ and the preferred direction $Q_{ef}$ as 
\begin{equation}
\begin{split}
w_{ef}(x)=&w_0\int_{\nu(x)+T_x\Gamma}|v|\,d\mu_x,\\
Q_{ef}(x)=&\frac{s_0}{w_{ef}(x)}\int_{\nu(x)+T_x\Gamma}\left(\frac{1}{|v|^2}v\otimes v -\frac{1}{3}I\right)|v|\,d\mu_x(v).
\end{split}
\end{equation}

Then we may write the homogenised surface energy as 
\begin{equation}\begin{split}
w_h(x,Q)=&\int_{\nu(x)+T_x\Gamma}\left(\frac{w_0}{2}|Q|^2-w_0s_0Q\cdot\left(\frac{1}{|v|^2}v\otimes v-\frac{1}{3}I\right)+\frac{2w_0s_0^2}{3}\right)|v|\,d\mu_x(v)\\
=&\frac{w_{ef}}{2}|Q-Q_{ef}|^2+R,\end{split}
\end{equation}
where $R$ is a remainder term that depends only on $x$ and not $Q$, and thus does not affect minimisers. We observe that the homogenised surface energy is thus of the same quadratic form, but $Q_{ef}$ plays the role of the preferred direction, and $w_{ef}$ as the new anchoring strength. As $|v|\geq 1$ for all $v\in \nu(x)+T_x\Gamma$, by the integral form of $w_{ef}$ we see that $w_{ef}>w_0$. That is, the anchoring strength can only increase. Similarly, if $f$ is any frame indifferent convex function, we have that for any $e\in\mathbb{S}^2$,
\begin{equation}
f\left(s_0\left(e_1\otimes  e_1-\frac{1}{3}I\right)\right)=g_f(s_0)
\end{equation} 
for some function $g_f:\mathbb{R}\to\mathbb{R}$. By applying Jensen's inequality, we see that 
\begin{equation}
\begin{split}
f(Q_{ef}(x))=&f\left(\frac{s_0}{\int_{\nu(x)+T_x\Gamma}|v|\,d\mu_x(v)}\int_{\nu(x)+T_x\Gamma}\left(\frac{1}{|v|^2}v\otimes v-\frac{1}{3}I\right)|v|\,d\mu_x(v)\right)\\
\leq &\frac{1}{\int_{\nu(x)+T_x\Gamma}|v|\,d\mu_x(v)}\int_{\nu(x)+T_x\Gamma}f\left(s_0\left(\frac{1}{|v|^2}v\otimes v-\frac{1}{3}I\right)\right)|v|\,d\mu_x(v)\\
\leq & \frac{1}{\int_{\nu(x)+T_x\Gamma}|v|\,d\mu_x(v)}\int_{\nu(x)+T_x\Gamma}g_f(s_0)|v|\,d\mu_x(v)=g_f(s_0).
\end{split}
\end{equation}

In particular, taking $f(Q)=\lambda_{\max}(Q)$, $f(Q)=-\lambda_{\min}(Q)$, and $ f(Q)=|Q|$, we see that $\lambda_{\max}(Q_{ef})\leq\frac{2s_0}{3}$, $\lambda_{\min}(Q_{ef})\geq -\frac{s_0}{3}$, and $|Q|\leq s_0\sqrt{\frac{2}{3}}$.

\section{Error estimates in periodic media}\label{sect:Error}

\subsection{Introduction of the problem}

Consider an $\ell$-periodic slab domain in two-dimensions, which represents the typical geometry of liquid crystal experiments, given explicitly as 
\begin{equation*}
\Omega'_\e=\{({x'},{y'})\in\mathbb{R}^2:{\varphi'_\e}({x'})<{y'}< {R'}\},
\end{equation*}
in which ${\varphi'_\e}:\mathbb{R}\to\mathbb{R}$ is an $\ell$-periodic function and let ${\partial\Omega'_\e}={\Gamma'_\e}\cup{\Gamma'_R}$, where we denote ${\Gamma'_\e}=\{({x'},{\varphi'_\e}({x'})):{x'}\in\mathbb{R}\}$ and ${\Gamma'_R}=\{({x'},{R'}):{x'}\in\mathbb{R}\}$. We consider a toy model, representative of paranematic systems as in {\cite{borvstnik1999interaction,galatola2001nematic,stark2002geometric}}, over $Q'\in\text{Sym}_0(2)=\{A\in\mathbb{R}^{2\times 2}:A^T=A,\, \text{Tr}(A)=0\}$. We consider solutions that respect the symmetry of the domain, that is, ${Q'}\in W^{1,2}_{loc}({\Omega'_\e},\text{Sym}_0(2))$ such that ${Q'}({x'}+\ell,{y'})={Q'}({x'},{y'})$ for almost every ${x'},{y'}$. The free energy per periodic cell, ${C\Omega'_\e}$, is to be given as 
\begin{equation*}
{\mathcal{F}'_\e}({Q'})=\int_{{C\Omega'_\e}}\frac{{L'}}{2}|\nabla {Q'}|^2+\frac{{c'}}{2}|{Q'}|^2\,\text{d}{x'}\,\text{d}{y'} +\int_{{C\partial\Omega'_\e}}\frac{{w'_0}}{2}\left|{Q'}-{s'_0}\left(\nu\otimes \nu-\frac{1}{2}I\right)\right|^2\,\text{d}\sigma({x'}).
\end{equation*}
Here ${C\Omega'_\e}=\{({x'},{y'})\in\mathbb{R}^2:{\varphi'_\e}({x'})<{y'}<{R'},\; 0\leq {x'}<\ell\}$, ${c'}>0$, ${C\partial\Omega'_\e} ={C\Gamma'_\e}\cup {C\Gamma'_R}$, with ${C\Gamma'_\e}=\{({x'},{\varphi'_\e}({x'})):{x'}\in[0,\ell)\}$ and ${C\Gamma'_R}=\{({x'},{R'}):{x'}\in[0,\ell)\}$, {$w'_0>0$, $s'_0\in\mathbb{R}$ and $\nu$ is the exterior normal.} We may non-dimensionalise the system by considering variables $(x,y)=\frac{2\pi}{\ell}({x'},{y'})$, $Q(x,y)=\frac{1}{{s'_0}}{Q'}({x'},{y'})$, $\mathcal{F}_\e= \frac{1}{{L'}{(s'_0)}^2}{\mathcal{F}'_\e}$, $c=\frac{{c'}\ell^2}{4{L'}\pi^2}$, $w_0=\frac{{w'_0}\ell}{2{L'}\pi}$, $R=\frac{2\pi}{\ell}{R'}$, $\varphi_\e(x)={\varphi'_\e}({x'})$ to give 
\begin{equation}\label{eq:free_energy_functional}
\mathcal{F}_\e(Q)=\int_{\Omega_\e}|\nabla Q|^2+c|Q|^2\,\text{d}x\,\text{d}y+\int_{\partial\Omega_\e}\dfrac{w_0}{2}\left|Q-\left(\nu\otimes \nu-\frac{1}{2}I\right)\right|^2\,\text{d}\sigma(x),
\end{equation}
with rescaled domain
\begin{align*}
\Omega_\e=\{(x,y):0\leq x<2\pi,\;\varphi_\e(x)<y<R\}.
\end{align*}
Moreover, we can write $\partial\Omega_\e=\Gamma_\e\cup\Gamma_R$, where $\Gamma_\e=\{(x,y):x\in[0,2\pi),\;y=\varphi_\e(x)\}$ and $\Gamma_R=\{(x,R):x\in[0,2\pi)\}$, so that we have:
\begin{equation}
\mathcal{F}_\e(Q)=\int_{\Omega_\e}|\nabla Q|^2+c|Q|^2\,\text{d}x\,\text{d}y+\int_{\Gamma_\e}\dfrac{w_0}{2}\left|Q-Q_{\varepsilon}^0\right|^2\,\text{d}\sigma_\e+\int_{\Gamma_R}\dfrac{w_0}{2}\left|Q-Q_R\right|^2\,\text{d}\sigma_R,
\end{equation}
with $Q_{\varepsilon}^0=\nu_\e\otimes \nu_\e-\frac{1}{2}I$ and $Q_R=\nu_R\otimes\nu_R-\frac{1}{2}I$, where $\nu_\e$ and $\nu_R$ are the outward normals to $\Gamma_\e$ and $\Gamma_R$.

We also consider the limit domain
\begin{align*}
\Omega_0=\{(x,y):x\in[0,2\pi),\;0<y<R\}
\end{align*} 
with $\partial\Omega_0=\Gamma_0\cup\Gamma_R$, where $\Gamma_0=\{(x,0):x\in[0,2\pi)\}$.

\subsection{Technical assumptions and main result}

\begin{assumption}\label{pp:1}
Let $\e>0$. We assume that $\varphi_\e(x)=\e\cdot\varphi(x/\e)$, where $\varphi:\mathbb{R}\to\mathbb{R}$ is a $C^2$ $2\pi$-periodic function with $\varphi\geq 0$.
\end{assumption}

\begin{remark}
Using \Cref{pp:1}, we obtain that $\Omega_\e\subset\Omega_0$ for all $\e>0$ and that $\Omega_\e\to\Omega_0$ as $\e\to 0$.
\end{remark}

\begin{figure}[h]\begin{center}
\begin{subfigure}[t]{0.4\textwidth}
\begin{center}
\includegraphics[width=\textwidth]{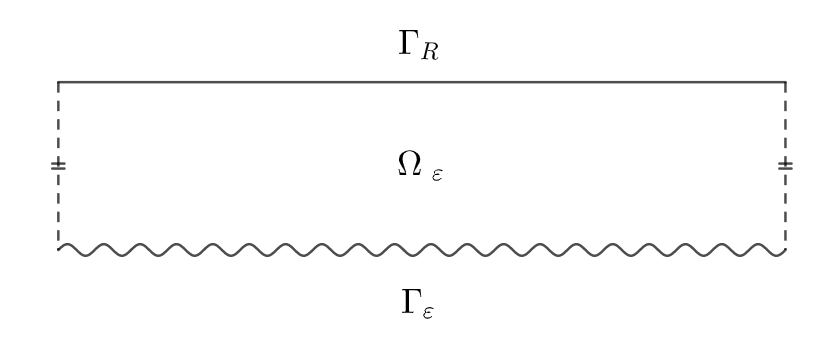}
\caption{
The oscillating domain $\Omega_\e$.
}
\end{center}
\end{subfigure}
\hspace{0.025\textwidth}
\begin{subfigure}[t]{0.4\textwidth}
\begin{center}
\includegraphics[width=\textwidth]{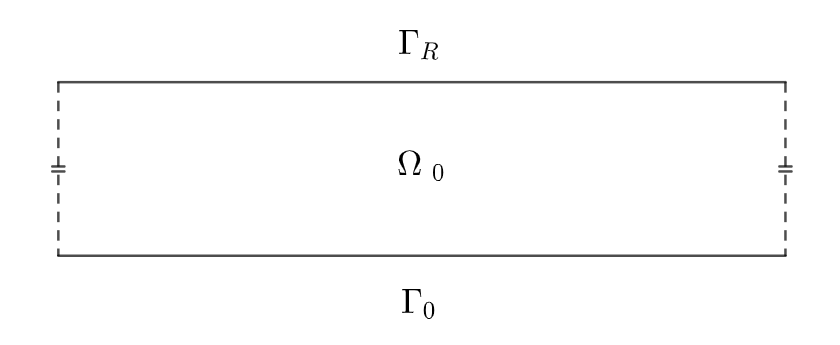}
\caption{ 
The limit domain $\Omega_0$.
}
\end{center}
\end{subfigure}
\end{center}
\end{figure}

\begin{remark}
In the next subsections, we write the first derivative of $\varphi$ as $\varphi'$. Moreover, we write $\|\varphi\|_{\infty}$ and $\|\varphi'\|_{\infty}$ instead of $\|\varphi\|_{L^{\infty}([0,2\pi))}$ and $\|\varphi'\|_{L^{\infty}([0,2\pi))}$.
\end{remark}

\begin{assumption}\label{pp:2}
We assume that:
\begin{align*}
0<\e=\dfrac{1}{2k}<\|\varphi\|^{-1}_{\infty}\cdot \dfrac{R}{2},\;\text{with}\;k\in\mathbb{N}^*,\;k>\|\varphi\|_{\infty}\cdot \dfrac{1}{R}.
\end{align*}
\end{assumption}

\begin{remark}
Using \Cref{pp:2}, we obtain that $\{(x,y)\;|\;x\in[0,2\pi),\;y\in(R/2,R)\}\subset\Omega_{\e}$, which tells us that the oscillations of $\Gamma_{\e}$ have an amplitude lower than half of the height of the domain $\Omega_0$.
\end{remark}

\begin{remark}
Using \Cref{pp:1}, the arclength parameter of the curve $\Gamma_\e$ can be described as the function $\gamma_{\e}:\mathbb{R}\to\mathbb{R}$, defined as:
\begin{align}\label{eq:gamma_e}
\gamma_{\e}(t)=\sqrt{1+\big(\varphi'(t)\big)^2},\;\forall t\in\mathbb{R}
\end{align}
and we obtain that:
\begin{align}\label{bounds:gamma_eps}
1\leq \gamma_{\e}(t)<\sqrt{1+\|\varphi'\|_{\infty}^2},\;\forall t\in\mathbb{R}.
\end{align}
Moreover, the outward normal $\nu_\e$ to $\Gamma_\e$ has the following form:
\begin{align}\label{eq:nu_eps_depends_only_on_x_e}
\nu_{\e}:=\nu_{\e}(x/\e)=\dfrac{1}{\gamma_{\e}(x/\e)}\big(\varphi'(x/\e),-1\big),
\end{align}
for all $x\in[0,2\pi)$.
\end{remark}

\begin{definition}\label{defn:g_functions}
Let $g_1,g_2:\mathbb{R}\to\mathbb{R}$ be two real functions defined as
\begin{center}
\begin{tabular}{ccc}
$g_1(t):=\dfrac{\big(\varphi'(t)\big)^2-1}{2\big(1+\big(\varphi'(t)\big)^2\big)}$ & \text{and} & $g_2(t):=\dfrac{-2 \varphi'(t)}{2\big(1+\big(\varphi'(t)\big)^2\big)}$
\end{tabular}
\end{center}
for all $t\in\mathbb{R}$ and let $\gamma,G_1,G_2\in\mathbb{R}$ be defined as:
\begin{center}
\begin{tabular}{ccc}
$\gamma:=\displaystyle{\dfrac{1}{2\pi}\int_0^{2\pi}\gamma_\e(t)\;\text{d}t,}$ & $G_1:=\displaystyle{\dfrac{1}{2\pi}\int_0^{2\pi}g_1(t)\cdot\gamma_\e(t)\;\text{d}t,}$ & $G_2:=\displaystyle{\dfrac{1}{2\pi}\int_0^{2\pi}g_2(t)\cdot\gamma_\e(t)\;\text{d}t.}$
\end{tabular}
\end{center}
\end{definition}

\begin{remark}\label{remark:Q_e_0_dependency}
Since $Q_{\varepsilon}^0=\nu_\e\otimes\nu_\e -\frac{1}{2}I$ and $\nu_{\e}=\nu_{\e}(x/\e)$, then $Q_{\varepsilon}^0(x/\e)=\begin{pmatrix}
g_1(x/\e) & g_2(x/\e)\\
g_2(x/\e) & -g_1(x/\e)
\end{pmatrix}$, for all $x\in[0,2\pi)$. Moreover, $\gamma$, $G_1$ and $G_2$ are constants and $\gamma\geq 1$.
\end{remark}

\begin{definition}
Let $w_{ef}:=\gamma w_0$ and $Q_{ef}:=\dfrac{1}{\gamma}\begin{pmatrix}
G_1 & G_2\\
G_2 & -G_1
\end{pmatrix}$.
\end{definition}

\begin{remark}
In this simplified model, $w_{ef}\in\mathbb{R}$ is constant and, since $\gamma\geq 1$, we have $w_{ef}\geq w_0$ (as previously observed in \Cref{subsubsect:LdG}. Moreover, $Q_{ef}\in\text{Sym}_0(2)$ is a constant $Q$-tensor.
\end{remark}

\begin{proposition}\label{prop:Q_eff}
We have $\gamma_{\e}(\cdot)\rightharpoonup \gamma$ and $\gamma_{\e}(\cdot/\e)Q_{\e}^0(\cdot/\e)\rightharpoonup \gamma Q_{ef}$ in $L^2([0,2\pi))$, as $\e\to 0$.
\end{proposition}

\begin{proof}
Since $\varphi$ and $\varphi'$ are $2\pi$-periodic, according to \Cref{pp:1}, and $\e^{-1}\in\mathbb{N}^*$, according to \Cref{pp:2}, then the functions $\gamma_{\e}$, $g_1\gamma_{\e}$ and $g_2\gamma_{\e}$ are also $2\pi$-periodic and they tend to their mean as $\e\to 0$ (see for instance \cite[Lemma 9.1]{cioranescu1999}). This implies the conclusion.
\end{proof}

\begin{definition}
Let $Q_\e$ be a minimiser of the functional $\mathcal{F}_{\e}$. This implies that $Q_\e$ verifies the following Euler-Lagrange equations:
\begin{align}\label{defn:PDE_Q_eps}
\begin{cases}
-\Delta Q_{\e}+cQ_{\e}=0\;\;\;\text{in}\;\Omega_{\e};\\
\\
\dfrac{\partial Q_{\e}}{\partial\nu_{\e}}+\dfrac{w_0}{2}Q_{\e}=\dfrac{w_0}{2}Q^0_{\e}\;\;\;\text{on}\;\Gamma_{\e};\\
\\
\dfrac{\partial Q_{\e}}{\partial\nu_R}+\dfrac{w_0}{2}Q_{\e}=\dfrac{w_0}{2}Q_R\;\;\;\text{on}\;\Gamma_R.
\end{cases}
\end{align}
\end{definition}

\begin{remark}
We prove in \Cref{subsect:reg} that there exists a unique $Q_{\e}\in W^{2,p}(\Omega_{\e})$ solution of \eqref{defn:PDE_Q_eps}, for any $1<p<+\infty$.
\end{remark}

\begin{definition}
Let $\nu_0(\cdot,0)=(0,-1)$ be the outward normal of $\Gamma_0$. We consider the following PDE:
\begin{align}\label{defn:PDE_Q_0}
\begin{cases}
-\Delta Q_0+cQ_0=0\;\;\;\text{in}\;\Omega_0;\\
\\
\dfrac{\partial Q_0}{\partial \nu_0}+\dfrac{w_{ef}}{2} Q_0=\dfrac{w_{ef}}{2}Q_{ef}\;\;\;\text{on}\;\Gamma_0;\\
\\
\dfrac{\partial Q_0}{\partial \nu_R}+\dfrac{w_0}{2}Q_0=\dfrac{w_0}{2}Q_R\;\;\;\text{on}\;\Gamma_R.
\end{cases}
\end{align}
\end{definition}

\begin{remark}
In \Cref{subsect:reg}, we prove that we have a unique $Q_0\in W^{2,p}(\Omega_0)$ solution of \eqref{defn:PDE_Q_0}, for any $p\in(1,+\infty)$.
\end{remark}

\begin{remark}
It is easy to see that under these assumptions, we have $Q_\e\rto Q_0$ in the sense from \Cref{defRugoseConv}. However, in this simplified case, we can obtain {\it quantitative estimates}, which are presented in the main theorem of this section:
\end{remark}

\begin{theorem}\label{thm:main} 
For any $p\in(2,+\infty)$, there exists an $\e$-independent constant $C$ such that:
\begin{align}
\|Q_0-Q_{\e}\|_{L^2(\Omega_{\e})}\leq C\cdot \e^{\frac{p-1}{p}},
\end{align}
where the constant $C$ depends on $c$, $w_0$, $p$, $\|\varphi\|_{\infty}$, $\|\varphi'\|_{\infty}$, $\Omega_0$ and $\|Q_0\|_{W^{1,\infty}(\Omega_0)}$.
\end{theorem}

\begin{remark}
The constant $C$ from \Cref{thm:main} can actually be chosen of the following form:
\begin{align*}
C=\text{max}\big\{1,\|\varphi\|_{\infty}^{(p-1)/p}\big\}\cdot\sqrt{1+\|\varphi'\|_{\infty}^2}\cdot C(w_0,c,p,\Omega_0,Q_0),
\end{align*}
where $C(w_0,c,p,\Omega_0,Q_0)$ is an $\e$-independent constant depending only on $w_0$, $c$, $p$, $\Omega_0$ and $\|Q_0\|_{W^{1,\infty}(\Omega_0)}$.
\end{remark}

\subsection{Regularity of $Q_{\e}$ and $Q_0$}\label{subsect:reg}

In this subsection, we prove that there exists a unique solution $Q_{\e}$ of \eqref{defn:PDE_Q_eps} and a unique solution $Q_0$ of \eqref{defn:PDE_Q_0}, with $Q_{\e}\in W^{2,p}(\Omega_{\e})$ and $Q_0\in W^{2,p}(\Omega_0)$, for any $p\in(1,+\infty)$.

\begin{remark}
It is easy to see that problems \eqref{defn:PDE_Q_eps} and \eqref{defn:PDE_Q_0} admit solutions from $W^{1,2}(\Omega_0)$ and $W^{1,2}(\Omega_{\e})$ by using, for example, direct methods, such as in \cite{dacorogna2007dm}, or the approach via Lax-Milgram theorem for elliptic problems, such as in \cite{evans2010second} or \cite{grisvard1985non}.
\end{remark}

\begin{definition}\label{defn:THE_transf}
We denote by $\Phi_{polar}$ the polar coordinates transform:
\begin{align*}
\Phi_{polar}(x,y)=(y\cos x,y\sin x),\;\forall (x,y)\in[0,2\pi)\times(0,R),
\end{align*}
by $\Phi_{tl}$ the following translation:
\begin{align*}
\Phi_{tl}(x,y)=(x,y+2R),\;\forall(x,y)\in[0,2\pi)\times(0,R)
\end{align*} 
and let $\Phi:\Omega_0\to \Phi(\Omega_0)$ be defined as $\Phi=\Phi_{polar}\circ \Phi_{tl}$. We define
\begin{align*}
\mathcal{U}_{\e}=\Phi(\Omega_{\e})\hspace{5mm}\text{and}\hspace{5mm}\mathcal{U}_0=\Phi(\Omega_0).
\end{align*}
\end{definition}

\begin{definition}
For any $a,b\in\mathbb{R}$ with $a,b>0$, we denote
\begin{align*}
A_{a,b}=\{(y\cos x,y\sin x)\;|\;x\in[0,2\pi),\;y\in(a,b)\}.
\end{align*}
\end{definition}

\begin{remark}
The transformation $\Phi:\Omega_0\to\mathcal{U}_0$ is smooth and bi-Lipschitz. Moreover, using \Cref{pp:1}, we have that $\mathcal{U}_{\e}$ is a bounded open domain from $\mathbb{R}^2$ with a $C^2$ boundary and, using \Cref{pp:2}, we have that $A_{5R/2,3R}\subset \mathcal{U}_{\e}\subset\mathcal{U}_0=A_{2R,3R}$.
\end{remark}

In order to prove that $Q_{\e}$ and $Q_0$ admit $W^{2,p}$ regularity, we use \cite[Theorem 2.4.2.6]{grisvard1985non}:

\begin{theorem}\label{thm:reg}
Let $\Omega$ be a bounded open subset of $\mathbb{R}^n$, with a $C^{1,1}$ boundary. Let $a_{i,j}$ and $b_i$ be uniformly Lipschitz functions in $\overline{\Omega}$ and let $a_i$ be bounded measurable functions in $\overline{\Omega}$. Assume that $a_{i,j}=a_{j,i}$, $1\leq i,j\leq n$ and that there exists $\alpha>0$ with
\begin{align*}
\displaystyle{\sum_{i,j=1}^n a_{i,j}(x)\xi_i\xi_j\leq -\alpha|\xi|^2}
\end{align*}
for all $\xi\in\mathbb{R}^n$ and almost every $x\in\overline{\Omega}$. Assume in addition that $a_0\geq \beta>0$ a.e. in $\Omega$ and that
\begin{align*}
\displaystyle{b_0b_{\nu}=b_0\sum_{j=1}^n b_j\nu^j\geq 0,\hspace{4mm}b_{\nu}\neq 0}
\end{align*}
on $\Gamma=\partial\Omega$. Then for every $f\in L^p(\Omega)$ and every $g\in W^{1-1/p,p}(\Gamma)$, there exists a unique $u\in W^{2,p}(\Omega)$, which is a solution of
\begin{align*}
\begin{cases}
\displaystyle{\sum_{i,j=1}^n D_i(a_{i,j}D_ju)+\sum_{i=1}^n a_iD_iu+a_0u=f\hspace{4mm}\text{in}\;\Omega}\\
\\
\displaystyle{\text{Tr}\bigg(\sum_{j=1}^n b_jD_ju+b_0u\bigg)=g\hspace{24mm}\text{on}\;\Gamma}
\end{cases}
\end{align*}
where $\text{Tr}$ is the trace operator.
\end{theorem}

\begin{corollary}\label{cor:Q_eps}
For any $p\in(1,+\infty)$, there exists a unique $Q_{\e}\in W^{2,p}(\Omega_{\e})$ which solves the problem \eqref{defn:PDE_Q_eps}.
\end{corollary}

The proof of this corollary can be found in \Cref{appendix:err_est}. In a similar way, we can show that:
\begin{corollary}\label{cor:Q_0}
For any $p\in(1,+\infty)$, there exists a unique $Q_0\in W^{2,p}(\Omega_0)$ which solves the problem \eqref{defn:PDE_Q_0}.
\end{corollary}

\begin{remark}
Using the method of separation of variables, one can find that $Q_0$ is of the form:
\begin{align*}
Q_0(x,y)=c_1\cdot e^{y\sqrt{c}}+c_2\cdot e^{-y\sqrt{c}},
\end{align*}
where $c_1$ and $c_2$ are two constant $Q$-tensors that can be found from:
\begin{align*}
\begin{cases}
c_1\cdot \bigg(\dfrac{w_{ef}}{2}-\sqrt{c}\bigg)+c_2\cdot\bigg(\dfrac{w_{ef}}{2}+\sqrt{c}\bigg)=\dfrac{w_{ef}}{2}\cdot Q_{ef};\\
\\
c_1\cdot e^{R\sqrt{c}}\cdot\bigg(\dfrac{w_0}{2}+\sqrt{c}\bigg)+c_2\cdot e^{-R\sqrt{c}}\cdot\bigg(\dfrac{w_0}{2}-\sqrt{c}\bigg)=\dfrac{w_0}{2}\cdot Q_R.
\end{cases}
\end{align*}
\end{remark}

\subsection{Some integral inequalities}\label{subsect:integral_inequalities}

\begin{definition}\label{defn:trace_ineq}
Let $p\in(1,+\infty)$. Let us consider the trace operator $\text{Tr}:W^{1,p}(\Omega_0)\to W^{1-1/p,p}(\partial\Omega_0)$. We denote $C_{tr}(p,\Omega_0)$ the constant given by the trace inequality, that is:
\begin{align*}
\|\text{Tr}(\omega)\|_{W^{1-1/p,p}(\partial\Omega_0)}\leq C_{tr}(p,\Omega_0)\cdot\|\omega\|_{W^{1,p}(\Omega_0)},\;\forall \omega\in W^{1,p}(\Omega_0).
\end{align*}
\end{definition}

\begin{remark}
For notation simplicity, we choose to write $v(\cdot,0)$ instead of $[\text{Tr}(v)]\big|_{\Gamma_0}(\cdot,0)$, whenever $v\in W^{1,p}(\Omega_0)$.
\end{remark}

\begin{definition}\label{defn:bilinear_funct}
We consider the following bilinear functional on $W^{1,2}(\Omega_{\e})\times W^{1,2}(\Omega_{\e})$:
\begin{align*}
a_{\e}(u,v)=\int_{\Omega_{\e}}\big(\nabla u\cdot\nabla v+c\cdot u\cdot v\big)\text{d}(x,y)+\dfrac{w_0}{2}\bigg(\int_{\Gamma_{\e}}u\cdot v\;\text{d}\sigma_{\e}+\int_{\Gamma_R}u\cdot v\;\text{d}\sigma_R\bigg),
\end{align*}
for any $u,v\in W^{1,2}(\Omega_{\e})$.
\end{definition}

\begin{remark}
For notation simplicity, whenever we choose $v\in W^{1,2}(\Omega_0)$, we write $a_{\e}(\cdot,v)$ instead of $a_{\e}(\cdot,v|_{\Omega_{\e}})$.
\end{remark}

The goal of this subsection is to prove the following proposition:

\begin{proposition}\label{prop:ineq_for_W_1_p_Omega_0}
Let $p\in(2,+\infty)$. Then there exists an $\e$-independent constant $C_I$ such that:
\begin{align*}
|a_{\e}(Q_0-Q_{\e},v)|\leq C_I\cdot \e^{\frac{p-1}{p}}\cdot\|v\|_{W^{1,p}(\Omega_0)},\;\forall\;v\in W^{1,p}(\Omega_0).
\end{align*}
\end{proposition}

\begin{remark}
In \Cref{appendix:err_est}, we prove that $a_{\e}(\cdot,v)$ is well-defined for all $v\in W^{1,p}(\Omega_0)$. Moreover, in order to obtain \Cref{prop:ineq_for_W_1_p_Omega_0}, we split $a_{\e}(Q_0-Q_{\e},v)$ into several parts, which are presented in the following definition.
\end{remark}

\begin{definition}\label{defn:I_integrals}
Let $v\in W^{1,p}(\Omega_0)$. We denote:
\begin{align*}
I_1&=-\displaystyle{\int_{\Omega_0\setminus\Omega_{\e}}\big(\nabla Q_0\cdot\nabla v+c\cdot Q_0\cdot v\big)\text{d}(x,y)},\\
I_{21}&=-\dfrac{w_0}{2}\int_0^{2\pi} (v(x,\e \varphi(x/\e))-v(x,0))\cdot Q^0_{\e}(x/\e)\cdot \gamma_{\e}(x/\e)\;\text{d}x,\\
I_{31}&=\dfrac{w_0}{2}\int_0^{2\pi} \big(Q_0(x,\e \varphi(x/\e))\cdot v(x,\e \varphi(x/\e))-Q_0(x,0)\cdot v(x,0)\big)\cdot\gamma_{\e}(x/\e)\;\text{d}x,\\
I_{22}&=\dfrac{w_0}{2}\int_0^{2\pi} v(x,0)\cdot\bigg(\gamma Q_{ef}-\gamma_{\e}(x/\e)Q^0_{\e}(x/\e)\bigg)\;\text{d}x,\\
I_{32}&=-\dfrac{w_0}{2}\int_0^{2\pi} Q_0(x,0)\cdot v(x,0)\cdot \bigg(\gamma-\gamma_{\e}(x/\e)\bigg)\;\text{d}x.
\end{align*}
\end{definition}

\begin{proposition}\label{prop:sum_of_integrals}
We have
\begin{align}\label{eq:sum_of_integrals}
a_{\e}(Q_0-Q_{\e},v)=I_1+I_{21}+I_{22}+I_{31}+I_{32},\;\forall\;v\in W^{1,p}(\Omega_0).
\end{align}
\end{proposition}

\begin{remark}
The proof of \Cref{prop:sum_of_integrals} can be found in \Cref{appendix:err_est}. Before proving \Cref{prop:ineq_for_W_1_p_Omega_0}, we obtain first estimates for each of the integrals from \Cref{defn:I_integrals}.
\end{remark}

\begin{proposition}\label{prop:I_1}
Let $p\in(1,+\infty)$. Then for any $v\in W^{1,p}(\Omega_0)$, we have:
\begin{align*}
\big|I_1\big|\leq C_1\cdot \e^{\frac{p-1}{p}}\cdot\|v\|_{W^{1,p}(\Omega_0\setminus\Omega_{\e})},\;\forall v\in W^{1,p}(\Omega_0),
\end{align*}
where
\begin{align*}
I_1&=-\displaystyle{\int_{\Omega_0\setminus\Omega_{\e}}\big(\nabla Q_0\cdot\nabla v+c\cdot Q_0\cdot v\big)\text{d}(x,y)}.
\end{align*}
and
\begin{align}\label{defn:C_1}
C_1=(2\pi\|\varphi\|_{\infty})^{\frac{p-1}{p}}\cdot\|Q_0\|_{W^{1,\infty}(\Omega_0)}\cdot \text{max}\{c,1\}.
\end{align}
\end{proposition}

\begin{proof} 
We apply H\" older inequality with coefficients $p$ and $p'=\dfrac{p}{p-1}$:
\begin{align*}
\big|I_1\big| &\leq \int_{\Omega_0\setminus\Omega_{\e}}\big|\nabla Q_0\cdot\nabla v\big|+c\big|Q_0\cdot v\big|\;\text{d}(x,y)\\
&\leq \bigg(\int_{\Omega_0\setminus\Omega_{\e}}\big|\nabla Q_0\big|^{p'}\;\text{d}(x,y)\bigg)^{1/p'}\cdot\bigg(\int_{\Omega_0\setminus\Omega_{\e}}\big|\nabla v\big|^p\;\text{d}(x,y)\bigg)^{1/p}+\\
&\hspace{10mm}+c\cdot \bigg(\int_{\Omega_0\setminus\Omega_{\e}}\big|Q_0\big|^{p'}\;\text{d}(x,y)\bigg)^{1/p'}\cdot\bigg(\int_{\Omega_0\setminus\Omega_{\e}}\big|v\big|^p\;\text{d}(x,y)\bigg)^{1/p}\\
&\leq \big|\Omega_0\setminus\Omega_{\e}\big|^{1/p'}\bigg(\|\nabla Q_0\|_{L^{\infty}(\Omega_0\setminus\Omega_{\e})}\cdot\|\nabla v\|_{L^p(\Omega_0\setminus\Omega_{\e})}+c\cdot \|Q_0\|_{L^{\infty}(\Omega_0\setminus\Omega_{\e})}\cdot\|v\|_{L^p(\Omega_0\setminus\Omega_{\e})}\bigg).
\end{align*}
We have:
\begin{align*}
\big|\Omega_0\setminus\Omega_{\e}\big|&=\int_{\Omega_0\setminus\Omega_{\e}}\;1\;\text{d}x=\int_0^{2\pi}\int_{0}^{\e \varphi(x/\e)}\;1\;\text{d}y\;\text{d}x\\
&=\int_0^{2\pi} \e \varphi(x/\e)\;\text{d}x\leq \e\cdot 2\pi\|\varphi\|_{\infty}.
\end{align*}

In the end, we obtain that:
\begin{align*}
|I_1|\leq C_1\cdot \e^{\frac{p-1}{p}}\cdot\|v\|_{W^{1,p}(\Omega_0\setminus\Omega_{\e})},
\end{align*}
where 
\begin{align*}
C_1=(2\pi\|\varphi\|_{\infty})^{\frac{p-1}{p}}\cdot\|Q_0\|_{W^{1,\infty}(\Omega_0)}\cdot \text{max}\{c,1\}.
\end{align*}
\end{proof}

For the integrals $I_{21}$ and $I_{31}$, from \Cref{defn:I_integrals}, we prove first the following lemma:

\begin{lemma}\label{lemma_C_2}
Let $1<q<p<+\infty$ and $\omega\in W^{1,p}(\Omega_0)$. Then:
\begin{align}\label{ineq:C_2}
\bigg(\int_0^{2\pi}\big|\omega(x,\e \varphi(x/\e))-\omega(x,0)\big|^q\;\text{d}x\bigg)^{1/q}\leq C_{2,q}\cdot \e^{\frac{p-1}{p}}\cdot \|\nabla \omega\|_{L^p(\Omega_0\setminus\Omega_{\e})},
\end{align}
with 
\begin{align}\label{defn:C_{2,q}}
C_{2,q}=(2\pi)^{\frac{p-q}{pq}}\|\varphi\|_{\infty}^{\frac{p-1}{p}}.
\end{align}
\end{lemma}

\begin{proof}
We prove the result first for $C^1(\overline{\Omega_0})$ functions. 

Let $\omega\in C^1(\overline{\Omega_0})$. Then we have the following inequality:
\begin{align*}
\big|\omega(x,\e \varphi(x/\e))-\omega(x,0)\big|\leq \int_0^{\e \varphi(x/\e)} \bigg|\dfrac{\partial \omega}{\partial y}(x,t)\bigg|\;\text{d}t,
\end{align*}
for any $x\in[0,2\pi]$. 

We apply H\" older inequality with exponents $q$ and $q'=\dfrac{q}{q-1}$:
\begin{align*}
|\omega(x,\e \varphi(x/\e))-\omega(x,0)|&\leq \bigg(\int_0^{\e \varphi(x/\e)}\bigg|\dfrac{\partial \omega}{\partial y}(x,t)\bigg|^q\text{d}t\bigg)^{1/q}\bigg(\int_0^{\e \varphi(x/\e)} 1^{q/(q-1)}\text{d}t\bigg)^{(q-1)/q}\\
\Rightarrow|\omega(x,\e \varphi(x/\e))-\omega(x,0)|^q &\leq\int_0^{\e \varphi(x/\e)} \big(\e \varphi(x/\e)\big)^{(q-1)}\cdot \bigg|\dfrac{\partial \omega}{\partial y}(x,t)\bigg|^q\text{d}t\\
\Rightarrow\int_0^{2\pi} |\omega(x,\e \varphi(x/\e))-\omega(x,0)|^q\text{d}x &\leq \int_0^{2\pi}\int_0^{\e \varphi(x/\e)} \big(\e \varphi(x/\e)\big)^{(q-1)}\cdot \bigg|\dfrac{\partial \omega}{\partial y}(x,t)\bigg|^q\text{d}t\;\text{d}x
\end{align*}

Let now $k=\dfrac{p}{q}>1$. We apply H\" older inequality with exponents $k$ and $k'=\dfrac{k}{k-1}=\dfrac{p}{p-q}$ to obtain:
\begin{align*}
\int_0^{2\pi} |\omega(x,\e \varphi(x/\e))-\omega(x,0)|^q\text{d}x &\leq \bigg(\int_0^{2\pi}\int_0^{\e \varphi(x/\e)} \big(\e \varphi(x/\e)\big)^{(q-1)\cdot\frac{p}{p-q}}\text{d}t\;\text{d}x\bigg)^{\frac{p-q}{p}}\cdot\\
&\hspace{5mm}\cdot\bigg(\int_0^{2\pi}\int_0^{\e \varphi(x/\e)} \bigg|\dfrac{\partial \omega}{\partial y}(x,t)\bigg|^{q\cdot\frac{p}{q}}\text{d}t\;\text{d}x\bigg)^{\frac{q}{p}}\\
&\leq \e^{(q-1)+\frac{p-q}{p}}\bigg(\int_0^{2\pi}\big(\varphi(x/\e)\big)^{1+\frac{pq-p}{p-q}}\text{d}x\bigg)^{\frac{p-q}{p}}\cdot\|\nabla \omega\|^q_{L^p(\Omega_0\setminus\Omega_{\e})}\\
&\leq \e^{\frac{q(p-1)}{p}}\bigg(\int_0^{2\pi} \big(\varphi(x/\e)\big)^{\frac{pq-q}{p-q}}\text{d}x\bigg)^{\frac{p-q}{p}}\cdot\|\nabla \omega\|^q_{L^p(\Omega_0\setminus\Omega_{\e})}\\
&\leq \e^{\frac{q(p-1)}{p}}\cdot \|\varphi\|_{\infty}^{\frac{q(p-1)}{p}}\cdot (2\pi)^{\frac{p-q}{p}}\cdot \|\nabla \omega\|^q_{L^p(\Omega_0\setminus\Omega_{\e})}.
\end{align*}

In the end, we get
\begin{align*}
\bigg(\int_0^{2\pi} |\omega(x,\e \varphi(x/\e))-\omega(x,0)|^q\text{d}x\bigg)^{1/q} &\leq C_{2,q}\cdot\e^{\frac{p-1}{p}}\cdot\|\nabla \omega\|_{L^p(\Omega_0\setminus\Omega_{\e})},
\end{align*}
with $C_{2,q}$ from \eqref{defn:C_{2,q}}. We conclude the proof by a classical density argument, using the embeddings $C^1(\overline{\Omega_0})\hookrightarrow W^{1,p}(\Omega_0)\hookrightarrow L^q(\partial\Omega_0)$.
\end{proof}

\begin{proposition}\label{prop:I_21}
Let $p\in(1,+\infty)$. For any $v\in W^{1,p}(\Omega_0)$, we have
\begin{align*}
|I_{21}|\leq C_{21}\cdot\e^{\frac{p-1}{p}}\cdot\|v\|_{W^{1,p}(\Omega_0\setminus\Omega_{\e})},
\end{align*}
where
\begin{align*}
I_{21}&=-\dfrac{w_0}{2}\int_0^{2\pi} (v(x,\e \varphi(x/\e))-v(x,0))\cdot Q^0_{\e}(x/\e)\cdot \gamma_{\e}(x/\e)\;\text{d}x
\end{align*}
and
\begin{align}\label{defn:C_21}
C_{21}=\dfrac{|w_0|\sqrt{2}}{4}\cdot (2\pi\|\varphi\|_{\infty})^{\frac{p-1}{p}}\cdot\sqrt{1+\|\varphi'\|_{\infty}^2}.
\end{align}
\end{proposition}

\begin{proof}
Let $1<q<p$ and $q'=\dfrac{q}{q-1}$. Then:
\begin{align*}
\dfrac{2}{|w_0|}|I_{21}|&\leq \int_0^{2\pi}\big|v(x,\e \varphi(x/\e))-v(x,0)\big|\cdot\bigg|Q^0_{\e}(x/\e)\cdot\sqrt{1+\big(\varphi'(x/\e)\big)^2}\bigg|\;\text{d}x\\
&\leq \bigg(\int_0^{2\pi}\big|v(x,\e \varphi(x/\e))-v(x,0)\big|^q\;\text{d}x\bigg)^{1/q}\cdot\\
&\hspace{10mm}\cdot\bigg(\int_0^{2\pi}\bigg|Q^0_{\e}(x/\e)\cdot\sqrt{1+\big(\varphi'(x/\e)\big)^2}\bigg|^{q'}\;\text{d}x\bigg)^{1/q'}\\
&\leq \bigg(\int_0^{2\pi}\big|v(x,\e \varphi(x/\e))-v(x,0)\big|^q\;\text{d}x\bigg)^{1/q}\cdot\\
&\hspace{10mm}\cdot\bigg(\int_0^{2\pi} \bigg(\dfrac{\sqrt{2}}{2}\bigg)^{q'}\cdot\bigg(\sqrt{1+\big(\varphi'(x/\e)\big)^2}\bigg)^{q'}\;\text{d}x\bigg)^{1/q'}\\
&\leq \dfrac{\sqrt{2}}{2}\cdot  \bigg(\int_0^{2\pi}\big|v(x,\e \varphi(x/\e))-v(x,0)\big|^q\;\text{d}x\bigg)^{1/q}\cdot \bigg(\int_0^{2\pi} \big(\sqrt{1+\|\varphi'\|_{\infty}^2}\big)^{q'}\;\text{d}x\bigg)^{1/q'}\\
&\leq \dfrac{\sqrt{2}}{2}\cdot (2\pi)^{1/q'}\cdot \sqrt{1+\|\varphi'\|_{\infty}^2}\cdot \bigg(\int_0^{2\pi}\big|v(x,\e \varphi(x/\e))-v(x,0)\big|^q\;\text{d}x\bigg)^{1/q},
\end{align*}
where we have used \eqref{bounds:gamma_eps} and that $|Q^0_{\e}(x/\e)|=\dfrac{\sqrt{2}}{2}$, for all $x\in[0,2\pi)$.

We apply now \Cref{lemma_C_2}, with the constant $C_{2,q}$ from \eqref{defn:C_{2,q}}, in order to obtain:
\begin{align*}
|I_{21}|&\leq C_{21}\cdot \e^{\frac{p-1}{p}}\cdot\|\nabla v\|_{L^p(\Omega_0\setminus\Omega_{\e})}.
\end{align*}
with
\begin{align*}
C_{21}=\dfrac{|w_0|\sqrt{2}}{4}\cdot (2\pi)^{\frac{q-1}{q}}\cdot \sqrt{1+\|\varphi'\|_{\infty}^2}\cdot C_{2,q}=\dfrac{|w_0|\sqrt{2}}{4}\cdot(2\pi\|\varphi\|_{\infty})^{\frac{p-1}{p}}\cdot\sqrt{1+\|\varphi'\|_{\infty}^2}.
\end{align*}
\end{proof}

\begin{proposition}\label{prop:I_31}
Let $p\in(2,+\infty)$. For any $v\in W^{1,p}(\Omega_0)$, we have:
\begin{align*}
|I_{31}|&\leq C_{31}\cdot\e^{\frac{p-1}{p}}\cdot\|v\|_{W^{1,p}(\Omega_0)},
\end{align*}
where
\begin{align*}
I_{31}&=\dfrac{w_0}{2}\int_0^{2\pi} \big(Q_0(x,\e \varphi(x/\e))\cdot v(x,\e \varphi(x/\e))-Q_0(x,0)\cdot v(x,0)\big)\cdot\gamma_{\e}(x/\e)\;\text{d}x
\end{align*}
and
\begin{align}\label{defn:C_31}
C_{31}=\dfrac{|w_0|}{2}\sqrt{1+\|\varphi'\|_{\infty}^2}\cdot(2\pi\|\varphi\|_{\infty})^{\frac{p-1}{p}}\cdot\|Q_0\|_{W^{1,\infty}(\Omega_0)}\cdot \big(|\Omega_0|^{1/p}\cdot C_{tr}(p,\Omega_0)+1\big).
\end{align}
\end{proposition}

\begin{proof}
Let $1<q<p$ and $q'=\dfrac{q}{q-1}$.

Using \eqref{bounds:gamma_eps}, we have:
\begin{align*}
&\dfrac{2}{|w_0|}|I_{31}| \leq \sqrt{1+\|\varphi'\|_{\infty}^2}\int_{0}^{2\pi}\big|\big(Q_0\cdot v\big)(x,\e \varphi(x/\e))-\big(Q_0\cdot v\big)(x,0)\big|\;\text{d}x\\
&\leq \sqrt{1+\|\varphi'\|_{\infty}^2}\cdot\bigg(\int_0^{2\pi}\big|Q_0(x,\e \varphi(x/\e))\big|\cdot\big|v(x,\e \varphi(x/\e))-v(x,0)\big|\;\text{d}x\bigg)+\\
&\hspace{5mm}+\sqrt{1+\|\varphi'\|_{\infty}^2}\cdot\bigg(\int_0^{2\pi}\big|v(x,0)\big|\cdot\big|Q_0(x,\e \varphi(x/\e))-Q_0(x,0)\big|\;\text{d}x\bigg)\\
&\leq \sqrt{1+\|\varphi'\|_{\infty}^2}\cdot\bigg(\int_0^{2\pi}\big|Q_0(x,\e \varphi(x/\e))\big|^{q'}\;\text{d}x\bigg)^{1/q'}\cdot\bigg(\int_0^{2\pi}\big|v(x,\e \varphi(x/\e))-v(x,0)\big|^q\;\text{d}x\bigg)^{1/q}+\\
&\hspace{5mm}+\sqrt{1+\|\varphi'\|_{\infty}^2}\cdot\bigg(\int_0^{2\pi}\big|v(x,0)\big|^p\;\text{d}x\bigg)^{1/p}\cdot\bigg(\int_0^{2\pi}\big|Q_0(x,\e \varphi(x/\e))-Q_0(x,0)\bigg|^{p'}\;\text{d}x\bigg)^{1/p'}
\end{align*}
where we have applied H\"{o}lder inequality with exponents $q$ and $q'$ for the first term and with exponents $p$ and $p'$ for the second one. 

For the first term, we apply \Cref{lemma_C_2} and use the $L^{\infty}(\Omega_0)$ bounds for $Q_0$ in order to obtain that:
\begin{align*}
&\bigg(\int_0^{2\pi}\big|Q_0(x,\e \varphi(x/\e))\big|^{q'}\;\text{d}x\bigg)^{1/q'}\cdot\bigg(\int_0^{2\pi}\big|v(x,\e \varphi(x/\e))-v(x,0)\big|^q\;\text{d}x\bigg)^{1/q}\leq \\
&\hspace{5mm}\leq (2\pi)^{1/q'}\cdot \|Q_0\|_{L^{\infty}(\Omega_0)}\cdot (2\pi)^{\frac{p-q}{pq}}\cdot \|\varphi\|_{\infty}^{\frac{p-1}{p}}\cdot \e^{\frac{p-1}{p}}\cdot\|\nabla v\|_{L^p(\Omega_0\setminus\Omega_{\e})}\\
&\hspace{5mm}\leq \big((2\pi\|\varphi\|_{\infty})^{\frac{p-1}{p}}\cdot\|Q_0\|_{L^{\infty}(\Omega_0)}\big)\cdot\e^{\frac{p-1}{p}}\cdot\|v\|_{W^{1,p}(\Omega_0)}
\end{align*}
since $\Omega_0\setminus\Omega_{\e}\subset\Omega_0$.

For the other term, we see that we can apply \Cref{lemma_C_2} with exponents $p'$ and $p$, since $p>2$ implies that $1<p'<p$, in order to obtain:
\begin{align*}
&\bigg(\int_0^{2\pi}\big|v(x,0)\big|^p\;\text{d}x\bigg)^{1/p}\cdot\bigg(\int_0^{2\pi}\big|Q_0(x,\e \varphi(x/\e))-Q_0(x,0)\bigg|^{p'}\;\text{d}x\bigg)^{1/p'}\leq\\
&\hspace{5mm}\leq \bigg(\int_{\Gamma_0}\big|v\big|^p\;\text{d}\sigma_0\bigg)^{1/p}\cdot (2\pi)^{\frac{p-p'}{pp'}}\cdot\|\varphi\|_{\infty}^{\frac{p-1}{p}}\cdot\e^{\frac{p-1}{p}}\cdot\|\nabla Q_0\|_{L^p(\Omega_0\setminus\Omega_{\e})}\\
&\hspace{5mm}\leq \big((2\pi\|\varphi\|_{\infty})^{\frac{p-1}{p}}\cdot\|\nabla Q_0\|_{L^{\infty}(\Omega_0)}\cdot|\Omega_0|^{1/p}\cdot C_{tr}(p,\Omega_0)\big)\cdot\e^{\frac{p-1}{p}}\cdot\|v\|_{W^{1,p}(\Omega_0)},
\end{align*}
where we use \Cref{defn:trace_ineq} and the fact that, if $1<p'<p$, then $\frac{p-p'}{pp'}<\frac{p}{pp'}=\frac{1}{p'}=\frac{p-1}{p}$.

In the end, we obtain that:
\begin{align*}
|I_{31}|&\leq C_{31}\cdot\e^{\frac{p-1}{p}}\cdot\|v\|_{W^{1,p}(\Omega_0)},
\end{align*}
where
\begin{align*}
C_{31}=\dfrac{|w_0|}{2}\sqrt{1+\|\varphi'\|_{\infty}^2}\cdot(2\pi\|\varphi\|_{\infty})^{\frac{p-1}{p}}\cdot\|Q_0\|_{W^{1,\infty}(\Omega_0)}\cdot \big(|\Omega_0|^{1/p}\cdot C_{tr}(p,\Omega_0)+1\big).
\end{align*}
\end{proof}

Let us now prove the following lemma.

\begin{lemma}\label{lemma:C_3}
Let us consider the case in which \Cref{pp:2} holds. Let $p\in(2,+\infty)$, \linebreak $\omega\in W^{1,p}(\Omega_0)$, $V$ be a Banach space and $b:\mathbb{R}\to V$ a $2\pi$-periodic function such that \linebreak $b\in L^{\infty}([0,2\pi))$, for which we write $\|b\|_{\infty}$ instead of $\|b\|_{L^{\infty}([0,2\pi))}$. Then:
\begin{align*}
\bigg|\int_0^{2\pi}\;\omega(x,0)\big(B-b(x/\e)\big)\;\text{d}x\bigg|\leq C_3\cdot \e^{\frac{p-1}{p}}\cdot\|\omega\|_{W^{1,p}(\Omega_0)},
\end{align*}
where $B=\displaystyle{\dfrac{1}{2\pi}\int_0^{2\pi}\;b(t)\;\text{d}t}$ and $C_3=(2\pi)^{\frac{2p-2}{p}}\cdot \|b\|_{\infty}\cdot C_{tr}(p,\Omega_0)$.
\end{lemma}

\begin{proof}
From \Cref{pp:2}, we have that $\e^{-1}=2k\in\mathbb{N}^*$. We write then:
\begin{align}\label{eqlemmaimp}
\int_0^{2\pi}\;\omega(x,0)\big(B-b(x/\e)\big)\;\text{d}x &=\sum_{j=0}^{2k-1}\int_{j\e\cdot 2\pi}^{(j+1)\e\cdot 2\pi}\;\omega(x,0)\big(B-b(x/\e)\big)\;\text{d}x\notag\\
&\hspace{-30mm}=\dfrac{1}{2\pi}\sum_{j=0}^{2k-1}\int_{j\e\cdot 2\pi}^{(j+1)\e\cdot 2\pi}\int_0^{2\pi}\;\omega(x,0)b(t)\;\text{d}t\;\text{d}x-\sum_{j=0}^{2k-1}\int_{j\e\cdot 2\pi}^{(j+1)\e\cdot 2\pi}\;\omega(x,0)b(x/\e)\;\text{d}x.
\end{align}

Using the change of variables $x=x'+j\e\cdot 2\pi$, we obtain:
\begin{align*}
\sum_{j=0}^{2k-1}\int_{j\e\cdot 2\pi}^{(j+1)\e\cdot 2\pi}\;\omega(x,0)b(x/\e)\;\text{d}x=\sum_{j=0}^{2k-1}\int_{0}^{\e\cdot 2\pi}\;\omega(x'+j\e\cdot 2\pi,0)b\bigg(\dfrac{x'+j\e\cdot 2\pi}{\e}\bigg)\;\text{d}x'
\end{align*}
and, since the function $b$ is $2\pi$-periodic and $j\in\mathbb{N}$, we can rewrite the last equality as:
\begin{align*}
\sum_{j=0}^{2k-1}\int_{j\e\cdot 2\pi}^{(j+1)\e\cdot 2\pi}\;\omega(x,0)b(x/\e)\;\text{d}x=\sum_{j=0}^{2k-1}\int_{0}^{\e\cdot 2\pi}\;\omega(x'+j\e\cdot 2\pi,0)b(x'/\e)\;\text{d}x'.
\end{align*}
Using now the change of variables $x'=\e t$, we get:
\begin{align*}
\sum_{j=0}^{2k-1}\int_{j\e\cdot 2\pi}^{(j+1)\e\cdot 2\pi}\;\omega(x,0)b(x/\e)\;\text{d}x=\sum_{j=0}^{2k-1}\int_{0}^{2\pi}\;\e\cdot \omega(\e t+j\e\cdot 2\pi,0)b(t)\;\text{d}t.
\end{align*}

Since
\begin{align*}
\e=\dfrac{1}{2\pi}\int_{j\e\cdot 2\pi}^{(j+1)\e\cdot 2\pi}\;1\;\text{d}x,
\end{align*}
then \eqref{eqlemmaimp} becomes:
\begin{align*}
&\int_0^{2\pi}\;\omega(x,0)\big(B-b(x/\e)\big)\;\text{d}x=\\
&=\dfrac{1}{2\pi}\sum_{j=0}^{2k-1}\int_{j\e\cdot 2\pi}^{(j+1)\e\cdot 2\pi}\int_0^{2\pi}\;\omega(x,0)b(t)\;\text{d}t\;\text{d}x-\dfrac{1}{2\pi}\sum_{j=0}^{2k-1}\int_{j\e\cdot 2\pi}^{(j+1)\e\cdot 2\pi}\int_{0}^{2\pi}\; \omega(\e t+j\e\cdot 2\pi,0)b(t)\;\text{d}t\;\text{d}x\\
&=\dfrac{1}{2\pi}\sum_{j=0}^{2k-1}\int_{j\e\cdot 2\pi}^{(j+1)\e\cdot 2\pi}\int_0^{2\pi}\;\big(\omega(x,0)-\omega(\e t+j\e\cdot 2\pi,0)\big)\cdot b(t)\;\text{d}t\;\text{d}x.
\end{align*}

Then
\begin{align*}
&\bigg|\int_0^{2\pi}\;\omega(x,0)\big(B-b(x/\e)\big)\;\text{d}x\bigg|\leq\dfrac{1}{2\pi}\sum_{j=0}^{2k-1}\int_{j\e\cdot 2\pi}^{(j+1)\e\cdot 2\pi}\int_0^{2\pi}\;\big|\omega(x,0)-\omega(\e t+j\e\cdot 2\pi,0)\big|\cdot \big|b(t)\big|\;\text{d}t\;\text{d}x\\
&\hspace{10mm}\leq \dfrac{\|b\|_{\infty}}{2\pi}\sum_{j=0}^{2k-1}\int_{j\e\cdot 2\pi}^{(j+1)\e\cdot 2\pi}\int_0^{2\pi}\;\big|\omega(x,0)-\omega(\e t+j\e\cdot 2\pi,0)\big|\;\text{d}t\;\text{d}x.
\end{align*}

We apply now the change of variables $t'=\e t+j\e\cdot 2\pi$ and drop the primes to obtain:
\begin{align*}
&\bigg|\int_0^{2\pi}\;\omega(x,0)\big(B-b(x/\e)\big)\;\text{d}x\bigg|\leq\dfrac{\|b\|_{\infty}}{2\pi}\sum_{j=0}^{2k-1}\int_{j\e\cdot 2\pi}^{(j+1)\e\cdot 2\pi}\int_{j\e\cdot 2\pi}^{(j+1)\e\cdot 2\pi}\;\dfrac{\big|\omega(x,0)-\omega(t,0)\big|}{\e}\;\text{d}t\;\text{d}x.
\end{align*}

Then:
\begin{align*}
&\sum_{j=0}^{2k-1}\int_{j\e\cdot 2\pi}^{(j+1)\e\cdot 2\pi}\int_{j\e\cdot 2\pi}^{(j+1)\e\cdot 2\pi}\e^{-1}\cdot \big|\omega(x,0)-\omega(t,0)\big|\;\text{d}t\;\text{d}x \;= \\
&\hspace{10mm}= \sum_{j=0}^{2k-1}\int_{j\e\cdot 2\pi}^{(j+1)\e\cdot 2\pi}\int_{j\e\cdot 2\pi}^{(j+1)\e\cdot 2\pi}\dfrac{\big|\omega(x,0)-\omega(t,0)\big|}{\big|(x,0)-(t,0)\big|}\cdot\dfrac{\big|(x,0)-(t,0)\big|}{\e}\;\text{d}t\;\text{d}x\\
&\hspace{10mm}\leq \sum_{j=0}^{2k-1}\bigg[\bigg(\int_{j\e\cdot 2\pi}^{(j+1)\e\cdot 2\pi}\int_{j\e\cdot 2\pi}^{(j+1)\e\cdot 2\pi}\dfrac{\big|\omega(x,0)-\omega(t,0)\big|^p}{\big|(x,0)-(t,0)\big|^p}\;\text{d}t\;\text{d}x\bigg)^{1/p}\cdot\\
&\hspace{20mm}\cdot \bigg(\int_{j\e\cdot 2\pi}^{(j+1)\e\cdot 2\pi}\int_{j\e\cdot 2\pi}^{(j+1)\e\cdot 2\pi}\dfrac{\big|x-t\big|^{p'}}{\e^{p'}}\;\text{d}t\;\text{d}x\bigg)^{1/p'}\bigg],
\end{align*}
where we have applied H\"{o}lder inequality with exponents $p$ and $p'$. Since $|x-t|\leq \e\cdot 2\pi$, for any $x,t\in[j\e\cdot 2\pi,(j+1)\e\cdot 2\pi]$, we obtain:
\begin{align*}
&\sum_{j=0}^{2k-1}\int_{j\e\cdot 2\pi}^{(j+1)\e\cdot 2\pi}\int_{j\e\cdot 2\pi}^{(j+1)\e\cdot 2\pi}\e^{-1}\cdot \big|w(x,0)-w(t,0)\big|\;\text{d}t\;\text{d}x \leq \\
&\hspace{5mm}\leq (2\pi)^{1+2/p'}\cdot \e^{2/p'}\cdot \sum_{j=0}^{2k-1}\bigg(\int_{j\e\cdot 2\pi}^{(j+1)\e\cdot 2\pi}\int_{j\e\cdot 2\pi}^{(j+1)\e\cdot 2\pi}\dfrac{\big|w(x,0)-w(t,0)\big|^p}{\big|(x,0)-(t,0)\big|^p}\;\text{d}t\;\text{d}x\bigg)^{1/p}\\
&\hspace{5mm}\leq (2\pi)^{1+2/p'}\cdot \e^{2/p'}\cdot \sum_{j=0}^{2k-1}\bigg(\int_{j\e\cdot 2\pi}^{(j+1)\e\cdot 2\pi}\int_0^{2\pi}\dfrac{\big|w(x,0)-w(t,0)\big|^p}{\big|(x,0)-(t,0)\big|^p}\;\text{d}t\;\text{d}x\bigg)^{1/p}.
\end{align*}

Let us denote now:
\begin{align*}
r_j=\int_{j\e\cdot 2\pi}^{(j+1)\e\cdot 2\pi}\int_0^{2\pi}\dfrac{\big|w(x,0)-w(t,0)\big|^p}{\big|(x,0)-(t,0)\big|^p}\;\text{d}t\;\text{d}x,\;\forall\;j\in\{0,1,\ldots,2k-1\}.
\end{align*}

We have that $r_j\geq 0$, for all $j\in\{0,1,\ldots,2k-1\}$. The function $\mathbb{R}\ni x\to x^{1/p}$ is concave since $p\in(2,+\infty)$, therefore we have the Jensen inequality:
\begin{align*}
\sum_{j=0}^{2k-1} r_j^{1/p}\leq 2k\cdot\bigg(\dfrac{1}{2k}\sum_{j=0}^{2k-1} r_j\bigg)^{1/p}.
\end{align*}

Since $2k=\e^{-1}$ and
\begin{align*}
\bigg(\sum_{j=0}^{2k-1} r_j\bigg)^{1/p}=\bigg(\int_0^{2\pi}\int_0^{2\pi}\dfrac{|\omega(x,0)-\omega(t,0)\big|^p}{|x-t|^p}\;\text{d}t\;\text{d}x\bigg)^{1/p}\leq \|\omega\|_{W^{1-1/p,p}(\Gamma_1)},
\end{align*}
due to the fact that the left hand side of the last inequality represents the Gagliardo seminorm defined on the space $W^{1-1/p,p}(\Gamma_1)$, we obtain that
\begin{align*}
&\sum_{j=0}^{2k-1}\int_{j\e\cdot 2\pi}^{(j+1)\e\cdot 2\pi}\int_{j\e\cdot 2\pi}^{(j+1)\e\cdot 2\pi}\e^{-1}\cdot \big|w(x,0)-w(t,0)\big|\;\text{d}t\;\text{d}x \leq \\
&\hspace{5mm} \leq (2\pi)^{\frac{3p-2}{p}}\cdot \e^{\frac{2p-2}{p}}\cdot (2k)^{1-1/p}\cdot\|\omega\|_{W^{1-1/p,p}(\Gamma_1)}\\
&\hspace{5mm}\leq (2\pi)^{\frac{3p-2}{p}}\cdot \e^{\frac{2p-2}{p}-1+\frac{1}{p}}\cdot\|\omega\|_{W^{1-1/p,p}(\Gamma_1)}\\
&\hspace{5mm}\leq (2\pi)^{\frac{3p-2}{p}}\cdot \e^{\frac{p-1}{p}}\cdot\|\omega\|_{W^{1-1/p,p}(\Gamma_1)}.
\end{align*}

Using \Cref{defn:trace_ineq}, we obtain that:
\begin{align*}
&\bigg|\int_0^{2\pi}\;\omega(x,0)\big(B-b(x/\e)\big)\;\text{d}x\bigg|\leq C_3\cdot \e^{\frac{p-1}{p}}\cdot\|\omega\|_{W^{1,p}(\Omega_0)}
\end{align*}
with
\begin{align*}
C_3=(2\pi)^{\frac{2p-2}{p}}\cdot\|b\|_{\infty}\cdot C_{tr}(p,\Omega_0).
\end{align*}
\end{proof}

\begin{proposition}\label{prop:I_32}
Let $p\in (2,+\infty)$. For any $v\in W^{1,p}(\Omega_0)$ we have:
\begin{align*}
|I_{32}|\leq C_{32}\cdot\e^{\frac{p-1}{p}}\cdot\|v\|_{W^{1,p}(\Omega_0)},
\end{align*}
where
\begin{align*}
I_{32}&=-\dfrac{w_0}{2}\int_0^{2\pi} Q_0(x,0)\cdot v(x,0)\cdot \bigg(\gamma-\gamma_{\e}(x/\e)\bigg)\;\text{d}x
\end{align*}
and
\begin{align}\label{defn:C_32}
C_{32}=\dfrac{|w_0|}{2}\cdot (2\pi)^{\frac{2p-2}{p}}\cdot \sqrt{1+\|\varphi'\|_{\infty}^2}\cdot C_{tr}(p,\Omega_0)\cdot\|Q_0\|_{W^{1,\infty}(\Omega_0)}.
\end{align}
\end{proposition}

\begin{proof}
Let $v\in W^{1,p}(\Omega_0)$. Since $Q_0\in W^{1,\infty}(\Omega_0)$, we have that $Q_0\cdot v\in W^{1,p}(\Omega_0)$. We apply \Cref{lemma:C_3} for $\omega=Q_0\cdot v$ and $b=\gamma_{\e}$, with $V=\mathbb{R}$. In this way, we obtain that:
\begin{align*}
|I_{32}|\leq \dfrac{|w_0|}{2}\cdot (2\pi)^{\frac{2p-2}{p}}\cdot \|\gamma_{\e}\|_{\infty}\cdot C_{tr}(p,\Omega_0)\cdot \e^{\frac{p-1}{p}}\cdot\|Q_0\cdot v\|_{W^{1,p}(\Omega_0)}.
\end{align*}
Using now that $\|\gamma_{\e}\|_{\infty}=\sqrt{1+\|\varphi'\|_{\infty}^2}$ and that $\|Q_0\cdot v\|_{W^{1,p}(\Omega_0)}\leq \|Q_0\|_{W^{1,\infty}(\Omega_0)}\cdot\|v\|_{W^{1,p}(\Omega_0)}$, we obtain the conclusion.
\end{proof}

\begin{proposition}\label{prop:I_22}
Let $p\in(2,+\infty)$. For any $v\in W^{1,p}(\Omega_0)$ we have:
\begin{align*}
|I_{22}|\leq C_{22}\cdot\e^{\frac{p-1}{p}}\cdot\|v\|_{W^{1,p}(\Omega_0)},
\end{align*}
where
\begin{align*}
I_{22}=\dfrac{w_0}{2}\int_0^{2\pi}v(x,0)\cdot\bigg(\gamma Q_{ef}-\gamma_{\e}(x/\e)Q_{\e}^0(x/\e)\bigg)\;\text{d}x
\end{align*}
and
\begin{align}\label{defn:C_22}
C_{22}=\dfrac{|w_0|\sqrt{2}}{4}\cdot (2\pi)^{\frac{2p-2}{p}}\cdot\sqrt{1+\|\varphi'\|_{\infty}^2}\cdot C_{tr}(p,\Omega_0).
\end{align}
\end{proposition}

\begin{proof}
Let $v\in W^{1,p}(\Omega_0)$. We apply \Cref{lemma:C_3} for $\omega=v$ and $b=\gamma_\e Q_{\e}^0$, with $V=\text{Sym}_0(2)$. In this way, we have:
\begin{align*}
|I_{22}|\leq \dfrac{|w_0|}{2}\cdot (2\pi)^{\frac{2p-2}{p}}\cdot \|Q_{\e}^0\cdot\gamma_{\e}\|_{\infty}\cdot C_{tr}(p,\Omega_0)\cdot \e^{\frac{p-1}{p}}\cdot \|v\|_{W^{1,p}(\Omega_0)}.
\end{align*}
Since $|Q_{\e}^0|=\dfrac{\sqrt{2}}{2}$ and $\|\gamma_{\e}\|_{\infty}=\sqrt{1+\|\varphi'\|_{\infty}^2}$, we obtain the conclusion.
\end{proof}

We are now able to prove \Cref{prop:ineq_for_W_1_p_Omega_0}.

\begin{proof}[Proof of \Cref{prop:ineq_for_W_1_p_Omega_0}]
We combine \eqref{eq:sum_of_integrals} with \Cref{prop:I_1,prop:I_21,prop:I_31,prop:I_32,prop:I_22} to obtain the conclusion, with, for example,
\begin{align*}
C_I=C_1+C_{21}+C_{31}+C_{22}+C_{32},
\end{align*}
where these constants are defined in \eqref{defn:C_1}, \eqref{defn:C_21}, \eqref{defn:C_31}, \eqref{defn:C_32} and \eqref{defn:C_22}.
\end{proof}

\begin{remark}
We can actually choose $C_I$ of the following form:
\begin{align*}
C_I=\text{max}\big\{1,\|\varphi\|_{\infty}^{(p-1)/p}\big\}\cdot\sqrt{1+\|\varphi'\|_{\infty}^2}\cdot C(w_0,c,p,\Omega_0,Q_0),
\end{align*}
where $C(w_0,c,p,\Omega_0,Q_0)$ is an $\e$-independent constant depending only on $w_0$, $c$, $p$, $\Omega_0$ and $Q_0$.
\end{remark}

\subsection{Proof of the error estimate}\label{subsect:proof_err_est}

The goal in this subsection would be to place instead of $v$ in the right-hand side of the inequality from \Cref{prop:ineq_for_W_1_p_Omega_0} something that depends on $(Q_0-Q_{\e})$, such that we can obtain \Cref{thm:main}.

Throughout this subsection, we fix $p\in(2,+\infty)$.

\begin{definition}
Let $u_{\e}:=Q_0-Q_{\e}$.
\end{definition}

\begin{remark}
The function $u_{\e}$ solves the following PDE:
\begin{align*}
\begin{cases}
-\Delta u_{\e}+cu_{\e}=0,\;\text{in}\;\Omega_{\e}\\
\\
\dfrac{\partial u_{\e}}{\partial \nu_{\e}}+\dfrac{w_0}{2}u_{\e}=g_{\e},\;\text{on}\;\Gamma_{\e}\\\
\\
\dfrac{\partial u_{\e}}{\partial \nu_R}+\dfrac{w_0}{2}u_{\e}=0,\;\text{on}\;\Gamma_R
\end{cases}
\end{align*}
where $g_{\e}=\dfrac{\partial Q_0}{\partial \nu_{\e}}+\dfrac{w_0}{2}\big(Q_0-Q_{\e}^0\big)$. Since $\Gamma_{\e}\subset\Omega_0$, we have no problems with defining $g_{\e}$. By \Cref{cor:Q_0,cor:Q_eps}, we also have that $u_{\e}\in W^{2,p}(\Omega_{\e})$ and $g_{\e}\in W^{1-1/p,p}(\Gamma_{\e})$, for any $1<p<+\infty$.
\end{remark}

We would like now to prove the following proposition (to be compared with \Cref{prop:ineq_for_W_1_p_Omega_0}):

\begin{proposition}\label{prop:ineq_Omega_eps_rhs}
There exists an $\e$-independent constant $C_0$ such that:
\begin{align*}
\big|a_{\e}(Q_0-Q_{\e},v)\big|\leq C_0\cdot\e^{\frac{p-1}{p}}\cdot\|v\|_{W^{1,p}(\Omega_{\e})},\;\forall v\in W^{1,p}(\Omega_{\e}).
\end{align*}
\end{proposition}

In order to do so, we need to construct an extension operator $E_{\e}:W^{1,p}(\Omega_{\e})\to W^{1,p}(\Omega_0)$ such that $E_{\e}\omega\equiv\omega$ in $\Omega_{\e}$ for any $\omega\in W^{1,p}(\Omega_{\e})$, with the operator $E_{\e}$ bounded uniformly in $\e$. For this, we adapt the bi-Lipschitz maps $\Phi_\e$ from \eqref{eqDefBilip} and $\Phi^{-1}_\e$ from \eqref{eqInverseBilip} to this simplified model. In \eqref{eqDefBilip} and \eqref{eqInverseBilip}, the transformations $\Phi_\e$ and $\Phi_\e^{-1}$ are only between $\Omega_\e$ and $\Omega_0$. In order to construct the desired extension, we introduce the following notation.

\begin{definition}
Let $\Omega_1^{\e}=\bigg\{(x,y)\;\bigg|\;x\in[0,2\pi),\;y\in\bigg(-\dfrac{R\e\varphi(x/\e)}{R-\e\varphi(x/\e)},R\bigg)\bigg\}$.
\end{definition}

\begin{remark}
Using \Cref{pp:2}, it is easy to see that $\Omega_1^{\e}$ is well defined, since $R-\e\varphi(x/\e)>0$ for all $x\in[0,2\pi)$.
\end{remark}

\begin{definition}\label{defn:new_Phi_eps}
We define $\Phi_\e:\Omega_1^\e\to\Omega_0$ as
\begin{align*}
\Phi_\e(x,y)=\bigg(x,y\cdot\dfrac{R-\e\varphi(x/\e)}{R}+\e\varphi(x/\e)\bigg),\;\forall(x,y)\in\Omega_1^{\e}
\end{align*}
and $\Phi_\e^{-1}:\Omega_0\to\Omega_1^{\e}$ as
\begin{align*}
\Phi_\e^{-1}(x,y)=\bigg(x,\dfrac{R\big(y-\e\varphi(x/\e)\big)}{R-\e\varphi(x/\e)}\bigg),\;\forall(x,y)\in\Omega_0.
\end{align*}
\end{definition}

\begin{remark}
We have $\Phi_\e(\Omega_0)=\Omega_\e$ and, using \Cref{pp:2}, we can prove the following sequence of inclusions:
\begin{align*}
\{(x,y)\;|\;x\in[0,2\pi),\;y\in(R/2,R)\}\subset \Omega_{\e}\subset \Omega_0\subset \Omega_1^{\e}\subset \{(x,y)\;|\;x\in[0,2\pi),\;y\in(-R,R)\}.
\end{align*}
\end{remark}

\begin{proposition}\label{prop:properties_of_Phi_e}
$\Phi_\e$ defines a family of $C^2$ uniformly bi-Lipschitz maps between $\Omega_1$ and $\Omega_0$. Moreover, there exists an $\e$-independent constant $C_\Phi$ such that:
\begin{align}\label{ineq:C_Phi_eps_0}
C_\Phi^{-1}\cdot \|\omega\|_{W^{1,p}(\Omega_{\e})}\leq \|\tilde{\omega}\|_{W^{1,p}(\Omega_0)}\leq C_\Phi\cdot \|\omega\|_{W^{1,p}(\Omega_{\e})},
\end{align}
for all $\omega\in W^{1,p}(\Omega_{\e})$, where $\tilde{\omega}=\omega\circ \Phi_{\e}\big|_{\Omega_0}$, and
\begin{align}\label{ineq:C_Phi_0_1}
C_\Phi^{-1}\cdot \|\omega\|_{W^{1,p}(\Omega_{0})}\leq \|\tilde{\omega}\|_{W^{1,p}(\Omega_1^{\e})}\leq C_\Phi\cdot \|\omega\|_{W^{1,p}(\Omega_{0})},
\end{align}
for all $\omega\in W^{1,p}(\Omega_0)$, where $\tilde{\omega}=\omega\circ \Phi_{\e}$.
\end{proposition}

\begin{proof}
Since $\varphi\in C^2$, then $\Phi_\e\in C^2$. Moreover, since the definition of $\Phi_\e$ from \Cref{defn:new_Phi_eps} is based on \eqref{eqDefBilip}, then one can argue similarly as in \Cref{propBilipInverse} to obtain that $\Phi_{\e}$ and its inverse are Lipschitz with Lipschitz constant independent of $\e$. More specifically, using the bound
\begin{align*}
0<\e\cdot \|\varphi\|_{\infty}<\dfrac{R}{2} 
\end{align*}
from \Cref{pp:1}, we can obtain that 
\begin{align*}
\bigg|\dfrac{\partial \Phi_\e}{\partial x}(x,y)\bigg|\leq \text{max}\{2R,1\}\cdot\sqrt{1+\|\varphi'\|_{\infty}^2}\;\;\;\text{and}\;\;\;\bigg|\dfrac{\partial \Phi_\e}{\partial y}(x,y)\bigg|\leq 1,\;\forall(x,y)\in\Omega_1^{\e},
\end{align*}
which implies that $\Phi_\e$ is Lipschitz with its Lipschitz constant bounded $\e$-independent. In the same way, the first order derivatives of $\Phi_\e^{-1}$ can be bounded $\e$-independent (using the same bound as above for $\e$). To obtain the constant $C_\Phi$, we use the same $\e$-independent bounds for the first order derivatives of $\Phi_{\e}$ and $\Phi_\e^{-1}$ when we apply chain rule in $\tilde{\omega}=\omega\circ \Phi_\e$ with $\omega\in W^{1,p}(\Omega_0)$.
\end{proof}

\begin{remark}
If we were to define $E_{\e}\omega=\omega\circ \Phi_\e\big|_{\Omega_0}$, for any $\omega\in W^{1,p}(\Omega_{\e})$, then we would not have had $E_{\e}\omega\equiv \omega$ inside of $\Omega_{\e}$, so we need the more sophisticated extension that will be provided in \Cref{defn:uniform_ext} below.
\end{remark}

\begin{corollary}\label{corollary:extension_0}
Let $\Omega_{2}:=\{(x,y)\;|\;x\in[0,2\pi),\;y\in(-R,R)\}$ and let $T:W^{1,p}(\Omega_0)\to W^{1,p}(\Omega_{2})$ the following extension operator:
\begin{align*}
T\omega(x,y)=\begin{cases}
\omega(x,y),\;\text{if}\;y\in(0,R);\\
\\
\omega\big(x,-y\big),\;\text{if}\;y\in(-R,0);
\end{cases}
\end{align*}
for any $\omega\in W^{1,p}(\Omega_0)$. Then there exists an $\e$-independent constant $C_{ext}(p,\Omega_0)>0$ such that:
\begin{align*}
\|T\omega\|_{W^{1,p}(\Omega_{2})}\leq C_{ext}(p,\Omega_0)\cdot\|\omega\|_{W^{1,p}(\Omega_0)},\;\forall\;\omega\in W^{1,p}(\Omega_0).
\end{align*}
\end{corollary}

\begin{remark}
The proof is a simple exercise which consists of applying the method of extending a Sobolev function by reflection against a flat boundary, illustrated in \cite{evans2010second} for $W^{1,p}(\Omega_0)$ functions.
\end{remark}

\begin{definition}\label{defn:uniform_ext}
Let $E_{\e}:W^{1,p}(\Omega_{\e})\to W^{1,p}(\Omega_0)$ defined as
\begin{align*}
E_{\e}\omega:=\bigg(\big(T(\omega\circ \Phi_\e\big|_{\Omega_0})\big)\bigg|_{\Omega_1^{\e}}\circ \Phi_\e^{-1}\bigg),
\end{align*}
for any $\omega\in W^{1,p}(\Omega_\e)$. In this way, $E_{\e}\omega\equiv \omega$ in $\Omega_{\e}$, for any $\omega\in W^{1,p}(\Omega_{\e})$.
\end{definition}

\begin{proposition}\label{prop:uniform_ext}
There exists an $\e$-independent constant $C_{ext}$ such that:
\begin{align*}
\|E_{\e}\omega\|_{W^{1,p}(\Omega_0)}\leq C_{ext}\cdot\|\omega\|_{W^{1,p}(\Omega_{\e})},\;\forall\;\omega\in W^{1,p}(\Omega_{\e}).
\end{align*}
\end{proposition}

\begin{proof}
Let $\omega\in W^{1,p}(\Omega_{\e})$ and $\tilde{\omega}=\omega\circ\Phi_\e\big|_{\Omega_0}$. Since the transformation $\Phi_{\e}^{-1}$ is bi-Lipschitz with its constants bounded $\e$-independent, then, using \eqref{ineq:C_Phi_0_1}:
\begin{align*}
\|E_{\e}\omega\|_{W^{1,p}(\Omega_0)}=\bigg\|\bigg(\big(T\tilde{\omega}\big)\bigg|_{\Omega_1^{\e}}\circ \Phi_\e^{-1}\bigg)\bigg\|_{W^{1,p}(\Omega_0)}\leq C_\Phi^{-1}\bigg\|\big(T\tilde{\omega}\big)\bigg|_{\Omega_1^{\e}}\bigg\|_{W^{1,p}(\Omega_1^{\e})}\leq C_\Phi^{-1}\big\|T\tilde{\omega}\big\|_{W^{1,p}(\Omega_2)},
\end{align*}
where in the last inequality we use that $T\tilde{\omega}\bigg|_{\Omega_1^{\e}}$ is a restriction of $T\tilde{\omega}$ from $\Omega_2$. Then:
\begin{align*}
\|E_{\e}\omega\|_{W^{1,p}(\Omega_0)}\leq C_\Phi^{-1}\|T\tilde{\omega}\|_{W^{1,p}(\Omega_2)}\leq C_\Phi^{-1}\cdot C_{ext}(p,\Omega_0)\cdot \|\tilde{\omega}\|_{W^{1,p}(\Omega_0)},
\end{align*}
where we have used \Cref{corollary:extension_0}. Since $\tilde{\omega}=\omega\circ \Phi_\e\big|_{\Omega_0}$, then, using \eqref{ineq:C_Phi_eps_0}, we obtain:
\begin{align*}
\|E_{\e}\omega\|_{W^{1,p}(\Omega_0)}\leq C_{ext}(p,\Omega_0)\cdot \|\omega\|_{W^{1,p}(\Omega_{\e})}.
\end{align*}
Therefore, we can actually choose $C_{ext}=C_{ext}(p,\Omega_0)$ given by \Cref{corollary:extension_0}.
\end{proof}

\begin{proof}[Proof of \Cref{prop:ineq_Omega_eps_rhs}]
Let $v\in W^{1,p}(\Omega_{\e})$. Then $E_{\e}v\in W^{1,p}(\Omega_0)$ and we can apply \Cref{prop:ineq_for_W_1_p_Omega_0} to obtain:
\begin{align*}
|a_{\e}(u_{\e},E_{\e}v)|\leq C_I\cdot \e^{\frac{p-2}{p}}\cdot \|E_{\e}v\|_{W^{1,p}(\Omega_0)}.
\end{align*}
Using \Cref{defn:uniform_ext} and \Cref{prop:uniform_ext}, we obtain:
\begin{align*}
\big|a_{\e}(u_{\e},v)\big|\leq \big(C_I\cdot C_{ext}\big)\cdot \e^{\frac{p-1}{p}}\cdot \|v\|_{W^{1,p}(\Omega_{\e})}.
\end{align*}
\end{proof}

\begin{corollary}
There exists a unique solution $v_{\e}\in W^{2,p}(\Omega_{\e})$ that solves the following PDE:
\begin{align}\label{defn:PDE_v}
\begin{cases}
-\Delta v_{\e}+cv_{\e}=u_{\e},\;\text{in}\;\Omega_{\e}\\
\dfrac{\partial v_{\e}}{\partial \nu_{\e}}+\dfrac{w_0}{2}v_{\e}=0,\;\text{on}\;\Gamma_{\e}\\
\dfrac{\partial v_{\e}}{\partial\nu_R}+\dfrac{w_0}{2}v_{\e}=0,\;\text{on}\;\Gamma_R
\end{cases}
\end{align}
where $u_{\e}=Q_0-Q_{\e}$.
\end{corollary}

\begin{proof}
The proof follows the same steps as in \Cref{cor:Q_eps}.
\end{proof}

\begin{proposition}\label{prop:properties_of_v_eps}
The function $v_{\e}$ satisfies the following inequality: 
\begin{align}
\|v_{\e}\|_{H^1(\Omega_{\e})}\leq \text{min}\{c^{-1/2},c^{-1}\}\cdot \|u_{\e}\|_{L^2(\Omega_{\e})},
\end{align}
where $c$ is the positive constant from the bulk energy defined in \eqref{eq:free_energy_functional}.
\end{proposition}

\begin{proof}
Let $w\in W^{1,p}(\Omega_{\e})$. Since $-\Delta v_{\e}+cv_{\e}=u_{\e}$ in $\Omega_{\e}$, we have:
\begin{align*}
\int_{\Omega_{\e}}u_{\e}\cdot w\;\text{d}(x,y)&=\int_{\Omega_{\e}}\big(-\Delta v_{\e}+cv_{\e}\big)\cdot w\;\text{d}(x,y) \\
&=\int_{\Omega_{\e}}\big(\nabla v_{\e}\cdot \nabla w+cv_{\e}\cdot w\big)\;\text{d}(x,y)-\int_{\partial \Omega_{\e}}\dfrac{\partial v_{\e}}{\partial \nu_{\e}}\cdot w\;\text{d}\sigma_{\e}\\
&=\int_{\Omega_{\e}}\big(cv_{\e}\cdot w+\nabla v_{\e}\cdot\nabla w\big)\;\text{d}(x,y)+\dfrac{w_0}{2}\int_{\Gamma_{\e}}v_{\e}\cdot w\;\text{d}\sigma_{\e} +\dfrac{w_0}{2}\int_{\Gamma_R}v_{\e}\cdot w\;\text{d}\sigma_R
\end{align*}

Taking $w=u_{\e}$, we obtain:
\begin{align}\label{eq:a_eps_u_eps_v_eps}
a_{\e}(u_{\e},v_{\e})&=\int_{\Omega_{\e}}\big(cu_{\e}\cdot v_{\e}+\nabla u_{\e}\cdot\nabla v_{\e}\big)\;\text{d}(x,y) +\dfrac{w_0}{2}\int_{\Gamma_{\e}}u_{\e}\cdot v_{\e}\;\text{d}\sigma_{\e} +\dfrac{w_0}{2}\int_{\Gamma_R}u_{\e}\cdot v_{\e}\;\text{d}\sigma_R \notag\\
&=\|u_{\e}\|^2_{L^2(\Omega_{\e})}.
\end{align}

Taking $w=v_{\e}$, we obtain:
\begin{align*}
c\|v_{\e}\|^2_{L^2(\Omega_{\e})}+\|\nabla v_{\e}\|^2_{L^2(\Omega_{\e})}+\dfrac{w_0}{2}\|v_{\e}\|^2_{L^2(\Gamma_{\e})}+\dfrac{w_0}{2}\|v_{\e}\|^2_{L^2(\Gamma_R)}&=\int_{\Omega_{\e}}u_{\e}\cdot v_{\e}\;\text{d}x.
\end{align*}

Now we can see that
\begin{align*}
c\|v_{\e}\|^2_{L^2(\Omega_{\e})}\leq c\|v_{\e}\|^2_{L^2(\Omega_{\e})}+\|\nabla v_{\e}\|^2_{L^2(\Omega_{\e})}+\dfrac{w_0}{2}\|v_{\e}\|^2_{L^2(\partial\Omega_{\e})} \leq \|u_{\e}\|_{L^2(\Omega_{\e})}\|v_{\e}\|_{L^2(\Omega_{\e})}
\end{align*}
which implies that $\|v_{\e}\|_{L^2(\Omega_{\e})}\leq c^{-1}\cdot \|u_{\e}\|_{L^2(\Omega_{\e})}$. In the same way,
\begin{align*}
\|\nabla v_{\e}\|^2_{L^2(\Omega_{\e})}\leq c\|v_{\e}\|^2_{L^2(\Omega_{\e})}+\|\nabla v_{\e}\|^2_{L^2(\Omega_{\e})}+\dfrac{w_0}{2}\|v_{\e}\|^2_{L^2(\partial\Omega_{\e})} \leq \|u_{\e}\|_{L^2(\Omega_{\e})}\|v_{\e}\|_{L^2(\Omega_{\e})}\leq c^{-1}\cdot \|u_{\e}\|^2_{L^2(\Omega_{\e})}
\end{align*}
using the last inequality proved. This implies that $\|\nabla v_{\e}\|_{L^2(\Omega_{\e})}\leq c^{-1/2}\cdot \|u_{\e}\|_{L^2(\Omega_{\e})}$. In the end, we obtain that:
\begin{align*}
\|v_{\e}\|_{H^1(\Omega_{\e})}\leq \text{min}\{c^{-1/2},c^{-1}\}\cdot \|u_{\e}\|_{L^2(\Omega_{\e})}.
\end{align*}
\end{proof}

\begin{definition}\label{defn:v_eps_on_annulus}
Let $\overline{v}_{\e}:=v_{\e}\circ \Phi_\e\big|_{\Omega_0}\circ \Phi^{-1}$, where $\Phi_{\e}$ is introduced in \Cref{defn:new_Phi_eps} and $\Phi$ in \Cref{defn:THE_transf}. Since $\Phi_{\e}(\Omega_0)=\Omega_\e$, then $\Omega_{\e}=\big(\Phi_{\e}\big|_{\Omega_0}\circ \Phi_{^-1}\big)(\mathcal{U}_0)$.
\end{definition}

\begin{corollary}
We have that $\overline{v}_{\e}\in W^{2,p}(\mathcal{U}_0)$, for any $p>2$.
\end{corollary}

\begin{proof}
Since $v_{\e}\in W^{2,p}(\Omega_{\e})$, $\Phi_{\e}$ and $\Phi$ are bi-Lipschitz with $\Phi\in C^2(\Omega_0)$ and $\Phi^{-1}$ smooth, we obtain the conclusion.
\end{proof}

\begin{remark}
Using \Cref{defn:v_eps_on_annulus}, we can see that $\overline{v}_{\e}$ solves a PDE of the form:
\begin{align}\label{defn:PDE_tilde_v_eps}
\begin{cases}
\mathcal{L}\overline{v}_{\e}+c\overline{v}_{\e}=\overline{u}_{\e},\;\text{in}\;\mathcal{U}_0\\
\\
\nabla\overline{v}_{\e}\cdot \overline{l_1}+\dfrac{w_0}{2}\overline{v}_{\e}=0,\;\text{on}\;\Phi(\Gamma_0)\\
\\
\nabla\overline{v}_{\e}\cdot \overline{l_2}+\dfrac{w_0}{2}\overline{v}_{\e}=0,\;\text{on}\;\Phi(\Gamma_R)
\end{cases}
\end{align}
where $\mathcal{L}$ is an uniformly elliptic operator, $\overline{l_1}\in C^1(\Phi(\Gamma_0))$, $\overline{l_2}\in C^1(\Phi(\Gamma_R))$ and $\overline{u_{\e}}:=u_{\e}\circ \Phi_\e\big|_{\Omega_0}\circ \Phi^{-1}\in W^{2,p}(\mathcal{U}_0)$.
\end{remark}

\begin{proposition}\label{prop:regularity}
There exists an $\e$-independent constant $C_{reg}(\mathcal{U}_0)$ such that $\overline{v}_{\e}$ satisfies the following inequality:
\begin{align*}
\|\overline{v}_{\e}\|_{H^2(\mathcal{U}_0)}\leq C_{reg}(\mathcal{U}_0)\cdot\big(\|\mathcal{L}\overline{v}_{\e}\|_{L^2(\mathcal{U}_0)}+\|\overline{v}_{\e}\|_{H^{1/2}(\partial\mathcal{U}_0)}+\|\overline{v}_{\e}\|_{H^1(\mathcal{U}_0)}\big).
\end{align*}
\end{proposition}

\begin{proof}
We apply \cite[Theorem 2.3.3.2]{grisvard1985non}, since all the required conditions are satisfied.
\end{proof}

We can now prove the main result of this section.

\begin{proof}[Proof of \Cref{thm:main}]

We apply first \Cref{prop:ineq_Omega_eps_rhs} with $v_{\e}\in W^{2,p}(\Omega_{\e})$:
\begin{align*}
\|u_{\e}\|^2_{L^2(\Omega_{\e})}=a_{\e}(u_{\e},v_{\e})\leq C_0\cdot\e^{\frac{p-1}{p}}\cdot\|v_{\e}\|_{W^{1,p}(\Omega_{\e})},
\end{align*}
where we also use \eqref{eq:a_eps_u_eps_v_eps}. We apply now \Cref{prop:properties_of_Phi_e}:
\begin{align*}
\|u_{\e}\|^2_{L^2(\Omega_{\e})}&\leq C_0\cdot\e^{\frac{p-1}{p}}\cdot\|v_{\e}\|_{W^{1,p}(\Omega_{\e})}\leq \big(C_0\cdot C_\Phi\big)\cdot\e^{\frac{p-1}{p}}\cdot\big\|v_{\e}\circ \Phi_\e\big|_{\Omega_0}\big\|_{W^{1,p}(\Omega_0)}.
\end{align*}

We now use the compact embedding $W^{2,2}(\Omega_0)\hookrightarrow\hookrightarrow W^{1,p}(\Omega_0)$ to obtain:
\begin{align*}
\|u_{\e}\|^2_{L^2(\Omega_{\e})}&\leq \big(C_0\cdot C_\Phi\cdot C_{emb}(\Omega_0)\big)\cdot\e^{\frac{p-1}{p}}\cdot\big\|v_{\e}\circ \Phi_\e\big|_{\Omega_0}\big\|_{W^{2,2}(\Omega_0)},
\end{align*}
where $C_{emb}(\Omega_0)$ is the constant given by the compact embbeding used.

Recalling \Cref{defn:THE_transf}, it is easy to see that there exists $C_{polar}>0$, which is $\e$-independent, such that:
\begin{align}\label{eq:ineq_C_polar}
C_{polar}^{-1}\|\omega\|_{W^{2,2}(\Omega_0)}\leq \big\|\omega\circ \Phi^{-1}\big\|_{W^{2,2}(\mathcal{U}_0)}\leq C_{polar}\|\omega\|_{W^{2,2}(\Omega_0)},\;\forall\;\omega\in W^{2,2}(\Omega_0).
\end{align}

In this way, we have:
\begin{align*}
\|u_{\e}\|^2_{L^2(\Omega_{\e})}&\leq C'_0\cdot\e^{\frac{p-1}{p}}\cdot\|\overline{v}_{\e}\|_{W^{2,2}(\mathcal{U}_0)},
\end{align*}
where $C'_0=C_0\cdot C_\Phi\cdot C_{emb}(\Omega_0)\cdot C_{polar}$.

We apply now \Cref{prop:regularity}:
\begin{align*}
\|u_{\e}\|^2_{L^2(\Omega_{\e})}&\leq \big(C'_0\cdot C_{reg}(\mathcal{U}_0)\big)\cdot\e^{\frac{p-1}{p}}\cdot\big(\|\mathcal{L}\overline{v}_{\e}\|_{L^2(\mathcal{U}_0)}+\|\overline{v}_{\e}\|_{H^{1/2}(\partial\mathcal{U}_0)}+\|\overline{v}_{\e}\|_{H^1(\mathcal{U}_0)}\big).
\end{align*}

Using \eqref{defn:PDE_tilde_v_eps} and the trace inequality for the trace operator $\text{Tr}:H^1(\mathcal{U}_0)\to H^{1/2}(\partial\mathcal{U}_0)$, we get:
\begin{align*}
\|u_{\e}\|^2_{L^2(\Omega_{\e})}&\leq \big(C'_0\cdot C_{reg}(\mathcal{U}_0)\big)\cdot\e^{\frac{p-1}{p}}\cdot\big(\big\|\overline{u}_{\e}-c\overline{v}_{\e}\|_{L^2(\mathcal{U}_0)}+(1+C_{tr}(\mathcal{U}_0))\|\overline{v}_{\e}\|_{H^1(\mathcal{U}_0)}\big)\\
&\leq \big(C'_0\cdot C_{reg}(\mathcal{U}_0)\big)\cdot\e^{\frac{p-1}{p}}\cdot\big(\|\overline{u}_{\e}\|_{L^2(\mathcal{U}_0)}+c\cdot\|\overline{v}_{\e}\|_{L^2(\mathcal{U}_0)}+(1+C_{tr}(\mathcal{U}_0))\cdot \|\overline{v}_{\e}\|_{H^1(\mathcal{U}_0)}\big)
\end{align*}
and, using \eqref{eq:ineq_C_polar}, we obtain:
\begin{align*}
\|u_{\e}\|^2_{L^2(\Omega_{\e})}&\leq \big(C'_0\cdot C_{reg}(\mathcal{U}_0)\cdot C_{polar}^{-1}\big)\cdot\e^{\frac{p-1}{p}}\cdot\\
&\cdot \big(\big\|u_{\e}\circ \Phi_\e\big|_{\Omega_0}\big\|_{L^2(\Omega_0)}+c\cdot\big\|v_{\e}\circ \Phi_\e\big|_{\Omega_0}\big\|_{L^2(\Omega_0)}+(1+C_{tr}(\mathcal{U}_0))\cdot \big\|v_{\e}\circ \Phi_\e\big|_{\Omega_0}\big\|_{H^1(\Omega_0)}\big).
\end{align*}

Using \Cref{prop:properties_of_Phi_e}, we get:
\begin{align*}
\|u_{\e}\|^2_{L^2(\Omega_{\e})}&\leq \big(C'_0\cdot C_{reg}(\mathcal{U}_0)\cdot C_{polar}^{-1}\cdot C_\Phi^{-1}\big)\cdot\e^{\frac{p-1}{p}}\cdot\\
&\cdot \big(\|u_{\e}\|_{L^2(\Omega_{\e})}+c\cdot\|v_{\e}\|_{L^2(\Omega_{\e})}+(1+C_{tr}(\mathcal{U}_0))\cdot \|v_{\e}\|_{H^1(\Omega_{\e})}\big).
\end{align*}
and then, using \Cref{prop:properties_of_v_eps}, we obtain:
\begin{align*}
\|u_{\e}\|_{L^2(\Omega_{\e})}&\leq \big(C_0\cdot C_{emb}(\Omega_0)\cdot C_{reg}(\mathcal{U}_0)\big)\cdot\bigg(2+(1+C_{tr}(\mathcal{U}_0))\cdot \text{min}\{c^{-1/2},c^{-1}\}\bigg)\cdot\e^{\frac{p-1}{p}},
\end{align*}
for any $p>2$.
\end{proof}

\section{Concluding remarks}

Within this work we have considered the effect of rugosity in liquid crystalline systems in the limit where the wavelength and amplitude of the surface oscillations go to zero at a comparable rate. 

 We have considered two different aspects of the problem:
 
 \begin{itemize}
 \item firstly a highly general formulation of the problem that is able to simultaneously treat many different common models of liquid crystals, including Oseen-Frank, and Landau-de Gennes with either polynomial or singular bulk potential. The fact that we are able to consider a broad set of models in one framework is due to the fact that their specific peculiarities are more apparent in the ``bulk" of the domain, and have little effect on the surface energy.
 
 \item Secondly we have considered a more simple linear model in a periodic two-dimensional geometry, which we are able to interpret as a toy model for paranematic systems. Within this simpler situation, we are able to make far more precise estimates on the convergence towards the homogenised limit, and near-explicit calculation of the relevant coefficients in the homogenised surface energy.
 \end{itemize}
 
  These two threads both show the same key feature, that the influence of the geometry of the oscillatory surfaces we consider may replaced by effective surface energies on a simpler limiting domain, which may be represented using particular types of weak limits; that of families of Young measures in the general case, and weak-$L^2$ limits in the second.
  
   The gap between the two cases leaves an obvious question, as to whether the more precise estimates for the linear two-dimensional system may be extended to more realistic non-linear models, both in higher dimensions and more complex geometries than that of a periodic cell.
   
    Furthermore, the inexplicit nature of the effective surface energies, described only as weak limits defined on potentially complex geometries, raises the question as to whether a more detailed analysis of such surface oscillations could give rise to more constructive methods for evaluating the effective surface energies, or more ambitiously, towards an inverse problem of designing them. 

\section{Acknowledgments}

This research is supported by the Basque Government through the BERC 2018-2021 program and by the Spanish State Research Agency through BCAM Severo Ochoa excellence accreditation SEV-2017-0718 and SEV-2013-0323-17-1 (BES-2017-080630) and through project PID2020-114189RB-I00 funded by Agencia Estatal de Investigaci\'on (PID2020-114189RB-I00 / AEI / 10.13039/501100011033)

\appendix

\section{Arithmetic on surface terms}
\label{appSurface}
To analyse the surface terms, we will need a quantity that is morally the wedge product of $N-1$ vectors in $\mathbb{R}^N$. We define precisely what we mean by this as follows.
 
\begin{definition}
Let $(v^i)_{i=1}^{N-1}$ be a collection of $N-1$ vectors in $\mathbb{R}^N$. We define, in a canonical orthonormal basis $(e_i)_{i=1}^N$ with $v^i_j=v^i\cdot e_j$, the $N\times N-1$ matrix $M$ to be the matrix with row $i$ equal to $(v_j^i)_{j=1}^N$. That is, 
\begin{equation}
M= \begin{pmatrix}
v^1_1 & v^1_2 & \hdots & v^1_N\\
v^2_1 & v^2_2 & \hdots & v^2_N\\
\vdots & \vdots & \ddots & \vdots\\
v^{N-1}_1 & v^{N-1}_2 & \hdots & v^{N-1}_N.
\end{pmatrix}
\end{equation}
Then for $1<i<N$, we define $M_i$ to be the $N-1\times N-1$ matrix which corresponds to $M$ with the $i$-th column removed. We define the product 
\begin{equation}
\bigwedge\limits_{i=1}^{N-1}v^i = \sum\limits_{i=1}^N \det(M_i)(-1)^i e_i. 
\end{equation}
\end{definition}

\begin{remark}
We present this definition in a self-contained manner, however it is simply a specific case of the exterior product \cite[Chapter 9]{lee2013smooth}. Using the language of differential geometry, we identifying vectors $v_i\in\mathbb{R}^N$ with {\it 1-vectors}, denoted $v_i\in \Lambda^1(\mathbb{R}^N)$. The exterior product of $(N-1)$ 1-vectors gives an $(N-1)$-vector, $\bigwedge\limits_{i=1}^{N-1}v_i =v_1\wedge v_2\wedge....\wedge v_{N-1}\in \Lambda^{N-1}(\mathbb{R}^N)$. The space $\Lambda^{N-1}(\mathbb{R}^N)$ may be canonically identified with the space of 1-vectors $\Lambda^1(\mathbb{R}^N)$ via the Hodge star operator, so that $\mathbb{R}^N\sim\Lambda^1(\mathbb{R}^N)\sim\Lambda^{N-1}(\mathbb{R}^N)$. In this sense, we may interpret $\bigwedge\limits_{i=1}^{N-1}v_i$ for $v_i\in\mathbb{R}^N$ as a vector in $\mathbb{R}^N$ in a canonical way.
\end{remark}

\begin{remark}
We abuse notation for brevity, and interpret the ``determinant" of the $N\times N$ matrix with first column corresponding to the canonical basis of $\mathbb{R}^N$ as 
\begin{equation}\label{eqFunkyDet}
\begin{split}
&\det\begin{pmatrix}
e_1 & e_2 & \hdots & e_N\\
v^1_1 & v^1_2 & \hdots & v^1_N\\
v^2_1 & v^2_2 & \hdots & v^2_N\\
\vdots & \vdots & \ddots & \vdots\\
v^{N-1}_1 & v^{N-1}_2 & \hdots & v^{N-1}_N
\end{pmatrix}\\
=&\det\begin{pmatrix}
v^1_2 & v^1_3 & \hdots & v^1_N\\
v^2_2 & v^2_3 & \hdots & v^2_N\\
\vdots & \vdots & \ddots & \vdots\\
v^{N-1}_2 & v^{N-1}_3 & \hdots & v^{N-1}_N
\end{pmatrix}e_1 - \det\begin{pmatrix}
v^1_1 & v^1_3 & \hdots & v^1_N\\
v^2_1 & v^2_3 & \hdots & v^2_N\\
\vdots & \vdots & \ddots & \vdots\\
v^{N-1}_1 & v^{N-1}_3 & \hdots & v^{N-1}_N 
\end{pmatrix}e_2+\hdots\\
=&  \sum\limits_{i=1}^N \det(M_i)(-1)^i e_i=\bigwedge\limits_{i=1}^{N-1}v^i .
\end{split}
\end{equation}
\end{remark}

\begin{remark}
It is immediate from the expression \eqref{eqFunkyDet} that if $1\leq j \leq N-1$,
\begin{equation}
v^j\cdot \bigwedge\limits_{i=1}^{N-1}v^i = \det\begin{pmatrix}
v^j_1 & v^j_2 & \hdots & v^j_N\\
v^1_1 & v^1_2 & \hdots & v^1_N\\
v^2_1 & v^2_2 & \hdots & v^2_N\\
\vdots & \vdots & \ddots & \vdots\\
v^{N-1}_1 & v^{N-1}_2 & \hdots & v^{N-1}_N
\end{pmatrix}=0,
\end{equation}
thus we are able to write the outer unit normal to $\Gamma_\e$, with respect to a correctly oriented orthonormal basis $(\tau_i)_{i=1}^{N-1}$ of $T_x\Gamma$ as 
\begin{equation}
\begin{split}
\Lambda_\e (x)= & \bigwedge\limits_{i=1}^{N-1}\langle D(x+\varphi_\e(x)\nu(x)),\tau_i\rangle,\\
\tilde{\nu}_\e = & \frac{1}{|\Lambda_\e|}\Lambda_\e.
\end{split}
\end{equation}
Furthermore, the surface element $\gamma_\e$ can be written in terms of $\Lambda_\e$ as 
\begin{equation}
\gamma_\e=|\Lambda_\e|.
\end{equation}
These representations involving the operator $\bigwedge$ thus permit us to simplify the surface energy terms sufficiently to produce more explicit results. 

We also note the further identity, that for any scalars $(c_i)_{i=1}^{N-1}$, vectors $v, (v^i)_{i=1}^{N-1}$, we have that 
\begin{equation}\label{eqRank1Stuff}
\begin{split}
v\cdot\bigwedge\limits_{i=1}^{N-1} (v^i +c_i v)=&\det\begin{pmatrix}
v_1 & v_2 & \hdots & v_N\\
v^1_1+c_1v_1 & v^1_2 +c_1v_2& \hdots & v^1_N+c_1v_N\\
v^2_1 +c_2v_1& v^2_2 +c_2v_2& \hdots & v^2_N+c_2v_N\\
\vdots & \vdots & \ddots & \vdots\\
v^{N-1}_1 +c_{N-1}v_1& v^{N-1}_2 +c_{N-1}v_2& \hdots & v^{N-1}_N+c_{N-1}v_N
\end{pmatrix}\\
=&\det\begin{pmatrix}
v_1 & v_2 & \hdots & v_N\\
v^1_1 & v^1_2 & \hdots & v^1_N\\
v^2_1 & v^2_2 & \hdots & v^2_N\\
\vdots & \vdots & \ddots & \vdots\\
v^{N-1}_1 & v^{N-1}_2 & \hdots & v^{N-1}_N
\end{pmatrix}=v\cdot\bigwedge\limits_{i=1}^{N-1}v^i.
\end{split}
\end{equation}
\end{remark}

\begin{proposition}\label{eqPropLeadingOrderSurface}
Let $\varphi_\e$ and domains $\Omega_\e,\Omega$ satisfy Assumptions \ref{assumpOmega}, \ref{assumpOmegaEps} and  $\hat{D}\varphi_\e$ as given in \Cref{assumpSurface2}. Then, for $\e$ sufficiently small, we may estimate $\gamma_\e(x)$, $\nu_\e(x)$ in terms of $\hat{D}\varphi_\e$ up to a quantity of order $\e$ as
\begin{equation}
\begin{split}
\gamma_\e=&|\hat{D}\varphi_\e|+O(\e),\\
\tilde{\nu}_\e=&=\frac{1}{|\hat{D}\varphi_\e|}\hat{D}\varphi_\e+O(\e)
\end{split}
\end{equation}
where $h(x)=O(\e)$ is taken to mean that $|h(x)|\leq C\e$ for some $C$ which is independent of $x,\e$.
\end{proposition}
\begin{proof}
First, we note that the map $\Gamma\ni x\mapsto x+\varphi_\e(x)\nu(x)\in\Gamma_\e$ is differentiable, and its derivative in the direction $\tau_i\in T_x\Gamma$ satisfies
\begin{equation}
\left\langle D\left(x+\varphi_\e(x)\nu(x)\right),\tau_i\right\rangle=\tau_i+\left\langle D\varphi_\e(x),\tau_i\right\rangle\nu(x)+\varphi_\e(x)\left\langle D\nu(x),\tau_i\right\rangle.
\end{equation}
Now, as $\Gamma$ is a $C^2$ compact manifold, this implies that $\nu$ is $C^1$ with bounded $C^1$ norm, and since $|\tau_i|=1$, we have that $|\varphi_\e(x)\langle D\nu(x),\tau_i\rangle|\leq ||\varphi_\e||_\infty ||D\nu||_\infty=O(\e)$. 
Now we recall that the quantity $\Lambda_\e$ defined by
\begin{equation}
\Lambda_\e(x)=\bigwedge\limits_{i=1}^{N-1}\left\langle D\left(x+\varphi_\e(x)\nu(x)\right),\tau_i\right\rangle
\end{equation}
provides expressions for $\tilde{\nu}_\e,\gamma_\e$, and turn our attention to $\Lambda_\e$. We may then write, as all quantities are uniformly bounded in $x,\e$, and the wedge product is locally Lipschitz continuous, that 
\begin{equation}
\Lambda_\e(x)=O(\e)+\bigwedge\limits_{i=1}^{N-1}\Big(\tau_i+\left\langle D\varphi_\e(x),\tau_i\right\rangle\nu(x)\Big).
\end{equation}
We note that the $O(\e)$ term contains all contributions related to the curvature of $\Gamma$, i.e., $D\nu(x)[\tau_i]$.
We may evaluate $\Lambda_\e$ componentwise, as $(\tau_i)_{i=1}^{N-1}$ with $\nu$ form an orthonormal basis for $\mathbb{R}^N$, and we see that 
\begin{equation}\begin{split}
\Lambda_\e(x)\cdot\nu=&\nu\cdot\bigwedge\limits_{i=1}^{N-1}\left(\tau_i+\left\langle D\varphi_\e(x),\tau_i\right\rangle \nu(x)\right)+O(\e)\\
=& \nu\cdot\bigwedge\limits_{i=1}^{N-1}\tau_i+O(\e)=1+O(\e),\end{split}
\end{equation}
via the identity \eqref{eqRank1Stuff}, and that $\bigwedge\limits_{i=1}^{N-1}\tau_i=\nu$. Then we may evaluate
\begin{equation}
\begin{split}
\Lambda_\e(x)\cdot\tau_i=&\Lambda_\e(x)\cdot\left(\tau_i+\left\langle D\varphi_\e(x),\tau_i\right\rangle\nu(x)\right)-\Lambda_\e(x)\cdot \left(\langle D\varphi_\e(x),\tau_i\rangle \nu(x)\right)\\
=& - \langle D\varphi_\e(x),\tau_i\rangle\Lambda_\e(x)\cdot\nu(x)=- \langle D\varphi_\e(x),\tau_i\rangle +O(\e).
\end{split}
\end{equation}
In particular, as $\nu$ and $(\tau_i)_{i=1}^{N-1}$ form an orthonormal basis, we must have that 
\begin{equation}
\begin{split}
\Lambda_\e(x)=&\left(\Lambda_\e(x)\cdot\nu(x)\right)\nu(x)+\sum\limits_{i=1}^{N-1}\left(\Lambda_\e(x)\cdot\tau_i(x)\right)\tau_i(x)\\
=&\nu(x)-\sum\limits_{i=1}^{N-1}\langle D\varphi_\e(x),\tau_i\rangle +O(\e)\\
=& \hat{D}\varphi_\e(x)+O(\e).
\end{split}
\end{equation}

Then, using that $|\hat{D}\varphi_\e|\geq 1$ and is thus certainly bounded away from zero, it is immediate that 
\begin{equation}
\begin{split}
\gamma_\e =& |\Lambda_\e | = |\hat{D}\varphi_\e|+O(\e)\\
\nu_\e =& \frac{1}{|\Lambda_\e|}\Lambda_\e = \frac{1}{|\hat{D}\varphi_\e|}\hat{D}\varphi_\e+O(\e)
\end{split}
\end{equation}
provided $\e$ is sufficiently small.
\end{proof}

With this, it is straightforward that \Cref{corollarySimplerSurface} holds.
\begin{proof}[Proof of \Cref{corollarySimplerSurface}]
As $w(x,\nu,u)$ is Lipschitz with Lipschitz constant bounded by $C(1+|u|^q)$ and $|D\varphi_\e|$ is uniformly bounded, the result is immediate as $u\in L^q(\Gamma)$.
\end{proof}

\begin{remark}
The consequence of this corollary is that the effective surface energy $w_h$ can be described as a limit of only first order quantities, that is, derivatives of $\varphi_\e$ and the unit normal vector $\nu$. In particular, curvature of the domain does not have an effect, and in the proof of \Cref{eqPropLeadingOrderSurface}, we see that this is because the influence of curvature via $\langle D\nu(x),\tau_i\rangle$ in $\gamma_\e,\nu_\e$ is multiplied by $\varphi_\e$, an order $\e$ quantity. In particular this suggests that ``higher order" behaviour as $\e\to 0$ could be affected by curvature of the limiting domain. 
\end{remark}

\section{Proofs for error estimate}\label{appendix:err_est}

In this subsection, we prove \Cref{cor:Q_eps}, that $a_{\e}(Q_{\e}-Q_0,v)$ is well-defined for any $v\in W^{1,p}(\Omega_0)$ and then we prove \Cref{prop:sum_of_integrals}.

\begin{proof}[Proof of \Cref{cor:Q_eps}]
We consider $\tilde{Q}_{\e}:\mathcal{U}_{\e}\to\text{Sym}_0(2)$ such that:
\begin{align*}
Q_{\e}(x,y)=\tilde{Q}_{\e}\big(\Phi(x,y)\big),\;\forall\;(x,y)\in\Omega_{\e},
\end{align*}
and we denote $(\tilde{x},\tilde{y})=\Phi(x,y)$ and $(x,y)=\big(\Phi^{-1}_1(\tilde{x},\tilde{y}),\Phi^{-1}_2(\tilde{x},\tilde{y})\big)$.

Then $\tilde{Q}_{\e}$ solves a PDE of the following form:
\begin{align}\label{eq:PDE_annulus}
\begin{cases}
\displaystyle{\sum_{i,j=1}^2D_i\big(a_{ij}D_j\tilde{Q}_{\e}\big)+\sum_{i=1}^2a_iD_i\tilde{Q}_{\e}+ c\tilde{Q}_{\e}=0,\;\text{in}\;\mathcal{U}_{\e};}\\
\\
\displaystyle{\sum_{i=1}^2b_{1i}D_i\tilde{Q}_{\e}+\dfrac{w_0}{2}\tilde{Q}_{\e}=\dfrac{w_0}{2}\tilde{Q}_{\e}^0,\;\text{on}\;\Phi(\Gamma_{\e});}\\
\\
\displaystyle{\sum_{i=1}^2b_{2i}D_i\tilde{Q}_{\e}+\dfrac{w_0}{2}\tilde{Q}_{\e}=\dfrac{w_0}{2}\tilde{Q}_R,\;\text{on}\;\Phi(\Gamma_R).}
\end{cases}
\end{align}
where $D_1=\dfrac{\partial}{\partial\tilde{x}}$ and $D_2=\dfrac{\partial}{\partial\tilde{y}}$. We have that $a_{ij}\in C^{\infty}(\mathcal{U}_{\e})$ with: 
\begin{align*}
\begin{pmatrix}
a_{11} & a_{12}\\
a_{21} & a_{22}
\end{pmatrix}=\begin{pmatrix}
-\tilde{y}^2-\dfrac{\tilde{x}^2}{\tilde{x}^2+\tilde{y}^2} & \tilde{x}\tilde{y}-\dfrac{\tilde{x}\tilde{y}}{\tilde{x}^2+\tilde{y}^2}\\
\tilde{x}\tilde{y}-\dfrac{\tilde{x}\tilde{y}}{\tilde{x}^2+\tilde{y}^2} & -\tilde{x}^2-\dfrac{\tilde{y}^2}{\tilde{x}^2+\tilde{y}^2}
\end{pmatrix}.
\end{align*}
The coefficients $a_i$ are also from $C^{\infty}(\mathcal{U}_{\e})$ with: 
\begin{align*}
\begin{pmatrix}
a_1\\
a_2
\end{pmatrix}=\begin{pmatrix}
\tilde{x}+D_1a_{11}+D_2a_{21}\\
\tilde{y}+D_1a_{12}+D_2a_{22}
\end{pmatrix}
\end{align*}
The coefficients $b_{1i}$ are from $C^1(\Phi(\Gamma_{\e}))$ and can be obtained explicitly from $\nabla Q_{\e}\cdot \nu_\e=\sum_{i=1}^2b_{1i}D_i\tilde{Q}_\e$. In the same way, $b_{2i}$ are from $C^1(\Phi(\Gamma_R))$ and can be obtained explicitly from $\nabla Q_\e\cdot(0,1)=\sum_{i=1}^2b_{2i}D_i\tilde{Q}_\e$. Moreover, $\tilde{Q}_{\e}^0=Q_{\e}^0\circ \Phi^{-1}$ and $\tilde{Q}_R=Q_R\circ \Phi^{-1}$. Since $Q_{\e}^0\in C^1(\Gamma_{\e})$ and $\Phi$ is a smooth bi-Lipschitz transformation, then $\tilde{Q}_{\e}^0\in C^1(\Phi(\Gamma_{\e}))$ which implies that $\tilde{Q}_{\e}^0\in W^{1,1-1/p}(\Phi(\Gamma_{\e}))$. In the same way, $\tilde{Q}_R\in W^{1,1-1/p}(\Phi(\Gamma_R))$.

Therefore, we can apply \Cref{thm:reg} to obtain that there exists a unique solution\linebreak $\tilde{Q}_{\e}^0\in W^{2,p}(\mathcal{U}_\e)$ of the problem \eqref{eq:PDE_annulus}. Using now that $\Phi$ is smooth, we obtain that there exists a unique solution $Q_{\e}\in W^{2,p}(\Omega_{\e})$ of the problem \eqref{defn:PDE_Q_eps}.
\end{proof}

Let $p\in(1,+\infty)$ and $v\in W^{1,p}(\Omega_0)$. We recall that $\Omega_{\e}\subset\Omega_0$, for all $\e>0$, hence we also have that $v\big|_{\Omega_{\e}}\in W^{1,p}(\Omega_{\e})$. Applying \Cref{cor:Q_eps} and \Cref{cor:Q_0} with exponent $p'=\dfrac{p-1}{p}\in (1,+\infty)$, we obtain that $Q_{\e}\in W^{1,p'}(\Omega_{\e})$ and that $Q_0\in W^{1,p'}(\Omega_0)$. Using \Cref{defn:bilinear_funct}, it is easy to see that $a_{\e}(Q_{\e}-Q_0,v)$ is well-defined.

\begin{proof}[Proof of \Cref{prop:sum_of_integrals}]

In the following paragraphs, we fix $p\in (1,+\infty)$ and $v\in W^{1,p}(\Omega_0)$. We recall now that $Q_{\e}$ solves weakly \eqref{defn:PDE_Q_eps}, which is the following PDE:
\begin{align*}
\begin{cases}
-\Delta Q_{\e}+cQ_{\e}=0\;\;\;\text{in}\;\Omega_{\e};\\
\\
\dfrac{\partial Q_{\e}}{\partial\nu_{\e}}+\dfrac{w_0}{2}Q_{\e}=\dfrac{w_0}{2}Q^0_{\e}\;\;\;\text{on}\;\Gamma_{\e};\\
\\
\dfrac{\partial Q_{\e}}{\partial\nu_R}+\dfrac{w_0}{2}Q_{\e}=\dfrac{w_0}{2}Q_R\;\;\;\text{on}\;\Gamma_R.
\end{cases}
\end{align*}

Then, using the integration by parts formula, we get:
\begin{align*}
0&=\int_{\Omega_{\e}} \big(-\Delta Q_{\e}+cQ_{\e}\big)\cdot v\;\text{d}(x,y)\\
&=\int_{\Omega_{\e}}\big(\nabla Q_{\e}\cdot\nabla v +cQ_{\e}\cdot v\big)\;\text{d}(x,y)-\int_{\Gamma_{\e}}\dfrac{\partial Q_{\e}}{\partial \nu_{\e}}\;\text{d}\sigma_{\e}-\int_{\Gamma_R}\dfrac{\partial Q_{\e}}{\partial \nu_R}\;\text{d}\sigma_R\\
&=\int_{\Omega_{\e}}\big(\nabla Q_{\e}\cdot\nabla v+cQ_{\e}\cdot v\big)\;\text{d}(x,y)+\dfrac{w_0}{2}\int_{\Gamma_{\e}}\big(Q_{\e}-Q_{\e}^0\big)\cdot v\;\text{d}\sigma_{\e}+\dfrac{w_0}{2}\int_{\Gamma_R}\big(Q_{\e}-Q_R\big)\cdot v\;\text{d}\sigma_R\\
&=a_{\e}(Q_{\e},v)-\dfrac{w_0}{2}\int_{\Gamma_{\e}}Q_{\e}^0\cdot v\;\text{d}\sigma_{\e}-\dfrac{w_0}{2}\int_{\Gamma_R}Q_R\cdot v\;\text{d}\sigma_R,
\end{align*}
according to \Cref{defn:bilinear_funct}. Hence
\begin{align}\label{eq:app_a_e_Q_e}
a_{\e}(Q_{\e},v)=\dfrac{w_0}{2}\int_{\Gamma_{\e}}Q_{\e}^0\cdot v\;\text{d}\sigma_{\e}+\dfrac{w_0}{2}\int_{\Gamma_R}Q_R\cdot v\;\text{d}\sigma_R.
\end{align}

For $Q_0$, we remind the reader that it solves weakly \ref{defn:PDE_Q_0}, which is the following PDE:
\begin{align*}
\begin{cases}
-\Delta Q_0+cQ_0=0\;\;\;\text{in}\;\Omega_0;\\
\\
\dfrac{\partial Q_0}{\partial\nu_0}+\dfrac{w_{ef}}{2}Q_0=\dfrac{w_{ef}}{2}Q_{ef}\;\;\;\text{on}\;\Gamma_0;\\
\\
\dfrac{\partial Q_0}{\partial\nu_R}+\dfrac{w_0}{2}Q_0=\dfrac{w_0}{2}Q_R\;\;\;\text{on}\;\Gamma_R.
\end{cases}
\end{align*}

Then, using the previous PDE and the integration by parts formula, we get:
\begin{align*}
0&=\int_{\Omega_0}\big(-\Delta Q_0+cQ_0)\cdot v\;\text{d}(x,y)\\
&=\int_{\Omega_0}\big(\nabla Q_0\cdot \nabla v+cQ_0\cdot v\big)\;\text{d}(x,y)-\int_{\Gamma_0}\dfrac{\partial Q_0}{\partial \nu_0}\cdot v\;\text{d}\sigma_0-\int_{\Gamma_R}\dfrac{\partial Q_0}{\partial \nu_R}\cdot v\;\text{d}\sigma_R\\
&=\int_{\Omega_0}\big(\nabla Q_0\cdot \nabla v+cQ_0\cdot v\big)\;\text{d}(x,y)+\dfrac{w_{ef}}{2}\int_{\Gamma_0}\big(Q_0-Q_{ef}\big)\cdot v\;\text{d}\sigma_0+\\
&\hspace{5mm}+\dfrac{w_0}{2}\int_{\Gamma_R}\big(Q_0-Q_R\big)\cdot v\;\text{d}\sigma_R\\
&=a_{\e}(Q_0,v)+\int_{\Omega_0\setminus\Omega_{\e}}\big(\nabla Q_0\cdot \nabla v+cQ_0\cdot v\big)\;\text{d}(x,y)+\dfrac{w_{ef}}{2}\int_{\Gamma_0}\big(Q_0-Q_{ef}\big)\cdot v\;\text{d}\sigma_0-\\
&\hspace{5mm}-\dfrac{w_0}{2}\int_{\Gamma_R}Q_R\cdot v\;\text{d}\sigma_R-\dfrac{w_0}{2}\int_{\Gamma_{\e}}Q_0\cdot v\;\text{d}\sigma_{\e},
\end{align*}
where we have used \Cref{defn:bilinear_funct}. Using \eqref{eq:app_a_e_Q_e}, we obtain that:
\begin{align*}
a_{\e}(Q_0,v)-a_{\e}(Q_{\e},v)&=-\int_{\Omega_0\setminus\Omega_{\e}}\big(\nabla Q_0\cdot \nabla v+cQ_0\cdot v\big)\;\text{d}(x,y)-\dfrac{w_0}{2}\int_{\Gamma_{\e}}Q_{\e}^0\cdot v\;\text{d}\sigma_{\e}-\\
&\hspace{3mm}-\dfrac{w_{ef}}{2}\int_{\Gamma_0}\big(Q_0-Q_{ef}\big)\cdot tv\;\text{d}\sigma_0+\dfrac{w_0}{2}\int_{\Gamma_{\e}}Q_0\cdot v\;\text{d}\sigma_{\e}
\end{align*}

Using \Cref{defn:I_integrals}, we already notice that:
\begin{align}\label{eq:desired_equality_only_I_1}
a_{\e}(Q_0-Q_{\e},v)=I_1-\dfrac{w_0}{2}\int_{\Gamma_{\e}}Q_{\e}^0\cdot v\;\text{d}\sigma_{\e}-\dfrac{w_{ef}}{2}\int_{\Gamma_0}\big(Q_0-Q_{ef}\big)\cdot v\;\text{d}\sigma_0+\dfrac{w_0}{2}\int_{\Gamma_{\e}}Q_0\cdot v\;\text{d}\sigma_{\e}.
\end{align}

Since $\Gamma_{\e}=\{\big(x,\e\varphi(x/\e)\big)\;|\;x\in[0,2\pi)\}$, $\Gamma_0=\{(x,0)\;|\;x\in[0,2\pi)\}$ and $\Gamma_R=\{(x,R)\;|\;x\in[0,2\pi)\}$, we obtain:
\begin{align*}
&\dfrac{w_0}{2}\int_{\Gamma_{\e}}Q_{\e}^0\cdot v\;\text{d}\sigma_{\e}=\dfrac{w_0}{2}\int_0^{2\pi}Q_{\e}^0(x/\e)\cdot v(x,\e\varphi(x/\e))\cdot \gamma_{\e}(x/\e)\;\text{d}x,\\
&\dfrac{w_{ef}}{2}\int_{\Gamma_0}\big(Q_0-Q_{ef}\big)\cdot v\;\text{d}\sigma_0=\dfrac{w_0}{2}\int_0^{2\pi}\big(Q_0(x,0)\cdot\gamma-Q_{ef}\cdot \gamma\big)\cdot v(x,0)\;\text{d}x\\
\text{and}\\
&\dfrac{w_0}{2}\int_{\Gamma_{\e}}Q_0\cdot v\;\text{d}\sigma_{\e}=\dfrac{w_0}{2}\int_0^{2\pi}Q_0(x,\e\varphi(x/\e))\cdot v(x,\e\varphi(x/\e))\cdot \gamma_{\e}(x/\e)\;\text{d}x,
\end{align*}
where we have used \eqref{eq:gamma_e}, \eqref{eq:nu_eps_depends_only_on_x_e} and \Cref{remark:Q_e_0_dependency}.

Using the previous equalities in \eqref{eq:desired_equality_only_I_1}, we obtain that:
\begin{align*}
a_{\e}(Q_0-Q_{\e},v)&=I_1-\dfrac{w_0}{2}\int_0^{2\pi}\big(v(x,\e\varphi(x/\e))-v(x,0)\big)\cdot\gamma_{\e}(x/\e)\cdot Q_{\e}^0(x/\e)\;\text{d}x-\\
&\hspace{-10mm}-\dfrac{w_0}{2}\int_0^{2\pi}v(x,0)\cdot\gamma_{\e}(x/\e)\cdot Q_{\e}^0(x/\e)\;\text{d}x+\dfrac{w_0}{2}\int_0^{2\pi}v(x,0)\cdot \gamma\cdot Q_{ef}\;\text{d}x-\\
&\hspace{-10mm}-\dfrac{w_0}{2}\int_0^{2\pi}Q_0(x,0)\cdot v(x,0)\cdot \gamma\;\text{d}x+\dfrac{w_0}{2}\int_0^{2\pi}Q_0(x,\e\varphi(x/\e))\cdot v(x,\e\varphi(x/\e))\cdot\gamma_{\e}(x/\e)\;\text{d}x.
\end{align*}

We see now that the second integral from the right-hand side from the last equality is $I_{21}$, according to \Cref{defn:I_integrals}, and that the next two terms generate $I_{22}$, according to the same definition as before. Hence:
\begin{align*}
a_{\e}(Q_0-Q_{\e},v)&=I_1+I_{21}+I_{22}+\\
&\hspace{-10mm}+\dfrac{w_0}{2}\int_0^{2\pi}\big(Q_0(x,\e\varphi(x/\e))\cdot v(x,\e\varphi(x/\e))-Q_0(x,0)\cdot v(x,0)\big)\cdot\gamma_{\e}(x/\e)\;\text{d}x+\\
&\hspace{-10mm}+\dfrac{w_0}{2}\int_0^{2\pi}Q_0(x,0)\cdot v(x,0)\cdot \gamma_{\e}(x/\e)\;\text{d}x-\dfrac{w_0}{2}\int_0^{2\pi}Q_0(x,0)\cdot v(x,0)\cdot \gamma\;\text{d}x.
\end{align*}

The last equality proves \eqref{eq:sum_of_integrals}.

\end{proof}

\medskip

\bibliography{bibl}

\end{document}